\documentclass[10pt,reqno,a4paper]{amsart}
\title{$P$-adic $L$-functions for $\mathrm{GL}(3)$}
\author{David Loeffler and Chris Williams}
\thanks{Supported by: Royal  Society  University  Research  Fellowship  ``L-functions  and  Iwasawa  theory'', EPSRC Standard Grant EP/S020977/1, and ERC Horizon 2020 Consolidator Grant \#101001051 ``Shimura varieties and the BSD conjecture'' (DL); EPSRC Postdoctoral Fellowship EP/T001615/1 (CW)}

\newcommand{\stbt}[4]{\left(\begin{smallmatrix}#1 & #2\\#3&#4\end{smallmatrix}\right)}
\newcommand{\sdthree}[3]{\left(\begin{smallmatrix}#1 \\ & #2\\ && #3\end{smallmatrix}\right)}

\usepackage{accents}

\newcommand{\Poref}{\tilde\Pi}

\newcommand{\sar}[2]{\ar@{}[#1]|-*[@]{#2}}

\newcommand{\Pif}{\Pi\subf}
\newcommand{\Phif}{{\Phi\subf}}
\newcommand{\subft}{_{\mathrm{f},t}}
\newcommand{\Phift}{\Phi\subft}

\usepackage[margin=1.2in]{geometry}

\usepackage[normalem]{ulem}
\usepackage{preamble}
\usepackage{mathtools, tikz-cd, mathdots, enumerate} % for \defeq

\newcommand{\LSk}{\sV_{(0,-j)}^{\GL_2}}
\newcommand{\LSkn}{\sV_{(0,-j)}}
\newcommand{\chihat}{\widehat{\chi}}
\newcommand{\omegahat}{\widehat{\omega}}
\newcommand{\etahat}{\widehat{\eta}}
\newcommand{\subf}{_{\mathrm{f}}}

\usepackage{multicol}
\usepackage{tikz}
\usetikzlibrary{positioning}

\newcommand{\Addresses}{{% additional braces for segregating \footnotesize
  \bigskip
  \footnotesize

 David Loeffler: University of Warwick, Mathematics Institute, Zeeman Building, Coventry CV4 7AL, United Kingdom $ \ \cdot \ $
 \emph{current address}: UniDistance Suisse, Schinerstrasse 18, 3900 Brig, Switzerland
  $ \ \cdot \ $ \\\texttt{david.loeffler@unidistance.ch} $ \ \cdot \ $ \url{http://orcid.org/0000-0001-9069-1877}

  \medskip
 Chris Williams: University of Nottingham, School of Mathematical Sciences, Nottingham NG7 2RD, United Kingdom $ \ \cdot \ $ \texttt{chris.williams1@nottingham.ac.uk} $ \ \cdot \ $  \url{http://orcid.org/0000-0001-8545-0286}

}}

\newcommand{\SL}{\mathrm{SL}}

\DeclareMathOperator{\Eis}{Eis}

\usepackage{etoolbox}
\newtoggle{appendix}
\togglefalse{appendix}

\usepackage{comment}
\usepackage{verbatim}
\iftoggle{appendix}{
\newenvironment{longversion}{}{} % use this to show
\newenvironment{shortversion}{\comment}{\endcomment} % use this to hide

}{

}
\newcommand{\GL}{\mathrm{GL}}
\newcommand{\Zhat}{\widehat{\Z}}

\newcommand{\smallthreemat}[9]{\left(\begin{smallmatrix}
#1 & #2 & #3 \\
#4 & #5 & #6 \\
#7 & #8 & #9\end{smallmatrix}\right)}

\usepackage{tocloft}

\usepackage{titlesec}
\titleformat{\subsubsection}[runin]% runin puts it in the same paragraph
{\normalfont\itshape}% formatting commands to apply to the whole heading
{\thesubsubsection.}% the label and number
{0.5em}% space between label/number and subsection title
{}% formatting commands applied just to subsection title
[.\hspace*{6pt}]% punctuation or other commands following subsection title
\titlespacing*{\subsubsection}{0pt}{5pt}{0pt}

\titleformat{\subsection}[runin]
{\normalfont\bfseries}% formatting commands to apply to the whole heading
{\thesubsection.}% the label and number
{0.5em}% space between label/number and subsection title
{}% formatting commands applied just to subsection title
[.\hspace*{6pt}]% punctuation or other commands following subsection title

\titleformat{\section}
{\large\bfseries}% formatting commands to apply to the whole heading
{\thesection.}% the label and number
{0.5em}% space between label/number and subsection title
{}% formatting commands applied just to subsection title
[]% punctuation or other commands following subsection title

\begin{document}

\begin{abstract}
	Let $\Pi$ be a regular algebraic cuspidal automorphic representation (RACAR) of $\mathrm{GL}_3(\mathbb{A}_{\mathbb{Q}})$. When $\Pi$ is $p$-nearly-ordinary for the maximal standard parabolic with Levi $\mathrm{GL}_1 \times \mathrm{GL}_2$, we construct a $p$-adic $L$-function for $\Pi$. More precisely, we construct a (single) bounded measure $L_p(\Pi)$ on $\Z_p^\times$ attached to $\Pi$, and show it interpolates all the critical values $L(\Pi\times\eta,-j)$ at $p$ in the left-half of the critical strip for $\Pi$ (for varying $\eta$ and $j$). This proves conjectures of Coates--Perrin-Riou and Panchishkin in this case. We also prove a corresponding result in the right half of the critical strip, assuming near-ordinarity for the other maximal standard parabolic.

 Our construction uses the theory of spherical varieties to build a ``Betti Euler system'', a norm-compatible system of classes in the Betti cohomology of a locally symmetric space for $\mathrm{GL}_3$. We work in arbitrary cohomological weight, allow arbitrary ramification at $p$ along the Levi factor of the standard parabolic, and make no self-duality assumption.  We thus give the first constructions of $p$-adic $L$-functions for RACARs of $\mathrm{GL}_n(\A_{\Q})$ of `general type' (i.e.\ those that do not arise as functorial lifts) for any $n >2$.
\end{abstract}

\maketitle

%\tableofcontents
%\footnotesize \tableofcontents \normalsize

%%%%%%%%%%%%%%%%%%%%%%%%%%%%%%%%%%%%%%%%%%%%%%%%%%%%%%%%%%%%%%%%%%%%
%
%                                ABSTRACT
%
%%%%%%%%%%%%%%%%%%%%%%%%%%%%%%%%%%%%%%%%%%%%%%%%%%%%%%%%%%%%%%%%%%%%

%%%%%%%%%%%%%%%%%%%%%%%%%%%%%%%%%%%%%%%%%%%%%%%%%%%%%%%%%%%%%%%%%%%%
%
%                                INTRO
%
%%%%%%%%%%%%%%%%%%%%%%%%%%%%%%%%%%%%%%%%%%%%%%%%%%%%%%%%%%%%%%%%%%%%

\section{Introduction}

\subsection{Context}

There is a wide network of conjectures, including the Birch--Swinnerton-Dyer and Bloch--Kato conjectures, that describe important arithmetic invariants in terms of special values of (complex) $L$-functions. One of the most successful tools in tackling these conjectures has been Iwasawa theory, in which one seeks to formulate analogous $p$-adic conjectures, replacing the complex analytic $L$-function with a \emph{$p$-adic $L$-function}. The resulting $p$-adic Iwasawa main conjectures are often more tractable than their complex counterparts, and provide beautiful, deep connections between analysis and arithmetic.

A crucial launching point for Iwasawa theory is proving \emph{existence} of $p$-adic $L$-functions. Coates and Perrin-Riou conjectured the existence of a $p$-adic $L$-function attached to every motive $M$ over $\Q$ that is ordinary with good reduction at $p$, and whose $L$-function has at least one critical value \cite{coatesperrinriou89,coates89}. This itself rests on Deligne's period conjecture \cite{Del79}. One may formulate automorphic realisations of these (motivic) conjectures, but proving these remains very difficult.

To illustrate this, consider the fundamental case where $\Pi$ is a regular algebraic cuspidal automorphic representation (RACAR) of $\GL_n(\A_{\Q})$. Then Coates and Perrin-Riou predict that when $\Pi$ is unramified and ordinary at $p$, there exists a $p$-adic measure on $\Zp^\times$ -- the $p$-adic $L$-function of $\Pi$ -- interpolating all its critical $L$-values at $p$; we recall the precise conjecture below.

For $n=1,2$, existence of $p$-adic $L$-functions has been known for decades (starting from e.g.\ \cite{KL64,MSD74}). For $n \geq 3$, however, our understanding of the conjecture remains poor, and it is known only in very special cases, e.g.\ symmetric squares of classical modular forms \cite{Sch88,Hid90,DD97}, RACARs of $\GL_{2n}$ with Shalika models \cite{AG94,Geh18, DJR18}, or Rankin--Selberg transfers from $\GL_n \times \GL_{n+1}$ to $\GL_{n(n+1)}$ \cite{Sch93,KMS00,Jan-Non-Abelian}. In particular, in \emph{all} known cases, $\Pi$ arises as a functorial lift from a group whose $L$-group is a proper subgroup of $\GL_n$: beyond the classical cases  of $n=1, 2$, there are no known constructions applying to automorphic representations $\Pi$ of `general type', i.e.\ not arising from a smaller group in this way.

In this paper, we prove existence of $p$-adic $L$-functions for $p$-ordinary RACARs of $\GL_3(\A_{\Q})$, making no self-duality or functorial lift assumptions. 	This provides the first examples of $p$-adic $L$-functions for general type RACARs of $\GL_n(\A_{\Q})$ for any $n > 2$.

We briefly highlight some further strengths of our construction:
\begin{itemize}\setlength{\itemsep}{0pt}
	\item[--] We work in arbitrary cohomological weight, and prove the so-called `Manin relations'; that is, for each critical region we construct a \emph{single} $p$-adic measure that sees $L$-values at all critical integers (rather than a separate measure for each different critical integer).

	\item[--] We actually prove a more general conjecture of Panchishkin \cite{Pan94} that refines Coates--Perrin-Riou: rather than Coates--Perrin-Riou's assumption of Borel-ordinarity at $p$, we impose $p$-ordinarity only along either the maximal standard parabolic $P_1$ with Levi $\GL_1 \times \GL_2$, or $P_2$ with Levi $\GL_2\times \GL_1$. Additionally, we give constructions (in both cases) assuming only near-ordinarity, hence extending \cite{Pan94} beyond the ordinary case. 

	\item[--] We allow arbitrary ramification at $p$. As a result, our construction applies even to a class of RACARs that have infinite slope at $p$ for the Borel subgroup (those that are nearly ordinary for $P_1$, but have infinite slope along the other maximal parabolic subgroup $P_2$; or vice versa).
\end{itemize}

We believe that our result is ``optimal'', in the sense that $p$-adic $L$-functions for RACARs of $\GL_3$ should only exist as bounded measures when these conditions hold. (Our methods can be adapted to construct finite-order tempered distributions interpolating $L$-values of RACARs which have positive, but sufficiently small, slope for $P_i$; this will be pursued elsewhere.)

\subsection{Our results}
To motivate the precise form of our main result, we first state Panchishkin's refinement of the Coates--Perrin-Riou conjecture for $\GL_n(\A_{\Q})$.

Let $n = 2m+1$ be odd\footnote{The case $n=2m$ is similar, but in many ways simpler: only the $(m, m)$ parabolic subgroup is relevant; and $L(\Pi \times \eta, j)$ is critical for all $j$ in a certain interval, independently of the parity of $\eta$, so we do not need to split the critical interval into the two halves $\Crit^{\pm}$.}. Let $\Pi$ be a (unitary\footnote{Any RACAR of $\GL_3(\A)$ is unitary up to some twist by an integral power of the norm character, so this assumption is harmless. This contrasts with the situation for $\GL_2(\A)$, where the analogous statement fails.}) RACAR of $\GL_{n}(\A_{\Q})$ with central character $\omega_\Pi$ and weight $\lambda = (\lambda_1,...,\lambda_{2m+1})$. Note $\lambda_i \geq \lambda_{i+1}$, $\lambda_m = 0$, and $\lambda_i = -\lambda_{2m+1-i}$.
 With these notations, for $\eta$ a Dirichlet character and $t \in \Z$, the $L$-value $L(\Pi \times \eta, t)$ is critical for $(t, \eta)$ lying in one of the following sets:
\begin{align*}
	\Crit^-(\Pi) &= \{(-j,\eta) \ \ \ : 0 \leq j \leq \lambda_{m-1}, \ \ \  \eta\omega_{\Pi}(-1) = (-1)^{j}\},\\
	\Crit^+(\Pi) &=  \{(j+1,\eta) : 0\leq j \leq \lambda_{m-1},\ \ \  \eta\omega_{\Pi}(-1) = (-1)^{j}\},
\end{align*}
See \cref{prop:critical}. These are the critical values of twists of $L(\Pi,s)$ in the left and right halves of the critical strip respectively.   We write $\Crit(\Pi)$ for the union of these sets. For $(t, \eta) \in \Crit(\Pi)$, let  $e_\infty(\Pi_\infty \times \eta_\infty, t)$ be the modified Euler factor at $\infty$ of \cite[\S1]{coates89}, which is a product of a rational number with powers of $i$ and $\pi$ (see \cref{sec:rationality}).

\subsubsection*{Algebraicity} As a precursor to $p$-adic interpolation, we first consider an algebraicity result for $L$-values. Let $E$ denote the rationality field of $\Pi$. If $\eta$ is a Dirichlet character, let $E[\eta]$ denote the extension of $E$ obtained by the values of $\eta$. The following is \cite[Period Conjecture]{coates89}, a reformulation of the conjectures of \cite{Del79} better suited to $p$-adic interpolation:

\begin{conjectureintro}[Algebraicity Conjecture for $\GL_{2m+1}$]
	\label{conj:intro-alg}
	There exist complex periods $\Omega_\Pi^\pm \in \C^\times$ such that for all $(t,\eta) \in \Crit^\pm(\Pi)$, we have
	\begin{equation}
		\label{eq:deligne}
		e_\infty\left(\Pi_\infty \times \eta_\infty, t\right)\cdot \frac{L(\Pi\times\eta, t)}{\Omega_\Pi^\pm} \cdot G(\eta^{-1})^{(n \pm 1)/2} \in E[\eta],
	\end{equation}
	where $G(\eta)$ is the Gauss sum. Moreover, this ratio depends $\operatorname{Gal}(E[\eta] / E)$-equivariantly on $\eta$.
\end{conjectureintro}

Note that it suffices to prove the conjecture (for all $\Pi$) for one choice of the sign $\pm$, since the functional equation interchanges $\Crit^{\pm}(\Pi)$ and $\Crit^{\mp}(\Pi^\vee)$; see  \S\ref{sec:functional equation} below. A partial result towards this conjecture is known for $n = 3$, by work of Mahnkopf and Kasten--Schmidt (see \S\ref{sec:mahnkopf}). Our first main result is:

\begin{theoremintro}\label{thm:intro-alg}
 \cref{conj:intro-alg} holds for $n = 3$.
\end{theoremintro}

\subsubsection*{P-adic interpolation}

Now suppose $p$ is a prime.  The second conjecture we state builds on the algebraicity conjecture \ref{conj:intro-alg} by predicting the algebraic parts of $L$-values vary $p$-adic analytically as we deform $(\eta,j)$. It implies that these values satisfy remarkable $p$-adic congruences, akin to higher-dimensional analogues of the famous Kummer congruences for the Riemann zeta function.

Let $P_i$ be the block-upper-triangular parabolic subgroup with Levi subgroup $\GL_i \times \GL_{n-i}$. A \emph{$P_i$-refinement} is a choice of irreducible representation of $\GL_i \times \GL_{n-i}$ appearing in the Jacquet module $J_{P_i}(\Pi_p)$. We define in \cref{sec:ordinarity} below a notion of \emph{slope} for $P_i$-refinements, and we say $\Pi$ is \emph{$P_i$-nearly-ordinary} if it admits a $P_i$-refinement of the minimal possible slope. We are most interested in $i = m, m+1$.

Let $\Crit_p^\pm(\Pi) \subset \Crit^\pm(\Pi)$ be the subset of $(t, \eta)$ where $\eta$ has $p$-power conductor. For such $(t, \eta)$, let  $e_p(\Pi_p \times \eta_p, t) \in E[\eta, G(\eta)]$ denote the modified Euler factor at $p$ of \cite[\S2]{coates89} (see \cref{sec:coates--perrin-riou}). Let $L^{(p)}(\Pi\times\eta, s)$ be the $L$-function without its Euler factor at $p$. Then we have the following conjecture, which is (a generalisation of) the conjectures of Coates--Perrin-Riou and Panchishkin:

\begin{conjectureintro}[$p$-adic $L$-functions Conjecture for $\GL_{2m+1}$]\leavevmode
\label{conj:intro}
\begin{itemize}
	\item[(i)]	Suppose $\Pi$ is $P_m$-nearly-ordinary. Then there exists a $p$-adic measure $L_p^-(\Pi)$ on $\Zp^\times$ such that for all $(t,\eta) \in \Crit_p^-(\Pi)$, we have the interpolation
	\begin{equation}\label{eq:interpolation intro}
		\int_{\Zp^\times} \eta(x)^{-1} x^t \cdot \mathrm{d}L_p^-\left(\Pi\right)(x) = e_\infty\left(\Pi_\infty \times \eta_\infty, t\right)e_p\left(\Pi_p \times \eta_p, t\right)\cdot \frac{L^{(p)}(\Pi\times\eta, t)}{\Omega_\Pi^-},
	\end{equation}
 and $L_p^-(\Pi)$ vanishes on characters of sign $\ne \omega_\Pi(-1)$.

	\item[(ii)] Suppose $\Pi$ is $P_{m+1}$-nearly-ordinary. Then there exists a $p$-adic measure $L_p^+(\Pi)$ on $\Zp^\times$ such that for all $(t ,\eta) \in \Crit_p^+(\Pi)$, we have the interpolation
\[
		\int_{\Zp^\times} \eta(x)^{-1} x^t \cdot \mathrm{d}L_p^+\left(\Pi\right)(x) = e_\infty\left(\Pi_\infty \times \eta_\infty, t\right)e_p\left(\Pi_p \times \eta_p, t\right)\cdot \frac{L^{(p)}(\Pi\times\eta, t)}{\Omega_\Pi^+}
	\]
 and $L_p^+(\Pi)$ vanishes on characters of sign $\ne -\omega_{\Pi}(-1)$.
	\end{itemize}
\end{conjectureintro}

We call $L_p^{\pm}(\Pi)$ the \emph{left-half} and \emph{right-half} $p$-adic $L$-functions of $\Pi$. Each of these objects, if it exists, is uniquely  determined by the interpolation property in the theorem (since it is a measure). If $\Pi$ is nearly ordinary for both $P_m$ and $P_{m+1}$, then we can define a single $p$-adic $L$-function $L_p(\Pi) = L_p^+(\Pi) + L_p^-(\Pi)$ interpolating critical values in both halves of the critical strip; but the conditions of near-ordinarity at $P_m$ and $P_{m+1}$ are independent of each other, so it can occur that only one of these objects is well-defined.

In this paper, we prove:
\begin{theoremintro}\label{thm:intro}
	\cref{conj:intro} holds in full when $n=2m+1=3$.\qed
\end{theoremintro}

As with Theorem A, it suffices to prove Theorem B for one choice of the sign (and all $\Pi$). This is because $\Pi$ is $P_{m+1}$-nearly-ordinary if $\Pi^\vee$ is $P_m$-nearly-ordinary, and vice versa; and via the functional equation, we show in \S \ref{sec:functional equation} that part (i) of the conjecture for $\Pi$ is equivalent to part (ii) of the conjecture for $\Pi^\vee$. This also gives a $p$-adic functional equation relating $L_p^{\pm}(\Pi)$ and $L_p^{\mp}(\Pi^\vee)$.

\subsection{Mahnkopf's work on algebraicity}\label{sec:mahnkopf}

We now describe the starting-point for our constructions, which is Mahnkopf's demonstration of a weakened form of the Algebraicity Conjecture for $n = 3$. He uses the Rankin--Selberg integral for $\GL_3 \times \GL_2$: taken with $\Pi$ on $\GL_3$ and an Eisenstein series on $\GL_2$, this integral computes a product of (twisted) $L$-functions for $\Pi$.

In \cite{Mah98} Mahnkopf gave a cohomological interpretation of this integral, as we now sketch. We take $\Pi$ to have weight $\lambda = (a,0,-a)$, and let $V_\lambda$ be the $\GL_3$-representation of highest weight $\lambda$.  Note in particular now that
\begin{equation}\label{eq:crit pi intro}
	\Crit^-_p(\Pi) = \{(-j,\eta) : 0 \leq j \leq a, \ \mathrm{cond}(\eta) \mid p^\infty, \  \eta\omega_{\Pi}(-1) = (-1)^{j}\}.
\end{equation}
For $j \geq 0$, let $V_{(0,-j)}^{\GL_2}$ be the $\GL_2$-representation of highest weight $(x,y) \mapsto y^{-j}$. Let $Y^{\GL_3}(\cU)$ be the locally symmetric space for $\GL_3$ (of some neat level $\cU$), and let $Y^{\GL_2}_1(p^n)$ be the locally symmetric space (of level $\Gamma_1(p^n)$) for $\GL_2$.  For a Dirichlet character $\eta$ of conductor $p^n$, one has:
\begin{itemize}\setlength{\itemsep}{0pt}
	\item[--] a compactly supported Betti class $\phi_\varphi \in \hc{2}(Y^{\GL_3}(\cU), V_\lambda^\vee(\C))$ attached to any $\varphi\in\Pi\subf$,
	\item[--] and Harder's weight $j$ Eisenstein class $\mathrm{Eis}^{j,\eta} \in \h^1\big(Y^{\GL_2}_1(p^n),V_{(0,-j)}^{\GL_2}(\C)\big)$.
\end{itemize}
We have the following crucial \emph{branching law} (cf.\ \cite[Lem.\ 3.1]{Mah00}):
\begin{equation}\label{eq:branching}
	\text{\emph{If $(-j,\eta) \in \Crit_p^-(\Pi)$, then $V_{(0,-j)}^{\GL_2} \subset V_\lambda\big|_{\GL_2}$ with multiplicity one.}}
\end{equation}
By pushing forward $\mathrm{Eis}^{j,\eta}$ to $Y^{\GL_3}$, and using \eqref{eq:branching}, one can then define a pairing
\begin{equation}\label{eq:intro pairing}
	\langle-,-\rangle : \hc{2}(Y^{\GL_3}(\cU), V_\lambda^\vee) \times  \h^1\left(Y^{\GL_2}_1(p^n),V_{(0,-j)}^{\GL_2}\right) \to \C
\end{equation}
such that $\langle \phi_\varphi,\mathrm{Eis}^{j,\eta}\rangle$ is a Rankin--Selberg integral; and for appropriate $\varphi$, one of the $L$-values it computes is $L(\Pi\times\eta, -j)$. Mahnkopf and Kasten--Schmidt used this in \cite{Mah98,Mah00,KS13} to prove that for each fixed $j$, \cref{conj:intro-alg} holds up to replacing $e_\infty(\Pi_\infty \times \eta_\infty, -j)$ with some inexplicit, but non-zero, $\widetilde{e}_\infty(\Pi_\infty \times \eta_\infty, -j) \in \C^\times$.

The new result we prove here, in order to complete the proof of \cref{thm:intro-alg}, is that the inexplicit quantity $\widetilde{e}_{\infty}(\Pi_\infty \times \eta_\infty, -j)$ in fact coincides with the modified archimedean factor $e_\infty(\Pi_\infty \times \eta_\infty, -j)$ of Coates--Perrin-Riou. Our proof of this will in fact be bound up with the proof of \cref{thm:intro}, which we now sketch.

\subsection{\texorpdfstring{$p$-adic interpolation}{P-adic interpolation}}
\label{sec:p-adic interpolation}
 In light of the equivalence between parts (i) and (ii) of the $p$-adic $L$-functions conjecture, it suffices to focus on part (i), and consider $P_1$-nearly-ordinary refinements and twists in $\Crit^-_p(\Pi)$.

Mahnkopf's work expresses $L(\Pi\times\eta,-j)$ as an integral involving the pushforward to $\GL_3$ of the Eisenstein class $\mathrm{Eis}^{j,\eta}$. Accordingly, for $p$-adic interpolation of $L$-values the aim is clear; we must:
\begin{enumerate}[(I)]\setlength{\itemsep}{0pt}
	\item\label{introsketch1} $p$-adically interpolate the pushforward of $\mathrm{Eis}^{j,\eta}$ to $Y^{\GL_3}$ as $\eta$ (hence $n$) varies;
	\item\label{introsketch2} $p$-adically interpolate the pushforward of $\mathrm{Eis}^{j,\eta}$ to $Y^{\GL_3}$ as $j$ varies.
\end{enumerate}
This strategy was already known to Mahnkopf over 20 years ago, and work towards (\ref{introsketch1}) was the focus of \cite{Mah00} and \cite{Ger15}. However, these works did not ultimately lead to constructions of $p$-adic $L$-functions, since their methods did not give sufficient control on the denominators of Eisenstein classes to construct a uniquely-determined $p$-adic $L$-function (see \S\ref{sec:literature} below).

In this paper, we solve (\ref{introsketch1}) and (\ref{introsketch2}). Our key innovations are:
\begin{itemize}\setlength{\itemsep}{0pt}
\item[--] the use of Be\u{\i}linson's \emph{motivic} Eisenstein classes $\mathrm{Eis}^j_{\mathrm{mot},\Phi_{\mathrm{f},n}}$ (see \cref{sec:Betti-Eisenstein}), and
\item[--] the  systematic variation of a \emph{third} parameter:  the definition of the pairing \eqref{eq:intro pairing}.
\end{itemize}
The motivic classes are \emph{norm compatible}, so their Betti realisations -- the \emph{Betti--Eisenstein classes} $\mathrm{Eis}^j_{\Phi_{\mathrm{f},n}}$ -- form a tower as $n$ varies. Kings has shown they interpolate well as $j$ varies (see \cref{sec:Eisenstein interpolation}). Whilst these classes are not integral, crucially for an auxiliary integer $c$ one can define `$c$-smoothed' modifications ${}_c \mathrm{Eis}_{\Phi_{\mathrm{f},n}}^j$, that \emph{are} integral and which remain norm-compatible (see \cref{sec:Eisenstein integrality}).

By varying the pairing, we build a machine that ports these compatibility and integrality properties from $\GL_2$ to $\GL_3$. To make this more precise, we elaborate on our approach, which is somewhat different to \cite{Mah98} from the outset. Let $H\defeq \GL_2\times \GL_1$, and let $\iota : H \hookrightarrow \GL_3$ be the map $(\gamma,z) \mapsto \smallmatrd{\gamma}{}{}{z}$. Our construction involves pulling back Eisenstein classes under the natural projection $H \to \GL_2$, applying the branching law \eqref{eq:branching}, pushing forward under $\iota : H \to \GL_3$, and then twisting by a certain operator $u\tau^n \in \GL_3(\Qp)$, where $u = \smallthreemat{1}{0}{1}{}{1}{0}{}{}{1}\smallthreemat{1}{}{}{}{}{1}{}{-1}{}$ and $\tau = \mathrm{diag}(p,1,1)$. Varying the parameters in this process varies the pairing \eqref{eq:intro pairing}. Ultimately, by `spreading out' this twisted pushforward map over a group algebra, we construct a machine that is, at level $n$, a map
\begin{align*}
\h^1\left(Y^{\GL_2}_1(p^n),V_{(0,-j)}^{\GL_2}(\Zp)\right) &\to \h^3\Big(Y^{\GL_3}(\cU),V_\lambda(\Zp)\Big) \otimes \Zp[(\Z/p^n)^\times],\\
{}_c\mathrm{Eis}^{j}_{\Phi_{\mathrm{f},n}} &\mapsto {}_c\xi_n^{[j]}.
\end{align*}
Note that since $Y^{\GL_3}$ is 5-dimensional, $\h^3(Y^{\GL_3}(\cU),V_\lambda)$ is Poincar\'e dual to $\hc{2}(Y^{\GL_3}(\cU),V_\lambda^\vee)$.

Importantly, $\overline{P}_1(\Zp) \cdot u \cdot H(\Zp)$ is open and dense in $\GL_3(\Zp)$, where we recall the parabolic $P_1 \subset \GL_3$ from above. This provides a link with the theory of \emph{spherical varieties}, via $\GL_3/H$, as explored in \cite{loeffler-parabolics}. Exploiting this link, in \cref{thm:norm relation} we show that for each fixed $0 \leq j \leq a$ as in \eqref{eq:crit pi intro}, the machine preserves norm-compatibility, but not on the nose: rather, for the Hecke operator $U_{p,1}$ at $P_1$, we prove that
\[
\mathrm{Norm}^{n+1}_n\big({}_c\xi_{n+1}^{[j]}\big) = U_{p,1} \big( {}_c\xi_n^{[j]}\big).
\]

Now let $\varphi \in \Pi$ be a suitable $U_{p,1}$-eigenform with an associated class $\phi_\varphi \in \hc{2}(Y^{\GL_3}(\cU),V_\lambda^\vee(\cO))$, where $\cO$ is a finite extension of $\Zp$. Let $\alpha_{p,1}$ be the $U_{p,1}$-eigenvalue of $\varphi$.  Suppose $\varphi$ is $P_1$-nearly-ordinary, i.e. $\alpha_{p,1} \in \cO^\times$; then for each fixed $j$, pairing $\alpha_{p,1}^{-n}\phi_\varphi$ with ${}_c\xi_n^{[j]}$ (under Poincar\'{e} duality \eqref{eq:poincare}) yields an element ${}_c\Xi^{[j]} \in \varprojlim \cO[(\Z/p^n)^\times] \cong \cO[\![\Zp^\times]\!]$. In \cref{sec:measure to integral}, we show that if $(-j,\eta) \in \Crit^-_p(\Pi)$, then integrating this measure against $\eta$ on $\Zp^\times$ interpolates a relevant Rankin--Selberg integral. In \cref{sec:local zeta integrals,sec:p-adic L-function} we show that for appropriate finite test data, this integral computes (a $c$-smoothed multiple of) the right-hand side of \eqref{eq:interpolation intro}, up to the precise factor at $\infty$. Taking this test data as input to our machine, and renormalising, yields a family of measures $\Xi^{[j]}$ that do not depend on the smoothing integer $c$ (see \S\ref{sec:getting rid of c}), such that $\Xi^{[j]}$ interpolates multiples of $L^{(p)}(\Pi\times\eta,-j)$ as $\eta$ varies. This then solves (\ref{introsketch1}).

It remains to prove (\ref{introsketch2}), the compatibility between the $a+1$ measures $\{\Xi^{[j]}\}_{j=0}^a$. This involves $p$-adically interpolating the branching law \eqref{eq:branching} as $j$ varies. Again we exploit the connection to spherical varieties: our set-up puts us in the framework of \cite{spherical_II}, which we apply to prove
\[
\int_{\Zp^\times} x^{j}f(x) \cdot\mathrm{d}\Xi^{[0]}(x)= \int_{\Zp^\times} f(x) \cdot \mathrm{d}\Xi^{[j]}(x),
\]
i.e.\ the $ \Xi^{[j]}$ are Tate twists of each other.  Thus the measure $\Xi^{[0]}$ on $\Zp^\times$ satisfies
\[
\int_{\Zp^\times} \eta(x)x^j \cdot  d\Xi^{[0]}(x) = 	\int_{\Zp^\times} \eta(x) \cdot  \Xi^{[j]}=  (*) \cdot L^{(p)}(\Pi\times\eta, -j),
\]
where the term $(*)$ is a product of a global period and certain local zeta integrals. Away from infinity, for a good choice of test data, these local zeta factors are explicitly computed in \cref{sec:local zeta integrals}, and shown to compute the correct interpolation factors. We define $L_p^-(\Pi)$ by precomposing $\Xi^{[0]}$ with the involution $\chi \mapsto \chi^{-1}$ on characters of $\Zp^\times$.

This leaves the local zeta integral at $\infty$. By definition, for each $j$ this Rankin--Selberg integral is the inexplicit factor $\widetilde{e}_\infty(\Pi_\infty \times \eta_\infty, -j)$ from \cite{Mah00}, which is non-zero by \cite{KS13,Sun17}. In particular, $L_p^-(\Pi)$ satisfies \eqref{eq:interpolation intro} up to replacing $e_\infty(\Pi_\infty \times \eta_\infty,-j)$ with $\widetilde{e}_\infty(\Pi_\infty \times \eta_\infty, -j)$.

\subsection{\texorpdfstring{The factor at $\infty$}{The factor at infinity}}\label{sec:intro alg}
To complete the proof of \cref{conj:intro}(i), it remains to prove $\widetilde{e}_\infty(\Pi_\infty \times \eta_\infty,-j) = e_\infty(\Pi_\infty \times \eta_\infty,-j)$, i.e.\ that the factor at infinity has the expected form.

By construction, the factor $\widetilde{e}_\infty(\Pi_\infty \times \eta_\infty,-j)$ depends only on $j$ and $\Pi_\infty$, and in turn $\Pi_\infty$ depends only on $\lambda$ and $\omega_\Pi$. Whilst we do not evaluate this integral directly, we \emph{do} know that it is non-vanishing (for all $\Pi_\infty$ and $j$) by \cite{KS13}. To get the correct interpolation factor $e_\infty$, we exploit the fact that \cref{thm:intro} is already known in full when $\Pi$ is a (twist of a) symmetric square from $\GL_2$. In this case, we compute the ratio of our measure $L_p^-(\Pi)$ with the symmetric square $p$-adic $L$-function, and show this ratio is constant. By (\ref{introsketch2}), we thus deduce that the $\widetilde{e}_\infty(\Pi_\infty \times \eta_\infty,-j)$'s satisfy the expected compatibility as $j$ varies. Up to a global renormalisation, we thus deduce $\widetilde{e}_\infty(\Pi_\infty \times \eta_\infty,-j) = e_\infty(\Pi_\infty \times \eta_\infty,-j)$ for all $-j$ in the left-half %\footnote{By considering $\Pi^\vee$ and reflecting in the functional equation, we also obtain this in the right-half.}
of the critical strip. Combining with \cref{sec:p-adic interpolation}, we deduce $L_p^-(\Pi)$ satisfies \eqref{eq:interpolation intro}, completing the proof of Theorems \ref{thm:intro-alg} and \ref{thm:intro} for sign ``$-$''. In \S\ref{sec:functional equation} we use the functional equation to show that this implies the sign ``$+$'' case as well, thus completing the proof.

\subsection{Relation to previous literature}\label{sec:literature}

The most notable partial results towards \cref{thm:intro} came in works of Mahnkopf \cite{Mah00} (for trivial weight) and Geroldinger \cite{Ger15} (cohomological weight), who gave constructions of algebraic $p$-adic distributions satisfying a partial form of the interpolation \eqref{eq:interpolation intro}. More precisely, for each critical integer $j$ for $\Pi$, they constructed a separate algebraic distribution $L_p^j(\Pi)$ (denoted $\mu_{\Pi,j}$ in \emph{op.\ cit}.) satisfying a formula similar to \eqref{eq:interpolation intro} for that fixed $j$ and all sufficiently ramified $\eta$.

However, using their methods, they were not able to sufficiently control the denominators of their distributions, or obtain any kind of compatibility for varying $j$. In particular, their distributions only defined functions on the set of \emph{locally-algebraic} characters of $\Zp^\times$ of degree $\le a$; they could not prove that their distributions had any uniquely-determined extension from this discrete set to the whole weight space of continuous characters of $\Zp^\times$. Thus their methods did not give a $p$-adic $L$-function in any usual sense (notwithstanding the titles of these papers). This is the fundamental problem we overcome in the present work: our methods give uniform control over the denominators of Eisenstein classes, allowing us to construct a uniquely-determined, bounded measure interpolating all critical values.

(Whilst our $p$-adic methods differ totally from those of \cite{Mah00} and \cite{Ger15}, these references do contain excellent accounts of a number of the automorphic aspects we require, particularly the local zeta integral at infinity.)

In the special case when $\Pi$ is a symmetric square lift of a RACAR of $\GL_2(\A_{\Q})$, a different construction of the $p$-adic $L$-function is possible, using Rankin--Selberg convolutions of the $\GL_2$ cusp form with half-integer weight theta-series; see e.g.~\cite{Sch88,Hid90,DD97}). However, these methods are completely specific to the case of symmetric-square lifts.

\subsection{Possible generalisations} Finally, we comment on possible future generalisations of the method developed here. Most immediately, an appropriate modification of our method should give a generalisation to non-critical slope RACARs of $\GL_3(\A_{\Q})$, which is being pursued by Dimitrakopoulou--Rockwood. It should also be possible to generalise these results to RACARs of $\GL_3(\A_F)$, where $F$ is a totally real field, by replacing Beilinson's motivic Eisenstein classes with the Hilbert--Eisenstein classes of Beilinson--Kings--Levin \cite{BKL14}. Proposition 3.5.5 \emph{op.\ cit}.\ yields the Hilbert analogue of \S\ref{sec:Eisenstein interpolation} of the present paper.

In the longer term, we hope to apply similar methods to construct $p$-adic $L$-functions for RACARs of $\GL_n$, over $\Q$ or totally real fields. For this, one requires a supply of well-behaved Eisenstein classes for $\GL_{n-1}$. The direct generalisation of our method is not obvious: for $n>3$, the group $\GL_{n-1}$ does not give rise to a Shimura datum, and it is not clear what the appropriate generalisation of \emph{motivic} Eisenstein classes should be. However, our methods only require the \emph{Betti realisation} of motivic classes, which are purely topological. It is reasonable to ask if there are norm-compatible families of integral Betti--Eisenstein classes on $\GL_{n-1}$, acting as a `Betti shadow' of some deeper, but presently mysterious, motivic structure; and with our methods, such classes should give a $p$-adic $L$-function for $\GL_n$.  In this direction, a direct construction of such Betti--Eisenstein classes for $\GL_2$ over totally real fields was given in \cite{NamAsai}.

\subsection*{Acknowledgements}
The authors would like to thank Chris Skinner, who introduced us to the argument in \cref{sec:interpolation 2} for computing ratios of local zeta integrals at infinity, and Harald Grobner and G\"unter Harder, who shared their expertise on \cref{conj:intro-alg}. We also thank Sarah Zerbes for numerous enlightening discussions, and Shih-Yu Chen, Andy Graham, and the referee for valuable comments on earlier drafts.

%%%%%%%%%%%%%%%%%%%%%%%%%%%%%%%%%%%%%%%%%%%%%%%%%%%%%%%%%%%%%%%%%%%%
%
%                          PRELIMINARIES
%
%%%%%%%%%%%%%%%%%%%%%%%%%%%%%%%%%%%%%%%%%%%%%%%%%%%%%%%%%%%%%%%%%%%%
\section{Preliminaries: automorphic representations}

\subsection{Characters}
\label{sect:dirichlet}

If $\chi: (\Z / N\Z)^\times \to \C^\times$ is a Dirichlet character, there is a unique finite-order Hecke character $\chihat : \Q^\times \backslash \A^\times / \R^\times_{> 0} \to \C^\times$ such that $\chihat(\varpi_\ell) = \chi(\ell)$ for all primes $\ell \nmid N$, where $\varpi_\ell$ is a uniformiser at $\ell$; conversely, every finite-order Hecke character is $\widehat\chi$ for a unique primitive Dirichlet character $\chi$.

Note that the restriction of $\chihat$ to $\widehat{\Z}^\times \subset \A\subf^\times$ is the \emph{inverse} of the composition $\widehat{\Z}^\times \twoheadrightarrow (\Z / N\Z)^\times \xrightarrow{\ \chi\ } \C^\times$. Moreover, if $\ell \mid N$, then $\widehat{\chi}_\ell(\varpi_\ell) = \chi^{(\ell)}(\ell)$, where we have written $\chi = \chi_\ell \chi^{(\ell)}$ a product of characters of $\ell$-power and prime-to-$\ell$ conductor. Given a representation $\Pi$ of $\GL_3(\A)$, we write $\Pi \times \chi$ for the representation $\Pi \times [\chihat \circ \det]$.

We let $\psi$ be the unique character of the additive group $\A / \Q$ such that $\psi(x) = \exp(-2\pi i x)$ for $x \in \R$. Note that the restriction of $\psi$ to $\Q_\ell$, for a finite prime $\ell$, maps $1/\ell^n$ to $\exp(2\pi i / \ell^n)$ for all $n \in \Z$.   Let $\varepsilon(\chihat_\ell,\psi_\ell)$ be the local $\varepsilon$-factor (with respect to the unramified Haar measure $\mathrm{d}x$), as defined in \cite{Tat79} for example. Then we have
\[
\prod_{\ell \mid N} \varepsilon(\chihat_\ell, \psi_\ell) = G(\chi) \defeq \sum_{a \in (\Z/ N\Z)^\times} \chi(a) \exp(2\pi i a / N)
\]
for $\chi$ primitive of conductor $N$. (We do not include the Archimedean root number here.)

\subsection{Algebraic representations}\label{sec:alg rep}

If $J$ is a split reductive group, and $\mu$ a $B_J$-dominant weight for some choice of Borel $B_J \subset J$, then we write $V_\mu^J$ for the unique irreducible algebraic representation of $J$ of highest weight $\mu$. When $J = \GL_3$, we will drop the $J$ and just write $V_\mu$.

Let $T$ denote the diagonal torus of $\GL_3$ (identified with $(\mathbb{G}_m)^3$ in the obvious fashion), and $B = T\ltimes N$ the upper-triangular Borel subgroup. The $B$-dominant weights for $\GL_3$ are of the form $\lambda = (a,b,c) \in \Z^3$, with $a \ge b \ge c$. If $E$ is any $\Q$-algebra, then we can realise $V_\lambda(E)$ as a space of polynomial functions on $\GL_3$, via
\begin{align*}
V_\lambda(E) = \{ f : \GL_3(E)& \to E \ : \ f\text{ algebraic, }\\
& f(n^-tg) = \lambda(t)f(g) \ \forall n^- \in N^-(E), t \in T(E)\},
\end{align*}
where $N^-(E)$ is the unipotent radical of the opposite Borel. We get a natural left-action of $\gamma \in \GL_3(E)$ on $V_\lambda(E)$ by $(\gamma \cdot f)(g) = f(g\gamma)$. Let $V_\lambda^\vee(E)$ denote the $E$-linear dual, with the dual left action $(\gamma \cdot \mu)(f) \defeq \mu(\gamma^{-1} \cdot f)$.

A weight $\lambda = (a,b,c)$ is \emph{pure} if $a + c = 2b$. These are precisely the weights such that $V_{\lambda}^\vee$ is isomorphic to a twist of $V_{\lambda}$.

%%%%%%%%%%%%%%%%%%%%%%%%%%%%%%%%%%%%%%%%%%%%%%%%
%                   Auto Repns
%%%%%%%%%%%%%%%%%%%%%%%%%%%%%%%%%%%%%%%%%%%%%%%%
\subsection{Automorphic representations for \texorpdfstring{$\GL_3$}{GL3}}

We recall some standard facts about automorphic representations of $\GL_3$ (for a fuller account, see \cite[\S3.1]{Mah05}, summarising \cite[\S3]{Clo90}). Let $\Pi$ be a cuspidal automorphic representation of $\GL_3(\A)$, with central character $\omega_{\Pi}$. We identify $\Pi$ with its realisation in $L^2_0(\GL_3(\Q)\backslash \GL_3(\A))$, considering any $\varphi \in \Pi$ as a function on $\GL_3(\A)$.

Let $\mathfrak{gl}_3 = \mathrm{Lie}(\GL_3)$. Recall that the centre of the universal enveloping algebra at $\infty$ acts on $\Pi_\infty$ via a ring homomorphism $Z(U(\mathfrak{gl}_3)_{\C}) \to \C$ (the \emph{infinitesimal character} of $\Pi_\infty$). We say $\Pi$ is \emph{regular algebraic of weight $\lambda$} if $\Pi_\infty$ has the same infinitesimal character as the irreducible algebraic representation $V_\lambda$, for some dominant integral weight $\lambda$. (This determines $\lambda$ uniquely.) We use the abbreviation ``RACAR'' for ``regular algebraic cuspidal automorphic representation''.

\subsubsection{Whittaker models}\label{sec:whittaker}
We denote the standard Whittaker model of $\Pi$ by
\begin{align*}
\cW_\psi : \Pi &\isorightarrow \cW_\psi(\Pi) \subset \mathrm{Ind}_{N(\A)}^{\GL_3(\A)}\psi,\\
\varphi &\mapsto W_\varphi(g) \defeq \int_{N(\Q)\backslash N(\A)}\varphi(ng)\psi^{-1}(n) \mathrm{d}n.
\end{align*}
where $\psi\smallthreemat{1}{x}{\star}{}{1}{y}{}{}{1} \defeq \psi(x+y)$. As $\psi$ is fixed throughout, we will often drop it from notation. We denote cusp forms in $\Pi$ by $\varphi$, and elements of $\cW_\psi(\Pi)$ by $W$.

\subsubsection{Cohomological automorphic representations}\label{sec:cohomological}

Let $K_{J,\infty}$ be a maximal compact subgroup of $\GL_3(\R)$,  $Z_{\GL_3,\infty} = Z_{\GL_3}(\R)$, and $(-)^\circ$ denote the identity component. Write $K_{3,\infty}^\circ =$ $K_{\GL_3,\infty}^\circ Z_{\GL_3,\infty}^\circ$ for shorthand. We say $\Pi$ is \emph{cohomological} with coefficients in an algebraic representation $W$ if
\begin{equation}\label{eq:infinite cohomology}
\h^\bullet\big(\mathfrak{gl}_3,K_{3,\infty}^\circ; \Pi_\infty \otimes W(\C) \big) \neq 0.
\end{equation}

\begin{proposition}\emph{\cite[Lem.\ 4.9]{Clo90}}.
If $\Pi$ is a RACAR of weight $\lambda$, then it is cohomological with coefficients in $W = V_{\lambda}^\vee$ (and this is the unique irreducible representation for which $\Pi$ is cohomological). Moreover, $\lambda$ is necessarily pure (as in \cref{sec:alg rep}).\qed
\end{proposition}

This cohomology is then supported in degrees $2$ and $3$ \cite[Lem.\ 3.14]{Clo90}, and in each of these degrees \eqref{eq:infinite cohomology} is necessarily one-dimensional \cite[(3.2),(3.4)]{Mah05}. We will consider throughout only the lowest degree $i= 2$; exactly as in \cite[\S3.1]{Mah00} (where it is denoted $\omega_\infty$) we choose a generator
\begin{equation}\label{eq:omega infinity}
\zeta_\infty \in \h^2\big(\mathfrak{gl}_3,K_{3,\infty}^\circ; \Pi_\infty \otimes V_{\lambda}^{\vee}(\C) \big).
\end{equation}

\begin{convention}
Let $\Pi$ be a RACAR of weight $\lambda$. As in \cite[\S1]{Mah00} (noting Mahnkopf's $l_0$ is our $2a+2$), without loss of generality we may (and always will) normalise so that $b = 0$, so (by purity) $\lambda = (a,0,-a)$ for some $a \geq 0$. In this case, we see that
\begin{equation}\label{eq:Pi_infty}
	\Pi_\infty \cong \mathrm{Ind}_{P_2(\R)}^{\GL_3(\R)} (D_{2a+3}, \mathrm{id}) \quad\text{or}\quad \mathrm{Ind}_{P_2(\R)}^{\GL_3(\R)} (D_{2a+3}, \mathrm{sgn}),
\end{equation}
where $P_2$ is the parabolic with Levi $\GL_2 \times \GL_1$, $D_{2a+3}$ is the discrete series representation of $\GL_2(\R)$ of lowest weight $2a + 3$, and $\mathrm{sgn}$ is the sign character. In particular, this implies that its central character $\omegahat_{\Pi}$ has finite order (i.e.~it is the adelic character associated to a Dirichlet character $\omega_\Pi$), and hence $\Pi$ is unitary.
\end{convention}

\begin{remark}
If $\omega_{\Pi}$ is \emph{odd}, then $\Pi_\infty \cong \mathrm{Ind}_{P_2(\R)}^{\GL_3(\R)} (D_{2a+3}, \mathrm{id})$, and $\Pi_\infty$ is the twist of this by $\mathrm{sgn}$ when $\omega_{\Pi}$ is even. In \cite{Mah00}, only the case of $\omega_{\Pi}$ odd is considered, but we allow both signs here.
\end{remark}

\subsubsection{Self-duality}

We say $\Pi$ is \emph{self-dual} if $\Pi^\vee \cong \Pi$, and more generally \emph{essentially self-dual} if $\Pi$ is isomorphic to a twist of $\Pi^\vee$. A theorem of Ramakrishnan \cite{Ram14} shows that if $\Pi$ is an essentially self-dual RACAR of $\GL(3)$, then there exists a non-CM-type cuspidal modular newform $f$ of weight $a + 2$, and a character $\nu$, such that $\Pi = \operatorname{Sym}^2(f) \otimes \nu$.

\begin{remark}
Ash and Pollack have conjectured \cite{AP08} that all level 1 RACARs of $\GL_3$ are self-dual, and arise as symmetric squares of level 1 cuspidal eigenforms for $\GL_2$. Examples which are not essentially self-dual do exist in higher levels; see the tables of \cite{GKTV97} for examples.
\end{remark}

\subsection{\texorpdfstring{The $L$-function of $\Pi$}{The L-function of Pi}}
\label{sec:L-function}

We let $L(\Pi, s) = \prod_{\ell < \infty} L(\Pi_\ell, s)$ denote the standard $L$-function of $\Pi$ (\emph{without} its Archimedean $\Gamma$-factors). We use the analytic normalisations here, so the Euler product defining $L(\Pi, s)$ converges for $\Re(s) > 1$.

Since $\Pi$ is cohomological, it is $C$-algebraic in the sense of \cite{BG14}, i.e.~there exists a number field $E$ such that $\Pif$ is definable as an $E$-linear representation. Since half the sum of the positive roots is in the weight lattice for $\GL_3$, it is also $L$-algebraic: that is, if $E$ is any number field over which $\Pif$ is definable, then for primes $\ell$ such that $\Pi_\ell$ is unramified, we have
\[
L(\Pi_\ell, s) = P_\ell\left(\ell^{-s}\right)^{-1}, \qquad P_\ell(X) = (1 - \alpha_\ell X)(1 - \beta_\ell X)(1 - \gamma_\ell X) \in E[X],
\]
where $\alpha_\ell, \beta_\ell, \gamma_\ell$ are units outside $\ell$, and have valuation $\ge -1-a$ at $\ell$. If $\ell$ is a ramified prime, then we still have $L(\Pi_\ell, s) = P_\ell(\ell^{-s})$ for some polynomial $P_\ell \in 1 + X E[X]$, but of degree $< 3$.

As $\Pi_\infty$ is given by $(2a+2)\oplus(\pm,0)$ in the notation of \cite[\S3]{Kna94}, at infinity we have
\[
L_\infty(\Pi_\infty \times \eta_\infty, s) = \Gamma_\R(s + 1 - \kappa) \Gamma_{\C}(s + a + 1), \qquad \kappa = \begin{cases}0 & \omega_\Pi\eta\text{ is even}\\1 & \omega_\Pi\eta \text{ is odd,}\end{cases}
\]
%where $\kappa = 0$ if $\omega_{\Pi}\eta$ is even or $\kappa = 1$ if $\omega_{\Pi}\eta$ is odd,
where $\Gamma_{\R}(s) = (2\pi)^{-s/2}\Gamma(s/2)$, and $\Gamma_{\C}(s) \defeq 2(2\pi)^{-s}\Gamma(s)$.
\begin{proposition}
\label{prop:critical}
Let $\eta$ be a Dirichlet character. Then the critical values of $L(\Pi \times \eta, s)$ are at $s = t$ for integers $t$ satisfying either
\[ \{ -a \le t \le 0 \text{ and } (-1)^{t} = \omega_{\Pi}\eta(-1) \}\]
or
\[ \{ 1 \le t \le 1+a \text{ and } (-1)^{t} = -\omega_{\Pi}\eta(-1) \}.\]
\end{proposition}

Note that none of these critical values can be zero (since $L(\Pi \times \eta, s) \ne 0$ for $\operatorname{Re}(s) \ge 1$, giving non-vanishing over $\mathrm{Crit}^+$; and non-vanishing over $\mathrm{Crit}^-$ follows via the functional equation). In particular, the near-central values $s = 0$ and $s = 1$ of $L(\Pi, s)$ are critical if and only if $\omega_{\Pi}$ is even.

\begin{remark}
An important example of $\Pi$ satisfying our conditions is the (normalised) symmetric square lift of a modular form $f$ of weight $k = a + 2 \ge 2$ and character $\varepsilon_f$. Then we have $L(\Pi, s) = L(\operatorname{Sym}^2 f, s + a+1)$ and $\omega_{\Pi}(-1) = (-1)^a$, so the above is consistent with the fact that $L(\operatorname{Sym}^2(f), 1)$ and $L(\operatorname{Sym}^2(f) \times \varepsilon_f^{-1}, k-1)$ are always critical values (independent of $a$).
\end{remark}

\subsubsection{Galois representations}\label{sec:galois}
By \cite{HLTT}, for each prime $p$ and embedding $\iota: E \hookrightarrow \overline{\Q}_p$,  there is a Galois representation $\rho_{\Pi, \iota}: \operatorname{Gal}(\overline{\Q} / \Q) \to \GL_3(\overline{\Q}_p)$, uniquely determined up to semisimplification, such that for all primes $\ell \ne p$ such that $\Pi_\ell$ is unramified, we have
\[ \det(1 - X \rho_{\Pi, \iota}(\operatorname{Frob}_{\ell}^{-1})) = P_\ell(X)\]
Here $\operatorname{Frob}_\ell$ is an arithmetic, and hence $\operatorname{Frob}_{\ell}^{-1}$ a geometric, Frobenius.

Conjecturally $\rho_{\Pi, \iota}$ should be de Rham at $p$, with Hodge--Tate weights $(-1-a, 0, 1+a)$; and if $\Pi_p$ is unramified, then it should be crystalline at $p$, and the eigenvalues of $\varphi$ on $\mathbb{D}_{\mathrm{cris}}(\rho_{\Pi, \iota})$ should be $( \alpha_p, \beta_p, \gamma_p)$. More generally, even if $\Pi_p$ is ramified, the Weil--Deligne representation $\mathbb{D}_{\mathrm{pst}}(\rho_{\Pi, \iota})$ should be related to $\Pi_p$ via the local Langlands correspondence. This conjecture is ``local-global compatibility for $\ell = p$''; it is known if $\Pi$ is essentially self-dual. None of our results will logically rely on this conjecture, or indeed on the existence of $\rho_{\Pi, \iota}$; but it serves as important motivation to explain why the statements are natural ones.

\subsection{\texorpdfstring{Rationality of $L$-values}{Rationality of L-values}}
\label{sec:rationality}

Recall \cref{prop:critical}, and let
\begin{align*}
\Crit^-(\Pi) &= \{ (-j, \eta): 0 \le j \le a,\ (-1)^j = \omega_{\Pi}\eta(-1)\} \\
\Crit^+(\Pi) &= \{ (j+1, \eta): 0 \le j \le a,\ (-1)^j = \omega_{\Pi}\eta(-1)\}
\end{align*}
so $L(\Pi \times \eta, j')$ is critical if and only if $(j', \eta) \in \Crit^-(\Pi) \sqcup \Crit^+(\Pi)$. For later use we write $\Crit^{\pm}_p(\Pi) = \{ (j', \eta) \in \Crit^{\pm}(\Pi): \eta \text{ has $p$-power conductor}\}$.

%\begin{convention}
%We will concentrate on $\Crit^-(\Pi)$, and replace $j$ with $-j$, with $0 \leq j \leq a$.
%\end{convention}

If $\Pi$ is a unitary RACAR of $\GL_3$, then conjecturally it has an attached motive $M_\Pi$ of weight 0. %We note that if $(-j,\eta) \in \Crit^-(\Pi)$, then
We note that for $t \in \Z$, the motive $M_\Pi(t)$ should have weight $-2t$, with Hodge decomposition $\h_B(M_\Pi(t)(\eta))\otimes \C = \h^{-t-a-1,-t+a+1} \oplus \h^{-t,-t} \oplus \h^{-t+a+1,-t-a-1}$, each summand 1-dimensional over $\C$. This motivates the following modified Euler factor at infinity:

\begin{definition}\label{def:modified infinity}
If $(-j,\eta) \in \Crit^-(\Pi)$, let 
\[
e_\infty(\Pi_\infty \times \eta_\infty, -j) \defeq i^{j-a-1} \cdot \Gamma_{\C}(a+1-j) = 2\cdot (2\pi i)^{j-a-1} \cdot \Gamma(a+1-j)
\]
corresponding to $t = -j$. If $(j+1,\eta) \in \Crit^+(\Pi)$, let
\begin{align*}
	e_\infty(\Pi_\infty \times \eta_\infty,j+1) &\defeq  (-i)^{-j-a-2}\cdot \Gamma_{\C}(a+2+j)\cdot  \frac{\Gamma_{\R}(j+1+\epsilon)}{(-i)^\epsilon \cdot \Gamma_{\R}(\epsilon -j)}
\end{align*}
	corresponding to $t = j+1$, and where $\epsilon = 0$ if $j$ is odd, and $\epsilon = 1$ if $j$ is even.
\end{definition}

	\begin{remarks}
  These factors would be denoted $\cL_\infty^{(i)}(M_\Pi(-j)(\eta))$ and $\cL_\infty^{(-i)}(M_\Pi(j+1)(\eta))$ respectively in \cite[\S1]{coates89}). In particular, we make opposite choices of parameter ($i$ vs.\ $-i$) in the two critical regions; and where defined, the factors are related by a functional equation
	 \begin{equation}\label{eq:modified infinity FE}
	 	e_\infty(\Pi_\infty^\vee,1 - t) = \gamma_\infty(\Pi_\infty, t) \cdot e_\infty(\Pi_\infty, t),
	 \end{equation}
	 (cf.\ equation (6) \emph{op.\ cit.}), where $\gamma_\infty(-)$ is the usual local $\gamma$-factor (defined with respect to $\psi$).
\end{remarks}

These factors $e_\infty(\Pi_\infty \times \eta_\infty, t)$ are the ones appearing in \cref{conj:intro-alg}. There are partial results towards this conjecture (see \cite{RagSac17} for an overview):

\begin{theorem}[Mahnkopf \cite{Mah98,Mah00}, Kasten--Schmidt \cite{KS13}]
\label{thm:mahnkopf}
 \cref{conj:intro-alg} holds for $n = 3$ up to replacing $e_\infty(\Pi_\infty \times \eta_\infty, t)$ with an inexplicit scalar $\tilde{e}_\infty(\Pi_\infty \times \eta_\infty, t) \in\C^\times$.
\end{theorem}

Note that it is not clear if the ratio $\frac{\tilde{e}_\infty(\Pi_\infty \times \eta_\infty, t)}{e_\infty(\Pi_\infty \times \eta_\infty, t)}$ is independent of $t$, so we cannot simply ``absorb'' it by renormalising the periods (although this is possible for $a = 0$, since then only one $t$-value can occur for each choice of sign $\pm$).

\begin{remark}\label{rem:periods}
We shall recall in \cref{sec:periods} the definition of a cohomological (Whittaker) period $\Theta_\Pi \in \C^\times$ associated to $\Pi$ (well defined up to $E^\times$). Note $(\omega_\Pi^{-1},1) \in \Crit^+(\Pi)$, and $L(\Pi\times\omega_\Pi^{-1},1) \neq 0$ (since it is far from the centre $s= \tfrac{1}{2}$ of the critical strip); then, precisely, we take $\Omega_\Pi^-$ to be an algebraic multiple of $\Theta_\Pi/L(\Pi\times\omega_\Pi^{-1},1)$. Analogously, $\Omega_\Pi^+$ will be an algebraic multiple of $\Theta_\Pi/L(\Pi\times \omega_\Pi^{-1}, 0)$.
\end{remark}

\subsection{Parabolics, refinements and near-ordinarity at \texorpdfstring{$p$}{p}}
 \label{sec:ordinarity}

 Let $p$ be a prime, and fix an embedding $\iota: E \hookrightarrow \overline{\Q}_p$. Consider the two maximal parabolic subgroups
 \[
  P_1 = \smallthreemat{\star}{\star}{\star}{0}{\star}{\star}{0}{\star}{\star}, \qquad
  P_2 = \smallthreemat{\star}{\star}{\star}{\star}{\star}{\star}{0}{0}{\star}.
 \]
 For each $i \in \{1, 2\}$, we can consider the normalised Jacquet module $J_{P_i}(\Pi_p)$, which is an admissible smooth representation of $(\GL_i \times \GL_{3-i})(\Q_p)$.

 \begin{definition}
  For $i = 1, 2$, a \emph{$P_i$-refinement} of $\Pi_p$ is a choice of an irreducible representation $\sigma_p \times \sigma'_p$ of $(\GL_i \times \GL_{3-i})(\Q_p)$ appearing as a subrepresentation of $J_{P_i}(\Pi_p)$. 
 \end{definition}
 
 Note that $\sigma_p\times \sigma'_p$ is a $P_1$-refinement of $\Pi_p$ if and only if $\sigma_p^{\prime \vee} \times \sigma_p^\vee$ is a $P_2$-refinement of $\Pi_p^\vee$. For a given $\Pi_p$ and $i$, $\sigma_p'$ is determined by $\sigma_p$ and vice versa, so we will often specify only $\sigma_p$. Accordingly, we say the $P_i$-refinement $\sigma_p \times \sigma_p'$ is \emph{unramified} if $\sigma_p$ is an unramified representation (but $\sigma_p'$ may be ramified). 
 
 \begin{definition}
  The \emph{slope} of a $P_i$-refinement $\sigma_p$ is the valuation of $\iota(\omega_{\sigma_p}(p))$, where $\omega_{\sigma_p}$ is the central character of $\sigma_p$ (and we assume, by enlarging $E$ if necessary, that $\sigma_p$ is defined over $E$).
 \end{definition}

 One can check that for a RACAR $\Pi$ of weight $(a, 0, -a)$, the slope of any $P_i$-refinement lies in the interval $[-1-a, 1 + a]$ (this follows from the relation to Hecke eigenvalues which we recall in the next section).

 \begin{definition}
  We say $\Pi$ is \emph{$P_i$-nearly-ordinary} if $\Pi_p$ admits a $P_i$-refinement of slope exactly $-1-a$. This refinement is unique if it exists. We say it is \emph{$P_i$-ordinary} if it is $P_i$-nearly-ordinary, and its unique nearly-ordinary $P_i$-refinement is unramified.
 \end{definition}

 \begin{remark}
  Clearly, if $\Pi_p$ is itself unramified, then any $P_i$-refinement of it must be unramified. In particular, `nearly-ordinary' and `ordinary' are equivalent if $\Pi_p$ is unramified. 
  
  Note, however, that ramified $\Pi_p$ can still admit unramified refinements.
 \end{remark}

 \begin{example}
  \label{ex:langlands}
  We briefly recall the classification of generic representations of $\GL_3$, and explain the conditions under which these are (nearly) ordinary. For simplicity, if $\chi$ is a character of $\Q_p^\times$ we write $v(\chi) = v_p(\iota(\chi(p))$.
  \begin{itemize}
  	\item If $\Pi_p$ is supercuspidal, it admits no $P_1$-refinements or $P_2$-refinements, and hence is never nearly-ordinary for any parabolic.

  	\item If $\Pi_p = \operatorname{St}_3 \otimes \lambda$ is a twist of the $\GL_3$ Steinberg representation by a character $\lambda$ (necessarily of finite order), then it has a unique $P_1$-refinement and a unique $P_2$-refinement, both of which have slope $-1$. Thus $\Pi_p$ is nearly-ordinary for both parabolics if $a = 0$ (and ordinary if $\lambda$ is unramified), but for $a > 0$ it is not nearly-ordinary for either $P_1$ or $P_2$.

  	\item If $\Pi_p$ is parabolically induced from a representation of $\GL_1 \times \GL_2$ of the form $\theta \times \pi$, where $\theta$ is a character and $\pi$ is supercuspidal, then its unique $P_1$-refinement is $\theta$, and its unique $P_2$-refinement is $\pi$. So it is $P_1$-nearly-ordinary if and only if $v(\theta) = -1-a$, whereas it is $P_2$-nearly-ordinary if and only if $v(\theta) = 1+a$. It is $P_1$-ordinary if it is $P_1$-nearly-ordinary and $\theta$ is unramified, while it is never $P_2$-ordinary.

  	\item If $\Pi_p$ is (irreducibly) induced from $\theta \times (\operatorname{St}_2 \otimes \lambda)$, where $\operatorname{St}_2$ is the $\GL_2$ Steinberg representation (and hence $\lambda^2 \theta = \hat\omega_{\Pi, p}$), then it has two $P_1$-refinements, namely $\theta$ and $\lambda |\cdot|^{1/2}$; note that $v(\theta) \le 1 + a$ implies $v(\lambda |\cdot|^{1/2}) \ge -1-\tfrac{a}{2}$. Thus $\Pi_p$ is $P_1$-nearly-ordinary in either of two (mutually exclusive) cases: if $v(\theta) = -1-a$ and $a$ is arbitrary; or if $a = 0$ and $v(\theta) = 1$. It is $P_1$-ordinary if $\theta$ is unramified in the former case, and if $\lambda$ is unramified in the latter. (There is a similar criterion for $P_2$-ordinarity, which we leave to the reader.)

  	\item If $\Pi_p$ is an irreducible principal series representation, induced from a character $\chi_1 \times \chi_2 \times \chi_3$ of the diagonal torus, then the possible $P_1$-refinements are exactly the $\chi_i$, and the $P_2$-refinements are the pairs $\{ \{\chi_1, \chi_2\}, \{\chi_2, \chi_3\}, \{\chi_3, \chi_1\}\}$. We can assume without loss of generality that $v(\chi_1) \le v(\chi_2) \le v(\chi_3)$; then $\Pi_p$ is $P_1$-nearly-ordinary if $v(\chi_1) = -1-a$, and $P_1$-ordinary if in addition $\chi_1$ is unramified; it is $P_2$-nearly-ordinary if $v(\chi_1\chi_2) = -1-a$, and $P_2$-ordinary if in addition $\chi_1$ and $\chi_2$ are unramified.
  \end{itemize}
 \end{example}

 \subsubsection*{Ordinarity and Galois representations}

  If $\sigma_p$ is a $P_i$-refinement of $\Pi_p$, then $\Pi_p$ is the unique generic constituent of $\operatorname{Ind}_{P_i}^G( \sigma_p \times \sigma'_p)$ for some $\sigma'_p$. Via compatibility of the local Langlands correspondence with parabolic induction, the Langlands parameter $\phi_{\Pi_p}$ has the form
  \[ \begin{pmatrix}\phi_{\sigma_p} & {\star}\\ 0 & \phi_{\sigma'_p} \end{pmatrix}. \]
  Assuming that the Galois representation $\rho_{\Pi, \iota}$ satisfies local-global compatibility at $p$, this gives an $i$--dimensional $(\varphi, N, G_{\Qp})$-stable subspace of $\mathbb{D}_{\mathrm{pst}}(\rho_{\Pi, \iota}|_{G_{\Qp}})$ isomorphic to $\phi_{\sigma_p}$.

  In general this does not arise from a subrepresentation of the Galois representation, since it may not be weakly admissible. However, the following is a straightforward check:

  \begin{proposition}
   Assuming the local-global compatibility conjecture, $\Pi_p$ is $P_i$-nearly-ordinary if and only if $\rho_{\Pi, \iota}|_{G_{\Qp}}$ preserves an $i$-dimensional subrepresentation accounting for the $i$ largest Hodge--Tate weights. Moreover, $\mathbb{D}_{\mathrm{pst}}$ of this subrepresentation is the Langlands parameter of the unique nearly-ordinary $P_i$-refinement.\qed
  \end{proposition}

  If the $P_r$-refinement $\sigma_p$ is unramified, then the Satake parameters of $\sigma_p$ are among the reciprocal roots of the Hecke polynomial $P_p(X)$. Thus the existence of an ordinary (rather than nearly-ordinary) $P_r$-refinement is equivalent to the Newton and Hodge polygons coinciding at $r$, which is the ``Panchishkin condition'' considered in \cite[Definition 5.5]{Pan94}.

  Thus $P_1$-ordinarity corresponds precisely to the \emph{Panchishkin condition} formulated in \cite{Pan94} for the existence of a $p$-adic $L$-function interpolating the $L$-values $L(\Pi \times \eta, -j)$, for $(-j, \eta) \in \Crit^-_p(\Pi)$. Similarly, $P_2$-ordinarity corresponds precisely to the Panchishkin condition for interpolation over $\Crit^+_p(\Pi)$.

 \subsection{The Coates--Perrin-Riou factor at \texorpdfstring{$p$}{p}}
  \label{sec:coates--perrin-riou}
  
  Let $t \in \Z$ be such that $(t, \mathrm{id}) \in \Crit(\Pi)$, and let $i = 1$ for $\Crit^-$ and $i = 2$ for $\Crit^+$. Let $\sigma_p$ be a $P_i$-refinement of $\Pi_p$.

  \begin{remark}
   In Galois-theoretic terms, $i$ is the number of Hodge--Tate weights of the twist $\rho_{\Pi, \iota}(t)$ which are strictly positive, and the condition that $L(\Pi, t)$ be critical implies that this is also the dimension of the $+1$ eigenspace for complex conjugation. So we are choosing a subrepresentation of the Weil--Deligne representation associated to $\rho_{\Pi, \iota}(t)$ whose dimension is equal to the number of positive Hodge--Tate weights.
  \end{remark}

  \begin{definition}\label{def:e_p}
  Recall $\psi$ is the additive character of $\A/\Q$ such that $\psi(x) = \mathrm{exp}(-2\pi i x)$ for $x \in \R$. We briefly denote this by $\psi^-$, and write $\psi^+$ for its inverse. 
   For $(t, \mathrm{id}) \in \Crit^\pm(\Pi)$, we set
   \[ e_p^\pm(\Pi_p, \sigma_p, t) = \gamma_p(\sigma_p, \psi^\pm, t)^{-1} \coloneqq \frac{L(\sigma_p, t)}{L(\sigma_p^\vee, 1 - t)\varepsilon_p(\sigma_p, \psi_p^\pm, t)} .\]
  \end{definition}

\begin{remark}
	    Here $\gamma_p(-)$ is the usual local $\gamma$-factor. 
	    
	    Note that in the epsilon factor, we make opposite choices of additive characters for the two critical regions; this ensures the validity of Lemma \ref{lem:FE at p}. 
	    Since the sign $\pm$ is determined by $t$, to avoid extraneous notation we henceforth always drop the $\pm$ (and $\psi^\pm$) from notation (except in the proof of \cref{lem:FE at p}), and just write $e_p(-)$ and $\gamma_p(-)$.
\end{remark}

   In this work we shall only consider nearly-ordinary refinements (although we hope that the more general finite-slope case will be studied in future works); since nearly-ordinary refinements are unique if they exist, we shall frequently omit $\sigma_p$ from the notation entirely, and write just $e_p(\Pi_p, t)$.

  \subsubsection*{Duality} If $\sigma_p\times\sigma'_p$ is a $P_i$-refinement of $\Pi_p$, then $\sigma_p^{\prime\vee}\times\sigma_p^\vee$ is a $P_{(3-i)}$-refinement of $\Pi_p^\vee$. Our hypotheses are symmetric under replacing $(\Pi, \sigma_p, t, i)$ and with $(\Pi^\vee, \sigma_p^{\prime \vee}, 1 - t, 3-i)$. We have the following relation between the Coates--Perrin-Riou factors for the two critical regions; compare \eqref{eq:modified infinity FE} above.
  
  \begin{lemma}\label{lem:FE at p}
  	We have
  \begin{equation}
   \label{eq:modified p FE}
   e_p(\Pi_p^\vee, \sigma_p^{\prime \vee}, 1 - t) = \gamma_p(\Pi_p, t) \cdot e_p(\Pi_p, \sigma_p, t).
  \end{equation}
  \end{lemma}

  \begin{proof}
   By symmetry we may suppose that $(t, \mathrm{id}) \in \Crit_p^-$ and $i = 1$, so the relation to be proved is
   \[ \gamma_p((\sigma_p')^\vee, \psi^+, 1 - t)^{-1} = \gamma_p(\Pi_p, \psi^-, t) \cdot \gamma_p(\sigma_p, \psi^-, t)^{-1}. \]
   Since $\gamma_p((\sigma_p')^\vee, \psi^+, 1 - t)^{-1} = \gamma_p(\sigma_p', \psi^-, t)$, this is equivalent to the ``inductivity'' property
   \[ 
    \gamma_p (\sigma_p, \psi^-, t) \cdot \gamma_p(\sigma_p', \psi^-, t) = \gamma_p (\Pi_p, \psi^-, t).
   \]
   As above, since $\sigma_p \times \sigma_p'$ embeds in $J_{P_i}(V)$, its associated Weil--Deligne representation has a filtration with graded pieces $\sigma_p$ and $\sigma_p'$. The result now follows from the fact that $\gamma$-factors are multiplicative in short exact sequences of Weil--Deligne representations, whether or not the sequence is split (see e.g.~\cite[Corollary 4.5(ii)]{helmmoss15}).
  \end{proof}

  \subsubsection*{Twists} For $\eta$ a Dirichlet character, we have $(t, \eta) \in \Crit^{\pm}(\Pi)$ if and only if $(t, \mathrm{id}) \in \Crit^{\pm}(\Pi \times \eta)$. We note that if $\sigma_p$ is a $P_i$-refinement of $\Pi_p$, then $\sigma_p \times \eta_p$ is a $P_i$-refinement of $\Pi_p \times \eta_p$; and this refinement is nearly-ordinary if $\sigma_p$ is.
  
  \begin{proposition}\label{prop:nearly-ordinary vs ordinary}
   To prove \cref{conj:intro}(i) for $n = 3$, it suffices to prove that it holds for all $\Pi$ that are $P_1$-ordinary at $p$ (not just $P_1$-nearly-ordinary).
  \end{proposition}
    
  \begin{proof}
   From the discussion above, if $\chi$ is any $p$-power-conductor Dirichlet character, \cref{conj:intro}(i) holds for $\Pi$ if and only if it holds for $\Pi \times \chi$ (and the measure $L_p^-(\Pi \times \chi)$ is the twist of $L_p^-(\Pi)$ by $\chi$). Given any $\Pi$ which is $P_1$-nearly-ordinary at $p$, its $P_1$-refinement $\sigma_p$ is a smooth character of $\Qp^\times$, and we can evidently find a Dirichlet character $\chi$ such that $\chi_p \sigma_p$ is unramified; thus $\Pi \times \chi$ is $P_1$-ordinary at $p$.
  \end{proof}
  
  \begin{remark}
   Note this reduction would not work for \cref{conj:intro}(ii), or for odd $n \ge 5$.
  \end{remark}
  
  \begin{example}\label{ex:unramified P1 refinement}
   If $\sigma_p$ is an unramified $P_1$-refinement of $\Pi_p$ with $\sigma_p(p)= \alpha_p \in E^\times$, then for all $(-j, \eta) \in \Crit_p^-(\Pi)$ we have
   \[ e_p\Big(\Pi_p \times \eta_p, \sigma_p \times \eta_p, -j\Big) \defeq \begin{cases}
   	G(\eta^{-1}) \cdot (p^{j+1}\alpha_p)^{-n} & \operatorname{cond}(\eta) = p^n > 1,\\[2mm]
   	\displaystyle\frac{(1-p^{-j-1}\alpha_p^{-1})}{(1 - p^j \alpha_p)} & \eta = 1.
   \end{cases}
   \]
  \end{example}

\subsection{Hecke operators}
\label{sec:U_p}

%Let $p$ be a prime, and define parahorics associated to the $P_i$ by
%\begin{equation}\label{eq:parahoric}
%\cJ_{p,i} \defeq \{g \in \GL_3(\Zp) : g \newmod{p} \in P_i(\F_p)\}.
%\end{equation}
Let $i \in \{1, 2\}$ and let $N_{P_i}$ and $M_{P_i}$ be the unipotent radical and Levi subgroup of $P_i$. If $K_p \subset \GL_3(\Z_p)$ is any open subgroup containing $N_{P_i}(\Zp)$ and having an Iwahori decomposition with respect to $P_i$, then we have a normalised Hecke operator
\begin{equation}\label{eq:tau}
U_{p, i} = p^a \left[K_p \ \tau_i \ K_p\right], \qquad 	\tau_1 =  \sdthree{p}{1}{1}, \qquad \tau_2 = \sdthree{p}{p}{1} \in \GL_3(\Qp).
\end{equation}
on $\Pif^{K_p}$. The scalar $p^a$ optimally ensures $U_{p,i}$ preserves a $\Z$-lattice in Betti cohomology (see \S\ref{sec:local systems}), so its eigenvalues are algebraic integers (cf.\ \cite[Rem.\ 3.13]{BDW20}; our $U_{p,i}$ would be $U_{p,i}^\circ$ \emph{op. cit}.).

 From Casselman's canonical lifting theorem, the subspace of $\Pi_p^{K_p}$ on which the Hecke operator $U_{p, i}$ has finite slope (the sum of its generalised eigenspaces with non-zero eigenvalues) is isomorphic to the invariants of the Jacquet module $J_{P_i}(\Pi)$ under $K_p \cap M_{P_i}$. Moreover, the action of $U_{p, i}$ on this finite-slope subspace is $p^{a+1}$ times the action of $\tau_i \in Z(M_{P_i})$ on $J_{P_i}(\Pi)$ (the $+1$ comes from the modulus character). This result originates in an unpublished note of Casselman; see \cite{EmertonJacquetI} for an account.

 Let us now consider a $P_i$-refinement $\sigma_p \times \sigma'_p$ satisfying the following conditions: $\sigma_p$ is unramified, and all irreducible constituents of $J_{P_i}(\Pi_p) / (\sigma_p \times \sigma'_p)$ have different central characters from $\sigma_p \times \sigma'_p$. Both of these conditions are is automatic if the $p$-refinement is ordinary (since the other constituents, if any, must have strictly larger slope), but is also satisfied in many non-ordinary cases.

 \begin{remark}
  This is a $\GL_3$ analogue of the ``$p$-regularity'' condition often encountered in $\GL_2$ theory; if $\Pi_p$ is principal-series and $i = 1$, a $P_1$-refinement is just a choice of one among the three characters from which $\Pi_p$ is induced, and we are assuming that our chosen character is distinct from the other two.
 \end{remark}

 Recalling $\omega_{\sigma_p}$ is the central character of $\sigma_p$ and letting $\alpha_p = \omega_{\sigma_p}(p)$, it follows that the $U_{p, i} = p^{a + 1} \alpha_p$ generalised eigenspace on $\Pi_p^{K_p}$ coincides with the actual eigenspace, and is isomorphic to the $(K_p \cap M_{P_i})$-invariants of $\sigma_p \times \sigma'_p$.

\begin{remark}
If $\Pi_p$ is unramified, then its $P_1$-refinements are the unramified characters mapping $p$ to the Satake parameters of $\Pi_p$, so our notation $\alpha_p$ is consistent with \cref{sec:L-function}.
\end{remark}

\subsection{\texorpdfstring{$P_1$}{P1}-refined new-vectors}
\label{sect:refnewforms}
We now specialise to the case of an unramified $P_1$-refinement $\sigma_p \times \sigma'_p$. As $\sigma_p$ is an unramified character, the refinement is determined by the number $\alpha_p = \sigma_p(p)$. 

For $r \ge 0$, consider the subgroup
\begin{equation}\label{eq:U_p(Pi)}
\cU_{1, p}^{(P_1)}(p^r) = \left\{ g: g = 1 \bmod \smallthreemat{\star}{\star}{\star} {p}{\star}{\star} {p^R}{p^r}{p^r}\right\} \subset \GL_3(\Zp),
\end{equation}
where $R = \max(r, 1)$. By construction this subgroup has an Iwahori decomposition with respect to $P_1$; and its intersection the $\GL_2$ factor of the Levi of $P_1$ is the level $p^r$ mirabolic subgroup appearing in Casselman's $\GL_2$ local new-vector theory. Hence we have the following:

\begin{proposition}\label{prop:P_1 conductor}
For all sufficiently large $r$, there exists a vector $\varphi_p$ in $\Pi_p$ which is invariant under $\cU_{1, p}^{(P_1)}(p^r)$, and on which $U_{p, 1}$ acts via $p^{a+1} \alpha_p$. The minimal such $r$, denoted $r(\Pi_p, \alpha_p)$, is equal to the conductor of the $\GL_2(\Qp)$-representation $\sigma'_p$; and for this $r$, the space of such vectors is 1-dimensional.
\end{proposition}

This follows from the two results of Casselman discussed above: the canonical lifting theorem, which identifies the $U_{p, 1} = p^{a+1} \alpha_p$ eigenspace with the $(\Zp^\times \times \cU_{1, p}^{(\GL_2)}(p^r))$-invariants of $\sigma_p \times \sigma'_p$, combined with the local new-vector theorem for $\GL_2$ applied to $\sigma'_p$.

We define the \emph{$P_1$-refined local Whittaker newvector} $W^{\alpha_p}$ to be the unique basis of the above 1-dimensional space in the Whittaker model of $\Pi_p$, normalised to be 1 at the identity. By comparing Hecke eigenvalues (using the Hecke-equivariance of Casselman's lifting), we have the following formula for its values along the torus:

\begin{proposition}\label{prop:torusWh}
We have $W^{\alpha_p}\left(\sdthree{p^{m+n}}{p^n}{1}\right) = 0$ if $m < 0$ or $n < 0$, and for $m, n \ge 0$ its values are given by
\[
p^{n/2}\left(\tfrac{\alpha_p}{p}\right)^{(m + n)} W_{\sigma'_p}^{\mathrm{new}} \left(\stbt{p^n}{}{}{1} \right)
\]
where $W_{\sigma'_p}^{\mathrm{new}}$ is the normalised new-vector of the $\GL_2(\Qp)$-representation $\sigma'_p$.\qed
\end{proposition}

\subsection{Ordinarity for unramified primes}

We briefly explain what the above definitions give in the (important!) special case when $\Pi_p$ is unramified. The parabolics $P_i$ correspond to two normalised Hecke operators in the spherical Hecke algebra (cf.\ \eqref{eq:tau})
\begin{equation}
\label{eq:normalised hecke 1}
T_{p, 1} = p^a \left[\GL_3(\Zp)\ \tau_1\ \GL_3(\Zp)\right], \qquad
T_{p, 2} = p^{a} \left[\GL_3(\Zp) \ \tau_2 \ \GL_3(\Zp)\right],
\end{equation}
for $\tau_i$ as in \eqref{eq:tau}. Let $a_{p, i}(\Pi)$ be the eigenvalue of $T_{p,i}$ on $\Pi_p^{\GL_3(\Zp)}$; as in \cref{sec:U_p}, these are algebraic integers. Recall the Satake parameters $\alpha_p,$ $\beta_p$ and $\gamma_p$ from \cref{sec:L-function}; we have
\[
a_{p, 1}(\Pi) = p^{a+1}(\alpha_p + \beta_p + \gamma_p),\qquad a_{p, 2}(\Pi) = p^{a+1}(\alpha_p \beta_p + \beta_p \gamma_p + \gamma_p \alpha_p).
\]
Moreover, $a_{p, 2} = \omega_{\Pi}(p) \cdot \overline{a_{p, 1}}$. From these formulae, the following is immediate:

\begin{lemma}\label{def:P_i ordinary}
$\Pi_p$ is $P_i$-ordinary (with respect to $\iota$) if the eigenvalue $a_{p, i}(\Pi)$ of $T_{p, i}$ acting on $\Pi_p$ is a $p$-adic unit. \qed
\end{lemma}

Hence, if we order the Satake parameters so that $v_p(\alpha_p) \le v_p(\beta_p) \le v_p(\gamma_p)$, then $\Pi_p$ is $P_1$-ordinary if and only if $v_p(\alpha_p) = -1-a$ (the smallest possible value), in which case the unique ordinary $P_1$-refinement is $\sigma_p \times \sigma'_p$, where $\sigma_p$ is the unramified character with $\sigma_p(p) = \alpha_p$, and $\sigma'_p$ is the unramified $\GL_2$-representation with Satake parameters $\{\beta_p, \gamma_p\}$.

On the other hand, $\Pi_p$ is $P_2$-ordinary if and only if $v_p(\gamma_p) = 1 + a$ (the largest possible), and if so its unique ordinary $P_2$-refinement is $\sigma_p \times \sigma'_p$, where $\sigma_p$ has Satake parameters $\{\alpha_p, \beta_p\}$ and $\sigma'_p$ is the unramified character $p \mapsto \gamma_p$.

\begin{remark}
Note that $P_1$-ordinarity and $P_2$-ordinarity are equivalent for essentially self-dual representations (of prime-to-$p$ level), but not for general RACARs. Explicit examples which are $P_1$-ordinary but not $P_2$-ordinary can be found in the computations of \cite{GKTV97}.%; their tables give the values $a_{p, 1}(\Pi)$ for examples of non-essentially-self-dual automorphic representations with $a = 0$, trivial character, and conductors $N = \{53, 58,$ $ 61, 79, 88, $ $153, $ $223\}$, for all unramified primes $p \le 173$. One checks that for their level 53 and 61 examples with $p = 3$, and their level 223 example with $p = 31$, there is a prime above $p$ in the coefficient field dividing $a_{p, 1}$ but not $\overline{a}_{p, 1} = a_{p, 2}$. Hence, for the embedding $\iota$ corresponding to this prime, $\Pi$ is $P_1$-ordinary but not $P_2$-ordinary; and for the conjugate embedding it is $P_2$-ordinary but not $P_1$-ordinary. (In all the other examples of $\Pi$ and $p$ covered by their data, $\Pi$ is ordinary for both parabolics.)
\end{remark}

In this unramified setting, the quantity $r(\Pi, \alpha_p)$ of \cref{prop:P_1 conductor} is 0, and the corresponding group $\cU_{1, p}^{(P_1)}(1)$ is the parahoric subgroup associated to $P_1$; so the $P_1$-refined new vector is simply the unique normalised $(U_{p, 1} = p^{a + 1} \alpha_p$)-eigenvector in the parahoric invariants of $\Pi_p$.

%%%%%%%%%%%%%%%%%%%%%%%%%%%%%%%%%%%%%%%%%%%%%%%%
%             Symmetric spaces, Level structures
%%%%%%%%%%%%%%%%%%%%%%%%%%%%%%%%%%%%%%%%%%%%%%%%

\section{Symmetric spaces and Betti cohomology}

\subsection{Symmetric spaces}\label{sec:symmetric spaces}

For any split reductive group $J$ over $\Q$, we define a symmetric space for $J$ by $\uhp_J = J(\R) / K_{J, \infty}^{\circ} Z_{J, \infty}^\circ,$ where $(-)^\circ$ denotes the identity component, $K_{J,\infty}$ is a maximal compact subgroup of $J(\R)$, and $Z_\infty = Z_J(\R)$. For a neat open compact subgroup $\cU \subset J(\A\subf)$, we define
\[
Y^J(\cU) = J(\Q) \backslash \left[ J(\A\subf)/\cU \times \uhp_{J} \right].
\]

Note that for $J = \GL_n$, the components of $Y^J(\cU)$ are indexed by the double quotient
\[
Y^{\GL_1}(\det \cU) = \Q^{\times} \backslash \A^\times / \det(\cU) \R_{>0}^\times \cong \Zhat^\times / \det(\cU).
\]
Each component of $Y^J(\cU)$ is the quotient of $\uhp_J^\circ$ by an arithmetic subgroup of $\SL_n(\Q)$.

Note $\uhp_{\GL_2}$ is the union of two copies of the complex upper half-plane, and has dimension 2 as a real manifold; while $\uhp_{\GL_3}$ is a 5-dimensional real manifold, again with two components.

\begin{remark}
Our conventions match the usual ones for modular curves when $J = \GL(2)$. However, note that some papers (such as \cite{HT17}) use the slightly different quotient $J(\Q) \backslash J(\A) / U K_\infty Z_\infty^\circ$. For example, if $J = \GL(2)$ and $U = \{ \stbt {*}{*}{0}{1} \bmod N\}$, our $Y^J(U)$ is the usual modular curve $Y_1(N)$, while the space considered in \cite{HT17} is the quotient of $Y_1(N)$ by an anti-holomorphic involution.
\end{remark}

\subsection{Betti cohomology and Hecke operators}\label{sec:local systems}

Let $J, \cU$ be as above. Given an algebraic representation $V$ of $J$ (over $\Q$), we have three possible constructions of local systems on $Y^{J}(\cU)$:

\begin{itemize}

\item a local system $\sV_{\Q}$ of $\Q$-vector spaces, given by the locally constant sections of the projection
\[
J(\Q) \backslash\big[\big(J(\A\subf) \times \uhp_{J}\big) \times  V(\Q)\big]/\cU \longrightarrow Y^J(\cU),
\]
with action $\gamma \cdot [(g,z),v] \cdot u = [(\gamma gu, \gamma z), \gamma \cdot v]$. (The functor $V \mapsto \sV_{\Q}$ is Pink's ``canonical construction'' functor, \cite{pink90}).

\item a local system $\sV_{\infty}$ of $\R$-vector spaces, given by the locally constant sections of
\[ J(\Q) \backslash\big[\big(J(\A\subf) \times J(\R)\big) \times  V(\R)\big]/\cU K_{J,\infty}^{\circ} Z_{J,\infty}^{\circ} \to Y^J(\cU),\]
with action $\gamma \cdot [(g,z),v] \cdot u k_\infty z_{\infty} = [(\gamma gu, \gamma z), (k_\infty z_{\infty})^{-1} \cdot v]$.

\item a local system of $\Qp$-vector spaces $\sV_p$ for any finite prime $p$, given by the locally constant sections of
\begin{equation}\label{eq:p-adic local system}
	J(\Q) \backslash\big[\big(J(\A\subf) \times \uhp_{J}\big) \times V(\Qp)\big]/\cU \longrightarrow Y^{J}(\cU),
\end{equation}
with action $\gamma[(g,z),v]u = [(\gamma gu, \gamma z), u_p^{-1} \cdot v]$.
\end{itemize}

If $\cU' \subset \cU$, then the formation of $\sV_?$ is compatible with the natural projection map $Y^J(\cU') \to Y^J(\cU)$, and we get natural maps $\h^\bullet(Y^J(\cU), \sV_?) \to \h^\bullet(Y^J(\cU'), \sV_?)$, where $\h^\bullet$ denotes Betti cohomology.  We write
\[
\h^\bullet(Y^J, \sV_{?}) = \varinjlim_{\cU} \h^\bullet(Y^J(\cU), \sV_{?})
\]
where $\cU$ varies over open compact subgroups of $J(\A\subf)$. This direct limit is naturally a representation of $J(\A\subf)$, and $\h^\bullet(Y^J(\cU),\sV_?) = \h^\bullet(Y^J,\sV_?)^{\cU}.$  We have direct analogues of these statements for compactly-supported cohomology $\h^\bullet_c$.

It is standard that there are canonical isomorphisms of local systems
\begin{equation}\label{eq:local system isos}
\sV_{\Q} \otimes \Qp \cong \sV_p,  \qquad \sV_{\Q} \otimes \R \cong \sV_{\infty}.
\end{equation}
For instance, the first of these two is given on sections by $[(g,z), v] \mapsto [(g,z), g_p^{-1}\cdot v]$. Details on all of this can be found in \cite[\S1.2]{Urb11}.

The Betti cohomology of $Y^J(\cU)$ is equipped with a natural action of the Hecke algebra $\C[\cU \backslash \GL_3(\A\subf) / \cU]$. The isomorphisms on cohomology groups induced by the isomorphisms \eqref{eq:local system isos} of local systems are all equivariant under the Hecke operators \cite[\S1.2]{Urb11}.

\label{sec:hecke}

For later use, we describe the $P_1$-Hecke operator at $p$ in more detail. Let $\cU = \cU^{(p)}\cU_p \subset \GL_3(\A\subf)$ be an open compact subgroup, with $\cU_p$ admitting an Iwahori decomposition relative to $P_1$, with notation as in \cref{sec:U_p}.  Recall $\tau_i$ from \eqref{eq:tau}; for ease of notation, we set $\tau \defeq \tau_1$. Then we have maps
\begin{align*}
Y^{\GL_3}(\cU)\xleftarrow{\ \mathrm{pr}_{\cU}^{\cU \cap \tau \cU\tau^{-1}}\ }	Y^{\GL_3}(\cU \cap \tau \cU\tau^{-1})
&\xrightarrow{\ \ \tau\ \ }
Y^{\GL_3}(\tau^{-1} \cU \tau \cap \cU)\\
&\xrightarrow{\ \mathrm{pr}_{\cU}^{\tau^{-1} \cU \tau \cap \cU}\ } Y^{\GL_3}(\cU),
\end{align*}
where the middle map is induced by right-translation of $\tau$ on $\GL_3(\A\subf)$ and the outside maps are the natural projection maps. Passing from left-to-right and right-to-left respectively, we get associated (normalised) Hecke operators
\begin{align*}
U_{p,1}' &\defeq p^a \cdot \left(\mathrm{pr}_{\cU}^{\cU \cap \tau \cU\tau^{-1}}\right)^* \circ \tau_* \circ \left(\mathrm{pr}_{\cU}^{\tau^{-1} \cU \tau \cap \cU}\right)_*,\\
U_{p,1} &\defeq p^a \cdot \left(\mathrm{pr}_{\cU}^{\tau^{-1} \cU \tau \cap \cU}\right)^* \circ \tau^* \circ \left(\mathrm{pr}_{\cU}^{\cU \cap \tau \cU\tau^{-1}}\right)_*,
\end{align*}
on the cohomology $\h^\bullet_*(Y^{\GL_3}(\cU),\sV_?)$, for $* = \varnothing$ or $\mathrm{c}$ and $\sV_?$ as in \cref{sec:local systems}. As in \eqref{eq:normalised hecke 1}, the scalars $p^a$ are for integral normalisation. The operators $U_{p,1}$ and $U_{p,1}'$ are adjoint to each other under Poincar\'e duality, as we make precise in \S\ref{sec:cup product pairing}.

We also define \emph{ordinary projectors} $e_{\mathrm{ord},1} \defeq \lim_{n\to \infty} U_{p,1}^{n!}$ and $e_{\mathrm{ord},1}' \defeq \lim_{n \to \infty}(U_{p,1}')^{n!}$; by definition, $U_{p,1}$ is invertible on the image of $e_{\mathrm{ord},1}$.

\subsection{Automorphic cohomology classes}\label{sec:auto cohom classes}

Let $\Pi$ be a RACAR of $\GL_3(\A)$. We now realise $\Pi$ in the compactly-supported Betti cohomology of $Y^{\GL_3}$. Let $\cU \subset \GL_3$ be any neat open compact subgroup such that $\Pif^{\cU} \ne 0$, and write $\sV_{\lambda,\Q}^\vee$ for the $\Q$-local system attached to $V_\lambda^\vee$. Via the cuspidal cohomology, and the natural map $\h^\bullet_{\mathrm{cusp}} \hookrightarrow \h^\bullet_{\mathrm{c}}$ of \cite[p.123]{Clo90}, there is an injection
\[
\h^\bullet\big(\mathfrak{gl}_3,K_{3,\infty}^\circ; \Pi_\infty\otimes V_{\lambda}^{\vee}(\C) \big)\otimes \cW_\psi(\Pif)^{\cU} \longhookrightarrow \h^\bullet_{\mathrm{c}}\left(Y^{\mathrm{GL}_3}(\cU), \sV_{\lambda,\Q}^\vee(\C)\right),
\]
compatible with the Hecke action of $\C[\cU \backslash \GL_3(\A\subf) / \cU]$. Via the choice of $\zeta_\infty$ from \eqref{eq:omega infinity}, this yields a map
\begin{equation}\label{eq:pi to cohom}
\phi_\Pi^{\cU}: \mathcal{W}(\Pif)^{\cU}\longhookrightarrow \hct^2\left(Y^{\mathrm{GL}_3}(\cU), \sV_{\lambda,\Q}^\vee(\C)\right).
\end{equation}
The $\Q$-local systems, and \eqref{eq:pi to cohom}, are compatible with varying $\cU$, giving a $\GL_3(\A\subf)$-equivariant map
\begin{equation}\label{eq:phi_Pi}
\phi_\Pi : \mathcal{W}(\Pif) \longhookrightarrow \hct^2\left(Y^{\mathrm{GL}_3}, \sV_{\lambda,\Q}^\vee(\C)\right) = \varinjlim_{\cU} \hct^2\left(Y^{\mathrm{GL}_3}(\cU), \sV_{\lambda,\Q}^\vee(\C)\right).
\end{equation}
We shall denote the image of this map by $\hct^2(\Pi, \C)$.

\subsubsection{Periods and rationality}\label{sec:periods}

The representation $\hct^2(Y^{\mathrm{GL}_3}, \sV_{\lambda,\Q}^\vee(\C))$ has a natural $\Q$-structure given by the cohomology of $\sV_{\lambda,\Q}^\vee(\Q)$. Since the complex representation is admissible and contains $\Pif$ with multiplicity 1, it follows that there is a number field $E$, the \emph{field of definition} of $\Pi$, such that the $E$-linear representation
\[ \hct^2(\Pi, E) = \hct^2(\Pi, \C) \cap \hct^2\left(Y^{\mathrm{GL}_3}, \sV_{\lambda,\Q}^\vee(E)\right) \]
is non-zero and gives an $E$-structure on $\hct^2(\Pi,\C)$. (This is the same field $E$ as from \cref{sec:L-function}. By the strong multiplicity one theorem for $\GL_3$, we may take $E$ to be the field generated by the Hecke eigenvalues of $\Pi$ at the unramified primes, although we shall not need this.)

From the uniqueness of Whittaker models, $\mathcal{W}(\Pif)$ also has a canonical $E$-structure $\mathcal{W}(\Pif, E)$, given by the functions $W \in \mathcal{W}(\Pif)$ which take values in $E\Q^{\mathrm{ab}}$ and satisfy
\[ W(g)^{\sigma} = W\left(\sdthree{\kappa(\sigma)^2}{\kappa(\sigma)}{1} g\right) \quad\text{for all $\sigma \in \operatorname{Gal}(E \Q^{\mathrm{ab}} / E)$},\]
where $\kappa: \operatorname{Gal}(\Q^{\mathrm{ab}} / \Q) \to \widehat{\Z}^\times$ is the cyclotomic character. By \cite[Prop.\ 3.1]{Clo90}, there exists $\Theta_\Pi \in \C^\times$ such that
\[
\phi_\Pi\left(\mathcal{W}(\Pif, E)\right) = \Theta_\Pi \cdot \hct^2(\Pi, E).
\]

\begin{remark}
Having fixed $\zeta_\infty$, the period $\Theta_\Pi$ is uniquely determined modulo $E^\times$. We may, however, rescale $\zeta_\infty$ by an arbitrary non-zero complex scalar, which then also rescales $\Theta_\Pi$.

Note that for $\GL_n$ with $n$ odd, we obtain only a single Whittaker period, since \eqref{eq:infinite cohomology} is 1-dimensional. This differs from the case of $\GL_n$ for $n$ even, where \eqref{eq:infinite cohomology} is 2-dimensional and there are two Whittaker periods (one for each choice of sign at $\infty$).
\end{remark}

\subsubsection{Integral lattices}

Let $V_\mu^J$ be an algebraic representation of a split reductive group $J$, with highest weight vector $v_\mu \in V_\mu^J$. An \emph{admissible lattice} is a lattice $\cL \subset V_\mu$ such that: (1) the map $J_{/\Q} \to \GL(V_\mu)$ extends to a map of $\Z$-group schemes $J_{/\Z} \to \GL(\cL)$, and (2) the intersection of the highest weight space in $V_\mu$ with $\cL$ is $\Z\cdot v_\lambda$.

The set of admissible lattices in $V_\mu^J$ (for a fixed choice of $v_\mu$) is finite (cf.~\cite[\S 4.2]{LSZ17}), and there are uniquely-determined maximal and minimal admissible lattices. Since the quotient of these two lattices is a finite abelian group, the base-extensions to $\Z_p$ of these two lattices coincide for all sufficiently large $p$ (depending on $\mu$); for instance, if $J = \SL_2$, then the minimal and maximal lattices in the $k$-th symmetric power of the standard representation coincide over $\Zp$ for all $p > k$.

We adopt the convenient (but somewhat misleading) notation that $V_{\mu,\Z}^{\GL_2}$ denotes the \emph{minimal} such lattice in $V_\mu^{\GL_2}$, whilst $V_{\lambda,\Z}^{\GL_3}$ denotes the \emph{maximal} such lattice in $V_\lambda^{\GL_3}$. This ensures that under the branching laws of the next section, $V_{\star,\Z}^{\GL_2}$ is always mapped into $V_{\star,\Z}^{\GL_3}$.

\subsubsection{Integral structures on cohomology}\label{sec:integral periods}

Now let $p$ be a prime, and $v \mid p$ a prime of $E$. Completing $E$ at $v$, we may consider
\[
\phi_{\Pi}(\cW(\Pi_{\mathrm{f}},E)) /\Theta_\Pi \subset \hct^2(Y^{\GL_3},\sV_{\lambda,\Q}^\vee(E_v)).
\]

Note as above that $\sV_{\lambda,\Q}^\vee(E_v) \cong \sV_{\lambda,p}^\vee(E_v)$ on $Y^{\GL_3}(\cU)$. We can define an integral version using the lattice $V_{\lambda,\Z}$ above (or, more properly, the analogue $V_{\lambda,\Z}^\vee$, which can be described similarly); the $\cO_{E,v}$-points $V_{\lambda,\Z}(\cO_{E,v})$ of this representation carry an action of $\GL_3(\Zp)$, so we get an associated local system  $\sV_{\lambda,p}^\vee(\cO_{E,v})$ of $\Zp$-modules (defined exactly as in \eqref{eq:p-adic local system}), and where we henceforth drop the $\Z$ for convenience.

 Then, for any level $\cU$, the space $\hct^2\left(Y^{\mathrm{GL}_3}(\cU), \sV_{\lambda,p}^\vee(\mathcal{O}_{E, v})\right)$ is a finitely-generated $\mathcal{O}_{E, v}$-module; and the quotient of this module by its torsion subgroup is an $\cO_{E, v}$-lattice in $\hct^2\left(Y^{\mathrm{GL}_3}(\cU), \sV_{\lambda,p}^\vee(E_v)\right)$. 

For any finite subset $\{W_i\} \subset \cW(\Pi_{\mathrm{f}},E)$, any open compact $\mathcal{U}$ stabilizing all the $W_i$, and any finite extension $L / \Qp$ into which $E$ embeds, we have classes
\begin{equation}
	\phi_\Pi(W_i)/\Theta_\Pi \in \hc{2}(Y^{\GL_3}(\cU),\sV_\lambda^\vee(L)),
\end{equation}
depending on the choices of $\zeta_\infty$ (determined up to $\C^\times$) and the period $\Theta_\Pi$ (which, given the choice of $\zeta_\infty$, is determined up to $E^\times$). Via the embedding $\iota$, we may further rescale $\Theta_\Pi$ to ensure all the $\phi_\Pi(W_i)/\Theta_\Pi$ take values in the lattice $\hc{2}(Y^{\GL_3}(\cU),\sV_\lambda^\vee(\cO_L)) / \mathrm{\{torsion\}}$. Such a normalisation depends on $\{W_i\}$.

\begin{remark}\label{rem:L}
	Henceforth we fix a choice of a finite extension $L / \Qp$ and an embedding $\iota : E \hookrightarrow L$ (factoring through $E_v$ for some prime $v \mid p$), and only use the local systems $\sV_{\lambda,p}^\vee(L)$ or $\sV_{\lambda,p}^\vee(\cO_{L})$ constructed from the action of $\GL_3(\Zp)$. To ease notation we will typically drop the subscript $p$ and write $\sV_\lambda^\vee$.
\end{remark}

\subsection{The cup product pairing}\label{sec:cup product pairing}
Let $\cU = \cU^{(p)}\cU_p\subset \GL_3(\A\subf)$ be a neat open compact subgroup and let $R$ be a $\Zp$-algebra. There is a (perfect) Poincar\'e duality pairing
\begin{equation}\label{eq:poincare}
\langle -,-\rangle_{\cU}: \frac{\hc{2}\big(Y^{\GL_3}(\cU), \sV_\lambda^\vee(R)\big)}{\mathrm{(torsion)}}  \times \frac{\h^3(Y^{\GL_3}(\cU),\sV_\lambda(R))}{\mathrm{(torsion)}} \longrightarrow R
\end{equation}
given by composing cup product, the natural pairing $V_\lambda^\vee(R) \otimes V_\lambda (R) \to R$, and integration over the smooth (5-dimensional) real manifold $Y^{\GL_3}(\cU)$. If $N_{P_1}(\Zp) \subset \cU_p$ and $\cU_p$ has an Iwahori decomposition with respect to $P_1$, then from \cref{sec:hecke} we have Hecke operators $U_{p,1}$ and $U_{p,1}'$ acting on $\h^\bullet_*(Y^{\GL_3}(\cU),\sV_\lambda(R))$, where $* \in \{\varnothing, \mathrm{c}\}$. Note that under Poincar\'e duality, pullback is adjoint to pushforward; thus the operator $U_{p,1}$ (acting on either factor) is adjoint to $U_{p,1}'$ (acting on the other factor) under $\langle-,-\rangle_{\cU}$.

\section{The subgroup \texorpdfstring{$H$}{H}}
\label{sect:groups}

\begin{definition}
Let $H = \GLt\times\GL_1$, and let $\iota: H \hookrightarrow \GL_3$ be the embedding
\begin{equation}\label{eq:iota}
	\left[\smallmatrd{a}{b}{c}{d},z\right] \longmapsto \smallthreemat{a}{b}{}{c}{d}{}{}{}{z}.
\end{equation}
\end{definition}

\begin{definition}\label{def:nu}
Let $\nu_1, \nu_2: H \to \GL_1$ be the homomorphisms given by
\[ \nu_1(\gamma, z) = \tfrac{\det \gamma}{z}\qquad \nu_2(\gamma, z) = z. \]
\end{definition}

Note $(\nu_1, \nu_2)$ gives an isomorphism $H / H^{\mathrm{der}} \cong \GL_1 \times \GL_1$ (this parametrisation of $H / H^{\mathrm{der}}$ may seem slightly unnatural, but will give us nicer formulae later), and $\det \circ \mathop{\iota} = \nu_1\nu_2^2$.

\subsection{Symmetric spaces}\label{sec:symmetric space H}

It is important to note that the embedding $\iota: H \to \GL_3$ does not induce a map on symmetric spaces, since the inclusion $Z_{\GL_3} \subset \iota(Z_H)$ is strict. To transfer cohomology classes from $Y^H$ to $Y^{\GL_3}$, we instead define
\[
\widetilde{\uhp}_H = H(\R) / \left(K_{H, \infty}^{\circ} \cdot \iota^{-1}(Z_{G, \infty}^\circ)\right),
\]
which maps naturally to both $\uhp_H$ and $\uhp_G$. If $\cU \subset H(\A\subf)$ is open compact, let
\begin{equation}
\label{eq:tilde Y}
\widetilde{Y}^H(\cU) \defeq H(\Q)\backslash\left[H(\A\subf)/\cU \times \widetilde{\cH}^H\right].
\end{equation}
For any open compact $\cV \subset \GL_3(\A\subf)$, we then have a diagram of maps
\begin{equation}\label{eq:tilde maps}
Y^H(\cV \cap H) \leftarrow \widetilde{Y}^H(\cV \cap H) \xrightarrow{\ \iota\ } Y^G(\cV),
\end{equation}
where the right-hand map is induced by $\iota$. Pulling back under the leftward arrow and pushing forward under the rightward arrow gives a map from the cohomology $Y^H$ to that of $Y^{\GL_3}$.

If $\cV$ is small enough, the left arrow is a fibre bundle with fibres isomorphic to $Z_{H, \infty}^\circ / \iota^{-1}(Z_{G, \infty}^\circ) \cong \R$. The spaces $Y^H$ (resp.\ $\widetilde{Y}^H$) have dimension 2 (resp.\ 3) as real manifolds.

We get local systems on the modified spaces $\widetilde{Y}^H(\cU)$ from \eqref{eq:tilde Y}, defined identically to \cref{sec:local systems}.

\subsection{Branching laws}\label{sec:branching laws}

Note $V^H_{(r, s; t)} = V_{(r,s)}^{\GL_2} \otimes V_{(t)}^{\GL_1}$ is the $H$-representation of highest weight $(\stbt x {}{}y, z) \mapsto x^r \cdot  y^s \cdot z^t$. We let $V_{(r,s;t),\Z}^H$ denote the \emph{minimal} admissible lattice in $V_{(r,s;t)}^H$. The following is equivalent to the well-known branching law from $\GL_3$ to $\GL_2$, e.g.~\cite[\S8]{goodmanwallach}.

\begin{proposition}\label{prop:branching}
Let $\lambda = (a, 0, -a)$. The restriction of $V_\lambda$ to $H$ (embedded via $\iota$) is given by
\[
\iota^*\left(V_\lambda\right) \cong \bigoplus_{0 \le i,j \le a} V^H_{(j, -i; i-j)}.
\]
\end{proposition}

For such $\lambda$, we shall fix choices of non-zero morphisms of $H$-representations over $\Q$
\[ \operatorname{br}^{[a,j]}: V^H_{(j, 0; -j)} \to \iota^*\left(V_{\lambda}\right)\]
for each $j$ as above. Note $V^H_{(j, 0; -j)} = V^H_{(0, -j; 0)} \otimes \|\nu_1^j\|$, hence our choice yields a pairing
\begin{equation}\label{eq:branching pairing}
\langle -,-\rangle_{a,j} : V_{\lambda}^\vee \times V_{(0,-j)}^{\GL_2} \to \mathbb{G}_a,\hspace{12pt} (\mu,v) \mapsto \mu\big( \mathrm{br}^{[a,j]}(\|\nu_1^j\| \otimes [v \otimes 1])\big).
\end{equation}

As in \cite[Prop 4.3.5]{LSZ17}, $\mathrm{br}^{[a,j]}$ maps the (minimal) admissible lattice $V^H_{(j,0;-j),\Z}$ into the (maximal) admissible lattice $V_{\lambda,\Z}$. Thus the pairing $\langle-,-\rangle_{a,j}$ also makes sense integrally.

%%%%%%%%%%%%%%%%%%%%%%%%%%%%%%%%%%%%%%%%%%%%%%%%%%%%%%%%%%%%%%%%%%%%
%
%                     Eis classes
%
%%%%%%%%%%%%%%%%%%%%%%%%%%%%%%%%%%%%%%%%%%%%%%%%%%%%%%%%%%%%%%%%%%%%

\section{Eisenstein series and classes for \texorpdfstring{$\GLt$}{GL2}}

We recall the theory of Eisenstein series and classes attached to adelic Schwartz functions. In particular, we recall the motivic Eisenstein classes of Be\u{\i}linson (see e.g.\ \cite[\S7]{LSZ17}) and describe their Betti realisations via adelic Eisenstein series. In this section all symmetric spaces will be for $\GL_2$, so we write simply $Y(\cU)$ (resp.\ $Y$) for $Y^{\GL_2}(\cU)$ (resp.\ $Y^{\GL_2}$).

If $j\geq 0$, recall $V_{(0,-j)}^{\GL_2}$ denotes the $\GL_2$-representation of highest weight $(0, -j)$. Similarly, in this section only we will drop the superscript and denote this simply $V_{(0,-j)}$.

\subsection{Schwartz functions}\label{sec:schwartz}

For a field $K/\Q$, write $\cS(\A\subf^2,K)$ for the Schwartz space of locally-constant, compactly-supported functions on $\A\subf^2$ with values in $K$, and $\cS_0(\A\subf^2, K)$ for the subspace of $\Phi$ with $\Phi(0, 0) = 0$. We also let $\cS(\R^2,K)$ be the usual space of Schwartz functions on $\R^2$, write $\cS_0(\R^2,K)$ for the subspace with $\Phi(0,0) = 0$, and let $\cS(\A^2,K) = \cS(\A\subf^2,K) \times \cS(\R^2,K)$ (and similarly for $\cS_0(\A^2,K)$).  We will make specific choices at infinity, depending on an integer $j\geq 0$, and use the notation $\cS_{(0)}(\A\subf^2, K)$ to mean ``$\cS_{0}(\A\subf^2, K)$ if $j = 0$ or $\cS(\A\subf^2, K)$ otherwise''.

Let $\chi$ be a Dirichlet character, corresponding to a finite-order Hecke character $\chihat$. If $\Phif \in \cS(\A\subf^2,K)$, let $R_\chi(\Phif) \in \cS(\A\subf^2,K(\chi))$ be its projection to the $\chihat^{-1}$-isotypical component, given by
\begin{equation}\label{eq:R_chi}
R_{\chi}(\Phif)(x,y) \defeq \int_{a \in \widehat{\Z}^\times} \chihat(a)\  \Phif(ax,ay)\  \mathrm{d}^\times a.
\end{equation}
We emphasise that this is a projection operator, \emph{not} a twisting operator: if $\chi_1,\chi_2$ are distinct primitive Dirichlet characters, then $R_{\chi_1} R_{\chi_2} (\Phif) = 0$.

\subsection{Eisenstein series}\label{sec:eisenstein series}
We now review some standard definitions of Eisenstein series from Schwartz functions, both adelically and classically, and relate these definitions.

\subsubsection{Adelic Eisenstein series} \label{sec:Eis-GL2}
Let $\Phi = \Phi_\infty\cdot \Phif \in \cS(\A^2,\C)$ and $\chihat : \Q^\times\backslash \A^\times \to \C^\times$ be a Hecke character. For $g \in \GL_2(\A)$ and $\Re(s) > 0$ we define a \emph{Godement--Siegel section}
\begin{equation}\label{eq:f_Phi}
f_{\Phi}(g; \chihat, s) \defeq \|\det g\|^s \int_{\A^\times} \Phi\Big((0, a)g\Big) \chihat(a) \|a\|^{2s}\, \mathrm{d}^\times a.
\end{equation}
This admits meromorphic continuation to $\C$ and gives an element of the family of $\GL_2$-principal series representations $I(\|\cdot\|^{s-1/2}, \chihat^{-1}\|\cdot\|^{-s+1/2})$, that is, we have \[f_\Phi\left(\smallmatrd{a}{b}{}{d}g; \chihat,s\right) = \chihat^{-1}(d)\|a/d\|^{s}f_\Phi(g;\chihat, s).\]
It thus defines an element of $\cC^\infty(B_2(\Q)\backslash\GL_2(\A),\C)$, for $B_2 \subset \GL_2$ the upper-triangular Borel. Define (cf.~\cite[\S 19]{jacquet72})
\begin{equation}\label{eq:E_Phi}
E_\Phi(g; \chihat, s) \defeq  \sum_{\gamma \in B_2(\Q) \backslash \GL_2(\Q)} f_\Phi(\gamma g; \chihat, s),
\end{equation}
which converges absolutely and locally uniformly on some right half-plane (for $\Re(s) > 1$ if $\chihat$ is unitary), defining a function on the quotient $\GL_2(\Q) \backslash \GL_2(\A)$ that transforms under the centre by $\chihat^{-1}$. It has meromorphic continuation in $s$, analytic if $\Phi(0, 0) = \hat\Phi(0,0) = 0$ or if $\chihat$ is ramified at some finite place.

\begin{comment}
\begin{remark}\label{rem:equivariance}
	Note that $g \cdot E_{\Phi}(-; \chihat, s) = \|\det g\|^s E_{g \cdot \Phi}(-; \chihat, s)$. Note also that if $\chihat_r = \|\cdot\|^r \chihat$, then $E_{\Phi}(g; \chihat_r, s) = \|\det g\|^{-r/2}E_{\Phi}(g; \chihat, s +
	\tfrac{r}{2})$. Identical statements hold for the $f_\Phi$.
\end{remark}
\end{comment}

\begin{comment}
\begin{remark}\label{rem:finite eisenstein}
	For any place $v$, the analogue of \eqref{eq:f_Phi} over $\Q_v^2$ gives a function $f_{\Phi_v}(g;\chihat_v,s)$ with $g \in \GL_2(\Q_v)$, and then (as in \eqref{eq:f j chi} and \eqref{eq:E j chi}) an element $f_{\Phi_v}^{j+2,\chi} \in I_{j}(\chihat_v)$. Similarly we can define $f_{\Phif}(g;\chi\subf,s)$ and $f_{\Phif}^{j+2,\chi}$. However, the definition of $E_\Phi^{j+2,\chi}$ from $f_\Phi^{j+2,\chi}$ only makes sense globally.
\end{remark}
\end{comment}

\subsubsection{Classical Eisenstein series} \label{sec:classical eisenstein}We now introduce classical Eisenstein series.

\begin{definition}\label{def:classical eisenstein}
\begin{itemize}\setlength{\itemsep}{0pt}
	\item[(i)] If $\Phif \in \cS_{(0)}(\A\subf^2,K)$ and $j \geq 0$, define
	\[
	\cE^{j+2}_\Phif(\tau;s) \defeq \frac{\Gamma(s+\tfrac{j+2}{2})}{(-2\pi i)^{j+2}\pi^{s-\tfrac{j+2}{2}}} \sum_{(m,n) \in \Q^2 \backslash (0,0)} \frac{\Phif(m,n) y^{s-\tfrac{j+2}{2}}}{(m\tau + n)^{j+2}|m\tau+n|^{2s-j-2}},
	\]
	where $s \in \C$ and $\tau \in \cH_{\GL_2}$. This is a classical real-analytic weight $j+2$ Eisenstein series.

	\item[(ii)] We extend this to a function on $\GL_2(\A_{\mathrm{f}}) \times \cH$ by setting
	\[
	\cE_{\Phif}^{j+2}(g\subf, \tau; s) \defeq \cE^{j+2}_{g\subf\cdot \Phif}(\tau;s).
	\]

	\item[(iii)] Finally, for a Dirichlet character $\chi$ we also define
	\[
	\cE_{\Phif}^{j+2}(\tau; \chi, s) \defeq \cE_{R_{\chi}(\Phif)}^{j+2}(\tau;s), \qquad \cE_{\Phif}^{j+2}(g\subf, \tau; \chi, s) \defeq \cE_{R_{\chi}(\Phif)}^{j+2}(g\subf, \tau;s).
	\]

\end{itemize}
We write $\cE^{j+2}_{\Phif}$ for each of these functions, but it will be clear from context which we use.
\end{definition}

\begin{remarks}\label{rem:eisenstein functional equation}
The function defined in (i) is denoted $E^{(j+2,\Phif)}(\tau;s)$ in \cite[Def.\ 7.1]{LPSZ19}. These series always converge absolutely for $\Re(s) \geq 1$ unless $j=0$ and $\chi$ is trivial, and even in this case they converge absolutely if $\widehat\Phi(0, 0) = 0$.

One has the functional equation $\cE_{\Phif}^{j+2}(\tau;s) = \cE_{\widehat{\Phi}\subf}^{j+2}(\tau;1-s)$, where $\widehat{\Phi}\subf$ is the Fourier transform (normalised as in \cite[\S8.1]{LPSZ19}); this explains the compatibility between our conventions and those of \cite[\S7]{LSZ17}.
The definition in (ii) ensures the association $\Phif \mapsto \cE_{\Phif}^{j+2}$ equivariant under $\GL_2(\A\subf)$.
\end{remarks}

Via the standard procedure (see \cite[\S I]{Wei71}), we now extend $\cE_{\Phif}^{j+2}$ to a function on $g = g\subf g_\infty \in \GL_2(\A)$. By Iwasawa decomposition, $g_\infty \in \GL_2(\R)$ can be written uniquely in the form $\smallmatrd{z}{}{}{z}\smallmatrd{y}{x}{0}{1}r(\theta),$ where $z \in \R^\times$ and $r(\theta) = \smallmatrd{\cos \theta}{\sin \theta}{-\sin \theta}{\cos \theta} \in \mathrm{SO}_2(\R)$. Then we define
\begin{align}\label{eq:classical on GL2(A)}
\cE_{\Phif}^{j+2}&(-;\chi,-) : \GL_2(\A) \times \C \longrightarrow \C,\\
(g, s) &\longmapsto \chihat_\infty^{-1}(z) \cdot |\det(g_\infty)|^{-s}\cdot y^{\tfrac{j+2}{2}} \cdot \mathrm{exp}[{i(j+2)\theta}] \cdot \cE_{\Phif}^{j+2}\left(g\subf, x+iy; \chi, s\right).\notag
\end{align}
This yields an automorphic form on $\GL_2(\A)$: the $y^{\tfrac{j+2}{2}}$ and $e^{i(j+2)\theta}$ are always present in extending from $\uhp_{\GL_2}$ to $\GL_2(\R)$ \cite{Wei71}, the $\chihat_\infty^{-1}(z)$ ensures the central character is correct, and $|\det(g_\infty)|^{-s}$ ensures $\cE_{\Phif}^{j+2}$ factors through $\GL_2(\Q)\backslash \GL_2(\A)$.

\subsubsection{Comparison of classical and adelic}\label{sec:comparison}
At infinity, for $j \geq 0$ we shall henceforth take
\begin{equation}\label{eq:schwartz infinity}
\Phi_\infty = \Phi_\infty^{j+2}(x, y) \defeq 2^{-1-j} (x + iy)^{j+2} \exp(-\pi(x^2 + y^2)).
\end{equation}
For this choice, by \cite[Prop.\ 10.1]{LPSZ19} we have:

\begin{proposition}\label{prop:compatibility}
If $\Phi = \Phi_\infty^{j+2} \cdot \Phif \in \cS(\A^2, \C)$, for some $j \ge 0$ and $\Phi_{\mathrm{f}} \in \cS(\A\subf^2, \C)$, and $\chi$ is a Dirichlet character, then the adelic and classical Eisenstein series are related by
\[
E_{\Phi}\Big(g\subf \stbt y x 0 1; \chihat, s\Big) = y^{\tfrac{j+2}{2}} \|\det g\subf\|^{s} \cdot \cE^{j+2}_{\Phi_{\mathrm{f}}}\Big(g\subf, x + iy; \chi, s\Big)
\]
for $g\subf \in \GL_2(\A\subf)$ and $x + iy \in \cH_{\GL_2}$.
\end{proposition}

%\begin{remark}
%	A sanity check for verifying the exponent of $\det(g\subf)$ is computing what happens to both sides on replacing $g\subf$ by $\smallmatrd{b}{}{}{b}g\subf$ for $b\in \Q_{>0}$. We see $E_{\Phi}$ is fixed, and a simple computation shows $\cE_{g\subf\cdot \Phi}^{j+2}$ translates like $\|\det \smallmatrd{b}{}{}{b}\subf\|^{-s} = b^{2s}$.
%\end{remark}

\begin{corollary}\label{cor:compatibility}
When $\Phi_\infty = \Phi_\infty^{j+2}$, for $g \in \GL_2(\A)$ we have
\[
E_\Phi(g; \chihat, s) = \|\det(g)\|^s \cE_{\Phif}^{j+2}(g;\chi, s).
\]
\end{corollary}

\begin{proof}
We have $E_\Phi(g  r(\theta), \chihat, s) = e^{i(j+2)\theta}E_{\Phi}(g,\chihat, s)$ (cf.\ \cite[\S3,(7.35)]{Bump}); and $E_{\Phi}$ left-translates as $\chihat^{-1}_\infty(z)$ under $\smallmatrd{z}{}{}{z} \in \GL_2(\R)$. The result follows from \cref{prop:compatibility}.
\end{proof}

\subsubsection{The special value $s = -j/2$}\label{sec:special value j/2}
We shall be chiefly interested in the special value $s = -j/2$ (cf.\ \cref{thm:motivic classes} below). Classically, this has the particularly nice form
\[
\cE^{j+2}_\Phif(\tau,-\tfrac{j}{2}) = \frac{(j+1)!}{(-2\pi i)^{j+2}}\sum_{(m,n) \in \Q^2 \backslash (0,0)} \frac{\widehat{\Phi}\subf(m,n)}{(m\tau + n)^{j+2}},
\]
where $\widehat{\Phi}\subf$ is the Fourier transform of $\Phi\subf$. This function is denoted $F^{j+2}_{\Phif}$ in \cite{LSZ17}. Passing to the adeles, this motivates the definitions
\begin{align}
\cE_{\Phif}^{j+2} : \GL_2(\A) &\longrightarrow \C, \hspace{12pt} g \longmapsto \cE_{\Phif}^{j+2}(g; -\tfrac{j}{2}),\notag \\
\cE_{\Phif}^{j+2,\chi} : \GL_2(\A) &\longrightarrow \C, \hspace{12pt} g \longmapsto \cE_{\Phif}^{j+2}(g;\chi, -\tfrac{j}{2}). \label{eq:classical E j chi}
\end{align}
Via \cref{cor:compatibility}, we are led to also consider
\begin{align}
\label{eq:E j chi}  E_{\Phi}^{j+2,\chi} : \GL_2(\A) &\longrightarrow \C, \hspace{12pt} g \longmapsto \|\det g\|^{j/2} E_\Phi\left(g;\chihat, -\tfrac{j}{2}\right).
\end{align}
Then by \cref{cor:compatibility}, we have
\begin{equation}\label{eq:compatibility}
\cE_{\Phif}^{j+2,\chi}(g) = E_{\Phif}^{j+2,\chi}(g).
\end{equation}

\begin{remark}
The functions $\cE_{\Phif}^{j+2,\chi}(g) = E_{\Phif}^{j+2,\chi}(g)$ depend $\GL_2(\A\subf)$-equivariantly on $\Phi$ and transform as elements of the global principal series representation
\begin{equation}\label{eq:I_j(chi)}
	I_j(\chihat) \defeq I\Big(\|\cdot\|^{-\tfrac{1}{2}}, \chihat^{-1}\|\cdot\|^{j+\tfrac{1}{2}}\Big) = \sideset{}{'}\bigotimes_{v < \infty} I_j(\chihat_v).
\end{equation}
Note $I_j(\chihat)$ is irreducible if $j > 0$. For $j = 0$ it does not even have finite length, since infinitely many of the local factors $I_0(\widehat\chi_v)$ are reducible.
\end{remark}

\subsection{Betti--Eisenstein classes}\label{sec:Betti-Eisenstein}
\subsubsection{Local systems and realisation maps}$\qquad$

\emph{Betti local systems:}
For $j \ge 0$, we let $\LSkn = \LSk$ denote the local system on the $\GL_2$ symmetric space associated to the $\GL_2$-representation $V_{(0,-j)}$ of highest weight $(0, -j)$. Note that if $a \in \Q_{> 0}$, $\stbt a 0 0 a$ acts on $\h^1(Y, \LSkn)$ as multiplication by $a^{-j}$.

\emph{De Rham local systems:} \label{sec:dR}To compare our (Betti)-Eisenstein classes to the classical Eisenstein classes of Harder, we go through a comparison to de Rham cohomology. The local system $\LSkn(\C)$ comprises the flat sections of a vector bundle $\sV_{(0,-j),\mathrm{dR}}$ with respect to the Gauss--Manin connection (see \cite[\S 1.18]{pink90}). This is defined over $\Q$, and we have a comparison isomorphism
\begin{equation}\label{eq:betti de Rham comparison}
C_{\mathrm{B,dR}} : \h^1\left(Y(\cU), \LSkn(\Q)\right) \otimes_{\Q} \C \isorightarrow \h^1_{\mathrm{dR}}\left(Y(\cU), \sV_{(0,-j),\mathrm{dR}}\right) \otimes_{\Q}\C.
\end{equation}
The pullback of $\sV_{(0,-j),\mathrm{dR}}$ to the upper half-plane is $\mathrm{Sym}^j\otimes \det^{-j}$ of the relative de Rham cohomology of $\C/(\Z+\tau\Z)$, so has a canonical section $(2 \pi i \mathrm{d}z)^{\otimes j}$, for $z$ a co-ordinate on $\C$.

\emph{Motivic local systems:}
Our Betti--Eisenstein classes will be the Betti realisation of Be\u{\i}linson's motivic Eisenstein classes. For any level $\cU$, there is a relative Chow motive $\sV_{(0,-j),\mathrm{mot}}(\Q)$ over $Y(\cU)$ attached to the representation $V_{(0,-j)}$ of $\GL_2/\Q$; this gives a coefficient system for motivic cohomology. Moreover, the Chow motive has a $\GL_2(\A\subf)$-equivariant structure, so its cohomology has a natural Hecke action.

Then we have realisations
\begin{align*}
r_{\mathrm{B}} : \h^1_{\mathrm{mot}}\left(Y(\cU), \sV_{(0,-j),\mathrm{mot}}(1)\right) &\longrightarrow \h^1\left(Y(\cU), \sV_{(0,-j)}(\Q)(1)\right)\\
&\hspace{40pt} \isorightarrow \h^1\left(Y(\cU), \sV_{(0,-j)}(\Q)\right), \\
r_{\mathrm{dR}} : \h^1_{\mathrm{mot}}\left(Y(\cU), \sV_{(0,-j),\mathrm{mot}}(1)\right) &\longrightarrow \h^1_{\mathrm{dR}}\left(Y(\cU), \sV_{(0,-j),\mathrm{dR}}\right).
\end{align*}
Base-changing to $\C$-coefficients, these are related by
\begin{equation}\label{eq:comparison}
C_{\mathrm{B,dR}}\circ r_{\mathrm{B}} = (2\pi i)^{-j-1} \cdot r_{\mathrm{dR}}
\end{equation}
(see e.g. \cite[\S2.2]{KLZ20} or \cite[pf.\ of Prop.\ 5.2]{LW18}, noting $T_j(j+1)$ \emph{op.\ cit.} is $\sV_{(0,-j)}(1)$ here).

\subsubsection{Betti--Eisenstein classes}
The main input in our construction is Be\u{\i}linson's family of motivic Eisenstein classes \cite{Bei86}, which we briefly summarise (cf.\ \cite[Thm.\ 7.2.2]{LSZ17}).

\begin{theorem}[Be\u{\i}linson] \label{thm:motivic classes}
For any $j \ge 0$, and any level $\cU$, there is a canonical $\GL_2(\A\subf)$-equivariant map, the \emph{motivic Eisenstein symbol},
\[
\cS_{(0)}(\A\subf^2, \Q)^{\cU}\to \h^1_{\mathrm{mot}}\left(Y(\cU), \sV_{(0,-j),\mathrm{mot}}(1)\right),\qquad \Phif \mapsto \Eis^j_{\mathrm{mot},\Phif},
\]
compatible with Hecke operators and changing $\cU$, such that the pullback to the upper half-plane of $r_{\mathrm{dR}}(\mathrm{Eis}_{\mathrm{mot},\Phif}^j)$ is the differential form $-\cE^{j+2}_\Phif(\tau;-\tfrac{j}{2}) \cdot (2\pi i\mathrm{d}z)^{\otimes j} \cdot 2\pi i\mathrm{d}\tau$ (cf.\ \cref{sec:dR}).
\end{theorem}

\begin{corollary}\label{cor:betti-eisenstein}
For any $j \ge 0$, and any level $\cU$, there is a canonical $\GL_2(\A\subf)$-equivariant map, the \emph{Betti--Eisenstein symbol},
\[
\cS_{(0)}(\A\subf^2, \Q)^{\cU}\to \h^1\left(Y(\cU), \LSkn(\Q)\right),\qquad \Phif \mapsto \Eis^j_{\Phif} \defeq r_{\mathrm{B}}(\mathrm{Eis}_{\mathrm{mot},\Phif}^j),
\]
compatible with Hecke operators and changing $\cU$, such that $C_{\mathrm{B,dR}}(\mathrm{Eis}_\Phif^j)$ is the differential form
$-\cE^{j+2}_\Phif(g) \cdot (\mathrm{d}z)^{\otimes j} \cdot \mathrm{d}\tau$ on $Y(\cU)$.
\end{corollary}

\begin{proof}
By \eqref{eq:comparison}, the pullback to the upper half-plane of $C_{\mathrm{B,dR}}(\mathrm{Eis}_\Phif^j)$ is $-\cE^{j+2}_{\Phif}(\tau;-\tfrac{j}{2}) \cdot (\mathrm{d}z)^{\otimes j} \cdot \mathrm{d}\tau$. The extension of $\cE_{\Phif}^{j+2}$ to $\GL_2(\A\subf) \times \cH$ in \cref{def:classical eisenstein}(ii) is the unique one preserving $\GL_2(\A\subf)$-equivariance, and the further extension of this to $\GL_2(\A)$ in \eqref{eq:classical on GL2(A)} is the unique one that is automorphic; it follows that the extension of the differential to $Y(U)$ has form $-\cE^{j+2}_{\Phif}(g;-\tfrac{j}{2})\otimes (\mathrm{d}z)^j \otimes \mathrm{d}\tau$. But $\cE^{j+2}_{\Phif}(g) = \cE^{j+2}_{\Phif}(g;-\tfrac{j}{2})$ by definition.
\end{proof}

\subsubsection{Integrality}\label{sec:Eisenstein integrality}
In general these Eisenstein symbols do not take values in the integral cohomology. However, we can work around this as follows. Let $c > 1$ be coprime to $6$, and let $\cU$ be a level of the form $\cU^{(c)} \times \prod_{\ell \mid c} \GL_2(\Z_\ell)$. We write ${}_c \cS(\A\subf^2, \Z)$ for the $\Z$-valued Schwartz functions of the form $\Phi\subf^{(c)} \times \prod_{\ell \mid c} \operatorname{ch}(\Z_\ell^2)$, and similarly ${}_c \cS_0(\A\subf^2, \Z)$.

\begin{theorem}\label{thm:integral betti eisenstein}
There exist homomorphisms
\[
{}_c \cS_{(0)}(\A\subf^2, \Z)^{\cU}\to \h^1\left(Y(\cU), \LSkn({\Z})\right),\qquad \Phif \mapsto {}_c\Eis^j_{\Phif},
\]
compatible with Hecke operators and changing $\cU$, such that after base-extending to $\Q$ we have
\[ {}_c\Eis^j_{\Phif} = \left(c^2 - c^{-j} \stbt{c}{0}{0}{c}^{-1}\right)\Eis^j_{\Phif} \]
where $\stbt{c}{0}{0}{c}$ is understood as an element of $\GL_2(\A\subf^{(c)})$.
\end{theorem}

\subsubsection{$p$-adic interpolation}\label{sec:Eisenstein interpolation}

Let us now fix a prime $p$, an open compact $\cU^{(p)} \subseteq \GL_2(\A\subf^{(p)})$, and a prime-to-$p$ Schwartz function $\Phi^{(p)} \in \cS(\A\subf^{(p)},\Z)^{\cU^{(p)}}$. We suppose $c$ is coprime to $p$, and that $\cU$ and $\Phi^{(p)}$ are unramified at the primes dividing $c$. After tensoring with $\Zp$, we can take the local systems $\sV_{(0,-j)}^{\GL_2}$ (and the class ${}_c \mathrm{Eis}_{\Phif}^j$) to have $\Zp$-coefficients.

For $t \geq 0$, let $\Phi_{p,t} \defeq  \operatorname{ch}[ (0, 1) + p^t\Zp^2]$, a Schwartz function at $p$. The Schwartz functions $\Phift \defeq \Phi^{(p)}\cdot \Phi_{p,t}$, for $t \ge 1$, are stable under the group $\cU_1(p^t) = \cU^{(p)} \cdot \{ \stbt{*}{*}{0}{1} \bmod p^t\}$, and compatible with the trace maps for varying $t$. Extending from $\Z$- to  $\Zp$-coefficients, we obtain an \emph{Eisenstein--Iwasawa class}
\[ {}_c \mathcal{EI}_{\Phi^{(p)}} \defeq ({}_c\mathrm{Eis}^0_{\Phi_{\mathrm{f},t}})_t \in \varprojlim_t \h^1\left(Y(\cU_1(p^t)), \Zp\right) \eqcolon \h^1_{\mathrm{Iw}}\left(Y(\cU_1(p^\infty)), \Zp\right) .\]

As explained in \cite[\S 9.1]{LSZ17}, there are natural ``moment'' maps
\[
\mom^j_t : \h^1_{\mathrm{Iw}}\left(Y(\cU_1(p^\infty)), \Zp\right) \to \h^1\left(Y(\cU_1(p^t)), \LSkn({\Zp})\right)
\]
for each $j \ge 0$ and $t \ge 0$; and if $t \ge 1$, we have
\[
\mom^j_t\left({}_c \mathcal{EI}_{\Phi^{(p)}}\right)= {}_c \Eis^j_{\Phift}.
\]
This is true by definition for $j = 0$; that it also holds for $j > 0$ is a deep theorem due to Kings. There is a similar statement for $t = 0$, but we need to replace the group $\cU_1(p^t)$ with the $\GL_2$ Iwahori subgroup $\cU_0(p)$. All of the above structures are compatible with the action of $\GL_2(\A\subf^{(pc)})$.

Note that the Iwasawa cohomology group containing ${}_c \mathcal{EI}_{\Phi^{(p)}}$ is a module over the Iwasawa algebra $\Lambda = \Zp[\![\Zp^\times]\!]$, and the moment map $\mom^j_t$ factors through the quotient where $(1 + p^t \Zp)^\times$ acts via $x \mapsto x^j$. If $\Phi^{(p)}$ belongs to the $\chihat$-eigenspace for the action of the centre, for some Dirichlet character $\chi$ of prime-to-$p$ conductor (taking values in a finite extension $\cO$ of $\Zp$), then the element
\[ \mathcal{EI}_{\Phi^{(p)}} \defeq \left(c^2 - c^{-\mathbf{j}} \chi(c) \right)^{-1} \otimes {}_c \mathcal{EI}_{\Phi^{(p)}} \in \h^1_{\mathrm{Iw}} \otimes \operatorname{Frac}(\Lambda) \]
is independent of $c$, where $\mathbf{j}$ is the universal character $\Zp^\times \hookrightarrow \Lambda^\times$. In particular, if $\chi$ is non-trivial, then we can choose $c$ so that the above factor is invertible in $\Lambda \otimes \Qp$, and hence define $\mathcal{EI}_{\Phi^{(p)}}$ as an element of $\h^1_{\mathrm{Iw}} \otimes_{\Zp} \Qp$. (We can even work integrally if $\chi \neq \mathbf{1} \newmod{p}$).

%%%%%%%%%%%%%%%%%%%%%%%%%%%%%%%%%%%%%%%%%%%%%%%%%%%%%%%%%%%%%

\section{Constructing the measure}\label{sec:construction}

%%%%%%%%%%%%%%%%%%%%%%%%%%%%%%%%%%%%%%%%%%%%%%%%%%%%%%%%%%%%%

We fix henceforth a prime $p$, allowing $p = 2$.
We now construct a $p$-adic measure interpolating pushforwards (to $\GL_3$) of Betti--Eisenstein classes. For this, we use the norm compatibility of these classes to systematically control their denominators after pushforward. Crucially, the measure we construct will be valued in the dual to degree 2 compactly supported cohomology for $\GL_3$.

\begin{notation}\label{not:construction}
We work with a fixed RACAR $\Pi$ in mind, of weight $\lambda = (a,0,-a)$, with $a \in \Z_{\geq 0}$. Our $p$-adic $L$-function will interpolate the critical $L$-values in the left half $\Crit_p^-(\Pi)$ of the critical strip. This will correspond to privileging the parabolic $P_1 \subset \GL_3$, with associated Hecke operator $U_{p,1}$, and considering $P_1$-refinements as in \cref{sec:ordinarity}. Let $\sigma_p$ be a $P_1$-refinement of $\Pi$; via \cref{prop:nearly-ordinary vs ordinary}, replacing $\Pi$ by a character twist if necessary, we may suppose that $\sigma_p$ is unramified. Our local level at $p$ will be
\begin{equation}\label{eq:U_p measure}
\cU_p \defeq \cU_{1,p}^{(P_1)}(p^{r(\Pi,\alpha_p)}) \subset \GL_3(\Zp)
\end{equation}
(see \eqref{eq:U_p(Pi)}, for $r(\Pi,\alpha_p)$ as in \cref{prop:P_1 conductor}, recalling $\alpha_p = \sigma_p(p)$).

Let $\Phi^{(p)} \in \cS(\A\subf^{(p),2},\Z)$. We fix a tame level $\cU^{(p)} \subset \GL_3(\A\subf^{(p)})$ such that $\Phi^{(p)}$ is fixed by $\cU^{(p)}\cap H$. We also fix an integer $c > 1$ coprime to $6p$ and to the levels of $\cU^{(p)}$ and $\Phi^{(p)}$. Let $\cU = \cU^{(p)}\cU_p$, which -- up to shrinking $\cU^{(p)}$ -- we may assume is neat.
\end{notation}

\subsection{Definition of the classes}\label{sec:define classes}

\subsubsection{Branching laws and Eisenstein classes}\label{sec:construction branching}
Recall $H = \GL_2 \times \GL_1$. For numerology purposes, our pushforward maps will be from $\widetilde{Y}^H$ to $Y^{\GL_3}$, so we now transfer our Eisenstein classes from $Y^{\GL_2}$ to $\widetilde{Y}^H$. If $\cU^H \subset H(\A\subf)$ is open compact, we have a natural composition
\[
\mathrm{pr}_{\GL_2} : \widetilde{Y}^H(\cU^H) \longrightarrow Y^H(\cU^H) \longrightarrow Y^{\GL_2}(\cU^{H}\cap \GL_2),
\]
where the first map is from \eqref{eq:tilde maps} and the second map is induced by projection $H \to \GL_2$.
Pullback gives a map
\[
\mathrm{pr}_{\GL_2}^* : \h^1\left(Y^{\GL_2}(\cU^{\GL_2}), \LSk\right) \longrightarrow \h^1\left(\widetilde{Y}^H(\cU^H), \sV^H_{(0,-j;0)}\right).
\]
For appropriate $\cU^H$ we freely identify ${}_c\mathrm{Eis}^j_\Phif$ with its image under this map.

Let $j \in \Z$ with $0 \le j \le a$. Recall the characters $\nu_1,\nu_2$ on $H$ from \cref{def:nu}. We have an isomorphism
\[
V_{(0,-j;0)}^H \otimes \|\nu_1^j\| \isorightarrow V_{(j,0;-j)}^H
\]
of $H$-representations. Passing this twist outside the cohomology inverts it, so for appropriate $\cU^H$ we may thus consider
\begin{align*}
\mathrm{tw}_j({}_c\mathrm{Eis}_\Phif^j) \defeq {}_c\mathrm{Eis}_\Phif^j \otimes \| \nu_1^{-j}\| \in &\h^1\left(\widetilde{Y}^H(\cU^H),\sV^H_{(0,-j;0)}\right)\left[\|\nu_1^{-j}\|\right]\\
& \cong \h^1\left(\widetilde{Y}^H(\cU^H),\sV^H_{(j,0;-j)}\right).
\end{align*}

Recall the morphisms $\operatorname{br}^{[a,j]} :  V^H_{(j, 0; -j)}\to \iota^*(V_\lambda)$ defined in \cref{sec:branching laws}. We may (and do) suppose that these are integrally defined, and pushing forward we obtain an $H(\A\subf)$-equivariant map
\[
\mathrm{br}^{[a,j]}_\star : \h^1\big(\widetilde{Y}^H(\cU^H), \sV_{(j,0;-j)}^H(\Zp)\big) \to \h^1\left(\widetilde{Y}^H(\cU^{H}), \iota^*\sV_{\lambda}(\Zp)\right).
\]

\subsubsection{Towers of level groups} \label{sec:level tower} We now pushforward under $\iota$.  Let $r = r(\Pi,\alpha_p)$ from \cref{prop:P_1 conductor}.

\begin{definition}\label{def:U_n}
For $n \ge 1$, let $\cU_n$ denote the group
\[
\cU^{(p)} \times \left\{ g \in \GL_3(\Zp) : g = 1 \bmod
\smallthreemat{\star}{p^n}{p^{n}}{\star}{\star}{\star} {p^{r}}{p^r}{p^r}\right\}.
\]
Write $u$ for the element $\smallthreemat{1}{0}{1}{}{1}{0}{}{}{1}\smallthreemat{1}{}{}{}{}{-1}{}{1}{} \in \GL_3(\Zp)$.
\end{definition}

We shall let $t \geq \mathrm{max}(n,r)$ and consider the morphism
\[
\iota_{n, t}: \widetilde{Y}^H(\cU_{1}^H(p^t) \cap u \cU_n u^{-1}) \xrightarrow{\ \iota\ } Y^{\GL_3}(u \cU_n u^{-1}) \xrightarrow{\ \ u\ \ } Y^{\GL_3}(\cU_n),
\]
where
\[
\cU_{1}^H(p^t) = \left(\cU^{(p)} \cap H(\A\subf^{(p)})\right) \times \{ (\stbt \star\star 0 1, \star) \bmod p^t\}.
\]

\begin{proposition}\label{prop:intersection level}
If $t \ge \mathrm{max}(n,r)$, then $\cU_{1}^H(p^t) \cap u \cU_n u^{-1}$ consists of all $(\stbt a b c d, z) \in \cU_{1}^H(p^t)$ with $b = 0\bmod p^{n}$, $a = z  \bmod p^n$. Its index in $\cU_1^H(p^t)$ is $p^{2n-1}(p-1)$.
\end{proposition}

\begin{proof}
Let $g = \smallthreemat{A}{B}{C}{D}{E}{F}{G}{H}{I} \in \cU_n$. Then suppose there exists $\left(\smallmatrd{a}{b}{c}{d}, z\right)\in \cU_{1}^H(p^t)$ with
\[
\smallthreemat{a}{b}{0}{c}{d}{0}{0}{0}{z} = \iota\Big(\smallmatrd{a}{b}{c}{d}, z\Big) = ugu^{-1} = \smallthreemat{A+D}{-(C+F)}{(B+E) - (A+D)}{-G}{I}{G-H}{D}{-F}{-D+E}.
\]
Then $D=F=0$, so $b = -(C+F) = -C \equiv 0 \newmod{p^n}$, and $z = -D+E = E$. Moreover $(B+E)-(A+D) = 0$ implies $a= A+D = A \equiv E = z \newmod{p^n}$, as required. Since $t \geq r$, the congruence conditions $\bmod p^r$ satisfied by $G,H$ and $I$ in the definition of $\cU_n$ impose no further restrictions on $c \equiv 0 \newmod{p^t}$ and $d \equiv 1 \newmod{p^t}$. The restrictions on $b$ and $z$ introduce $p^n$ and $p^{n-1}(p-1)$ to the index respectively.
\end{proof}

\subsubsection{Construction at fixed level $n$}
Recall the map $\nu = (\nu_1,\nu_2): H \to \GL_1 \times \GL_1$ given by $(\gamma, z) \mapsto (\tfrac{\det \gamma}{z}, z)$. By \cref{prop:intersection level}, we have $\nu_1(h_p) = 1\newmod{p^n}$ for all $h \in \cU_{1}^H(p^t) \cap u \cU_n u^{-1}$. Thus $\nu_1$ induces a locally constant map
\begin{equation}\label{eq:nu_n}
\nu_{1, (n)}: \widetilde{Y}^H\left(\cU_{1}^H(p^t) \cap u \cU_n u^{-1}\right) \twoheadrightarrow \Delta_n \defeq (\Z/p^n)^\times.
\end{equation}

\begin{remark}\label{rem:dirichlet convention}
Here we have identified $\Delta_n$ with a quotient of $\Q^\times \backslash \A^\times$, in such a way that a uniformiser at a prime $\ell \ne p$ in $\A\subf^\times$ is mapped to $\ell \bmod p^n$. This is consistent with our convention for Dirichlet characters elsewhere in the text (but its restriction to $\Zp^\times \subset \A^\times$ is the inverse of the obvious reduction map).
\end{remark}
Composing with the canonical map $\Delta_n \to \Zp[\Delta_n]^\times$, $\delta \mapsto [\delta]$, we obtain a locally constant function $[\nu_{1, (n)}]$ on $\widetilde{Y}^H\left(\cU_{1}^H(p^t) \cap u \cU_n u^{-1}\right)$ with values in the group ring $\Zp[\Delta_n]$, i.e.~an element of $\h^0\left(\widetilde{Y}^H\left(\cU_{1}^H(p^t) \cap u \cU_n u^{-1}\right), \Zp\right) \otimes_{\Zp} \Zp[\Delta_n]$.

\begin{notation}\label{not:eta_2}  For technical reasons (see \cref{prop:get rid of c}), we introduce a second auxiliary character. Let $\eta_2$ be an even Dirichlet character of prime-to-$p$ conductor, which (for notational convenience) we suppose to be $\Zp$-valued. Shrinking $\cU^{(p)}$ if necessary, we may suppose $\cU_{1}^H(p^t) \cap u \cU_n u^{-1} \subseteq \ker(\widehat{\eta}_{2})$, and thus regard $\widehat{\eta}_{2} \circ \nu_2$ as a class in $\h^0\left(\widetilde{Y}^H\left(\cU_{1}^H(p^t) \cap u \cU_n u^{-1}\right), \Zp\right)$.
\end{notation}

\begin{definition}\label{def:zaj}
For any $t \geq \mathrm{max}(n,r)$, we define
\begin{align*}
	{}_cz_{n}^{[a,j]}(\cU^{(p)}, \Phi^{(p)}) &\defeq p^{an} \left[ \big(\iota_{n,t, \star} \circ \operatorname{br}^{[a,j]}\circ \operatorname{tw}_i\big) \otimes 1\right]\left({}_c \Eis^{j}_{\Phift} \cup [\nu_{1, (n)}] \cup (\widehat{\eta_2} \circ \nu_2)\right)\\
	& \in   \h^3\left(Y^{\GL_3}(\cU_n), \sV_\lambda(\Zp)\right) \otimes_{\Zp} \Zp[\Delta_n].
\end{align*}
Let $z_{n}^{[a,j]}(-) \in \h^3\left(Y^{\GL_3}(\cU_n), \sV_\lambda(\Qp)\right) \otimes_{\Zp} \Zp[\Delta_n]$ be the analogue defined instead using $\Eis_{\Phi\subf^t}^j$ (without the $c$ factor).
\end{definition}

These classes are independent of $t$ for $t \ge \mathrm{max}(n,r)$, since the maps $\iota_{n, t}$ for varying $t$ are compatible with the natural trace maps. The normalisation term $p^{-an}$ will later ensure compatibility with the normalised Hecke operator $U_{p,1}$ (see \eqref{eq:tau}).

\begin{remark}\label{rem:connected components}
Philosophically,  this process formally `spreads out' the cohomology over (unions of) connected components.  More precisely, let $D = \mathrm{cond}(\eta_2)$. Then we have a map
\[
\nu_{1,(n)} \times \nu_{2,D} : \widetilde{Y}^H\left(\cU_{1}^H(p^t) \cap u \cU_n u^{-1}\right) \twoheadrightarrow \Delta_n \times (\Z/D)^\times,
\]
and
\[
\widetilde{Y}^H\left(\cU_{1}^H(p^t) \cap u \cU_n u^{-1}\right) = \bigsqcup_{(x,y) \in \Delta_n \times (\Z/D)^\times} \widetilde{Y}^H_{x,y},
\]
where $\widetilde{Y}^H_{x,y} \defeq [\nu_{1,(n)} \times \nu_{2,D}]^{-1}(x,y)$ is a union of connected components. Let $i_{x,y} : \widetilde{Y}^H_{x,y} \hookrightarrow \widetilde{Y}^H\left(\cU_{1}^H(p^t) \cap u \cU_n u^{-1}\right)$ be the inclusion. Then
\[
{}_c \mathrm{Eis}_{\Phift}^{j} \cup [\nu_{1,(n)}] \cup (\widehat{\eta_2}\circ\nu_2)= \sum_{x \in\Delta_n} \left[\bigg(\sum_{y\in(\Z/D)^\times} \eta_2(y) \cdot i_{x,y}^*\left({}_c \mathrm{Eis}_{\Phift}^{j}\right)\bigg) \otimes [x]\right].
\]
\end{remark}

Our chief interest is in a modified version of these classes.

\begin{definition}\label{def:V_n}
Let $\tau \defeq \tau_1 \in \GL_3(\Zp)$ from \eqref{eq:tau}, and let
\[
\cV_n \defeq \tau^{-n} \cU_n \tau^n =  \cU^{(p)} \times \left\{ g \in \GL_3(\Zp) : g = 1 \bmod
\smallthreemat {\star}{\star}{\star} {p^n}{\star}{\star} {p^{n+r}}{p^r}{p^r}\right\}.
\]
\end{definition}
Translation by $\tau^n$ gives a map $Y^G(\cU_n) \to Y^G(\cV_n)$. Crucially, $\cV_n \subset \cU_p$ from \eqref{eq:U_p measure}. We use this to define a pushforward map $(\tau^n)_\star$ on cohomology with coefficients in $\sV_\lambda$, scaling by $p^a$ (as in our definition of Hecke operators) in order to obtain a map on the integral lattices $\sV_\lambda(\Zp)$.

\begin{definition}\label{def:xi aj}
We set
\[
{}_c\xi^{[a,j]}_n(-) = \left[\tau^n_\star \otimes 1\right] \left({}_c z_{n}^{[a,j]}(-)\right) \in   \h^3(Y^{\GL_3}(\cV_n), \sV_\lambda(\Zp)) \otimes_{\Zp} \Zp[\Delta_n].
\]
Similarly, define $\xi^{[a,j]}_n(-)$ (with $\Qp$ coefficients) to be the analogue using $z_{n}^{[a,j]}(-)$.
\end{definition}

The dependence on $c \in \Z$ is given by
\begin{equation}
\label{eq:c-dependence}
{}_c\xi_{n}^{[a,j]}(\cU^{(p)}, \Phi^{(p)}) = \left(c^2 - c^{-j} \eta_2(c)^{-1}  \langle c \rangle \otimes [c]^{-1}\right) \cdot \xi_{n}^{[a,j]}\left(\cU^{(p)}, \Phi^{(p)}\right),
\end{equation}
where $[c]$ is the class of $c$ in $\Delta_n = (\Z/p^n)^\times$, and $\langle c \rangle$ denotes pullback by $\operatorname{diag}(c, c, c)^{-1} \in Z_G(\Zhat^{(c)})$.
\subsubsection{Summary}
The following diagram summarises the construction of ${}_c \xi_n^{[a,j]}$. For notational convenience,  we write $\cU^H_{n} \defeq \cU^H_1(p^n)\cap u\cU_nu^{-1}$.

\begin{equation}\label{eq:main diagram}
\xymatrix@C=1mm@R=5mm{
	{}_c\mathrm{Eis}_{\Phift}^{j} \in \ar@{|-->}[ddddddd] 	&&&&&&&  \hbox to 3em{\hss$\h^1\left(Y^{\GL_2}(\cU_{n}^H\cap \GL_2),  \sV_{(0,-i)}^{\GL_2}(\Zp)\right)$\hss}\ar[d]^{\mathrm{pr}_{\GL_2}^*}  &\\
	&&&&&&&  \hbox to 3em{\hss$\h^1\left(\widetilde{Y}^H(\cU^H_{n}),\sV^H_{(0,-j;0)}(\Zp)\right)$\hss}	 \ar@{=}[d]^{\mathrm{tw}_j}&\\
	&&&&&&&  \hbox to 3em{\hss$\h^1\left(\widetilde{Y}^H(\cU^H_{n}),\sV^H_{(j,0;-j)}(\Zp)\right)$\hss}	 \ar[d]^{\cup [\nu_{1, (n)}] \cup (\hat\eta_2 \circ \nu_2)}&\\
	&&&&&&& \hbox to 3em{\hss$\h^1\left(\widetilde{Y}^H(\cU^H_{n}),\sV^H_{(j,0;-j)}(\Zp)\right) \otimes \Zp[\Delta_n]$\hss}\ar[d]^-{\mathrm{br}^{[a,j]} \otimes 1}& \\
	&&&&&&& \hbox to 3em{\hss$\h^1\left(\widetilde{Y}^H(\cU^H_{n}),\sV_\lambda(\Zp)\right) \otimes \Zp[\Delta_n]$\hss} \ar[rdd]^-{\iota_{n,t,*} \otimes 1} &\\
	&&&&&&&&\\
	&&&&&&&&	\h^3\left(Y^{\GL_3}(\cU_n),\sV_\lambda(\Zp)\right) \otimes \Zp[\Delta_n]\ar[d]^-{\tau^n_\star \otimes 1} \\
	p^{-an} {}_c \xi_n^{[a,j]}  \ \ \in
	&&&&&&& & \h^3\left(Y^{\GL_3}(\cV_n),\sV_\lambda(\Zp)\right) \otimes \Zp[\Delta_n].
}
\end{equation}

\subsubsection{Varying $n$} If $\cU' \subset \cU \subset \GL_3(\A\subf)$, let $\mathrm{pr}_{\cU}^{\cU'}$ denote the natural projection map $Y^{\GL_3}(\cU') \to Y^{\GL_3}(\cU)$. We get associated pullback/pushforward maps on cohomology. Also let $\mathrm{norm}^{\Delta_{n+1}}_{\Delta_n}$ denote the natural projection map $\Zp[\Delta_{n+1}] \to \Zp[\Delta_n]$.

\begin{theorem}[Norm relation]\label{thm:norm relation}
For any $n \ge 1$ we have
\[
\left[\left(\mathrm{pr}^{\cV_{n+1}}_{\cV_n}\right)_\star \otimes \operatorname{norm}_{\Delta_n}^{\Delta_{n+1}}\right]\left(\xi_{n+1}^{[a,j]}(-)\right) = \left[U_{p,1}' \otimes 1\right] \cdot \xi_{n}^{[a,j]}(-)
\]
as elements of $\h^3(Y^{\GL_3}(\cV_n), \sV_\lambda(\Zp)) \otimes_{\Zp} \Zp[\Delta_n].$
\end{theorem}

\begin{proof}
This is an instance of \cite[Prop.\ 4.5.2]{loeffler-parabolics}, elaborated in \S4.6 `\emph{The Betti setting}' and \S5.2.3 \emph{op.\ cit}. It is an analogue of \cite[Thm.\ 3.13]{LW18}. As this result is key to our construction, for the convenience of the reader we sketch the proof in this special case, translating the notation of \cite{loeffler-parabolics} into our setting. We shall drop the indices $[a,j]$ and $c$ here for brevity.

We first prove an analogue for the elements $z_n$ (as in \cite[Prop.\ 4.5.1]{loeffler-parabolics}). We claim that
\begin{equation}\label{eq:z_n norm relation}
	\left[ 1 \otimes \operatorname{norm}_{\Delta_n}^{\Delta_{n+1}}\right](z_{n+1}) = p^{a}\left[\left(\operatorname{pr}_{\cU_n}^{\cU_{n+1}}\right)^* \otimes 1\right](z_n)
\end{equation}
as elements of $\h^3(Y^{\GL_3}(\cU_{n+1}), \sV_\lambda(\Zp)) \otimes_{\Zp} \Zp[\Delta_n]$.	To prove this, one fixes $t \geq n+1$ and checks that the horizontal maps in the diagram
\begin{equation}\label{eq:cartesian}
	\xymatrix@C=20mm{
		\widetilde{Y}^H(\cU_{1}^H(p^t) \cap u \cU_{n+1} u^{-1})\ar[d] \ar[r]^-{\iota_{n+1,t}} & Y^G(\cU_{n+1})\ar[d] \\
		\widetilde{Y}^H(\cU_{1}^H(p^t) \cap u \cU_n u^{-1}) \ar[r]^-{\iota_{n,t}} & Y^G(\cU_n)
	}
\end{equation}
are injective. (Note: this injectivity is the crucial reason for introducing the twisting map $u$; cf.\ \cite[Rem.\ 4.12]{LW18}). By injectivity, the diagram is Cartesian as both vertical maps have degree $p^2$ (using \cref{prop:intersection level}). Now, from the definitions we can write
\begin{equation}\label{eq:cartesian 2}
	\left[1 \otimes \operatorname{norm}_{\Delta_n}^{\Delta_{n+1}}\right](z_{n+1}) = p^{a(n+1)}\big[\big(\iota_{n+1, t,\star} \circ \operatorname{br}\circ \operatorname{tw}\big) \otimes 1\big] \left({}_c \Eis^{j}_{\Phift} \cup [\nu_{(n)}]\right).
\end{equation}
The class ${}_c \Eis_{\Phift}^j \cup [\nu_{(n)}]$ (at the top left corner) is a pullback from level $\cU_{1}^H(p^t) \cap u \cU_n u^{-1}$ (that is, the bottom left corner). Thus the right-hand side of \eqref{eq:cartesian 2} is obtained by passing from the bottom-left to top right of \eqref{eq:cartesian} along the left and top arrows, and \eqref{eq:z_n norm relation} (where the right-hand side is obtained along the bottom and right arrows) follows from the compatibility of pushforward and pullback in Cartesian diagrams.

Now recall that on cohomology at level $\cU_n$, we have	$U_{p,1}' = p^{a} \big(\mathrm{pr}^{\tau^{-1} \cU_n \tau \cap \cU_n}_{\cU_n}\big)_\star\circ \tau_\star \circ \big(\mathrm{pr}_{\cU_n}^{\cU_n \cap \tau \cU_n\tau^{-1}} \big)^*.$
One easily sees $\cU_{n+1} = \cU_n \cap \tau \cU_n\tau^{-1}$. Applying the map $\left[\left((\mathrm{pr}^{\tau^{-1} \cU_n \tau \cap \cU_n}_{\cU_n}\big)_\star\circ \tau_\star \right)\otimes 1]\right]$ to \eqref{eq:z_n norm relation}, we see
\begin{equation}\label{eq:z_n norm relation 2}
	\left[(\mathrm{pr}^{\tau^{-1} \cU_n \tau \cap \cU_n}_{\cU_n}\big)_\star\circ \tau_\star \otimes  \operatorname{norm}_{\Delta_n}^{\Delta_{n+1}}\right] (z_{n+1}) = \left[U_{p,1}' \otimes 1\right] \cdot z_n
\end{equation}
as elements of $\h^3(Y^{\GL_3}(\cU_{n}), \sV_\lambda(\Zp)) \otimes_{\Zp} \Zp[\Delta_n].$ One checks that
\begin{align*}
	(\tau^n)_\star \circ 	\left(\mathrm{pr}^{\tau^{-1} \cU_n \tau \cap \cU_n}_{\cU_n}\right)_\star \circ \tau_\star &= 	\left(\mathrm{pr}^{\tau^{-n-1} \cU_n \tau^{n+1} \cap \tau^{-n}\cU_n\tau^n }_{\tau^{-n}\cU_n\tau^n}\right)_\star \circ (\tau^{n+1})_\star\\
	&= \left(\mathrm{pr}_{\cV_n}^{\cV_{n+1}}\right)_\star \circ (\tau^{n+1})_\star,
\end{align*}
whilst a simple check on single coset representatives shows that $(\tau^n)_\star \circ U_{p,1}' = U_{p,1}' \circ (\tau^n)_\star$. Since $\xi_n = (\tau^n)_\star(z_n)$ by definition, the result follows by applying $[(\tau^n)_\star \otimes 1]$ to \eqref{eq:z_n norm relation 2}.
\end{proof}

\subsection{Pairing with a \texorpdfstring{$\GL_3$}{GL3}-eigenclass}\label{sec:pairing with cusp form}

%Recall that in Notation \ref{not:construction}, we fixed a $P_1$-refinement $\sigma_p$ of $\Pi$, and took $\alpha_p = \sigma_p(p)$.  

Recall that in Notation \ref{not:construction}, we fixed a level subgroup $\cU^{(p)} \subset \GL_3(\A\subf)$ fixing a Schwartz function $\Phi^{(p)} \in \cS(\A\subf^{(p),2},\Z)$. Let $L/\Qp$ be a finite extension with ring of integers $\cO_L$, and
\[
\phi \in \hc{2}\left(Y^{\GL_3}(\cU),\sV_\lambda^\vee(\cO_L)\right)
\]
be any cohomology class. Let $\cU_n$ be as in \cref{def:U_n}, $\cV_n \defeq \tau^{-n}\cU_n\tau^n$, and let $\phi_n$ be the image of $\phi$ in the cohomology of $Y^{\GL_3}(\cV_n)$, recalling that $\cV_n\subset\cU$. Recall the pairing $\langle-,-\rangle_{\cV_n}$ from \cref{sec:cup product pairing}. Let $\mathrm{d}h\subf$ denote the unramified Haar measure on $H(\A\subf)$, which defines a volume $\mathrm{vol}(\cU^{(p)}\cap H)$.

\begin{proposition}\label{cor:ind of U}
The element of $\cO_L[\Delta_n]$ defined by
\begin{equation}\label{eq:measure finite level}
	{}_c\Xi_n^{[a,j]}(\phi, \Phi^{(p)}) \defeq \operatorname{vol}(\cU^{(p)} \cap H) \cdot \left\langle \phi_n,\  {}_c \xi_{n}^{[a,j]}\left( \cU^{(p)}, \Phi^{(p)}\right) \right\rangle_{\cV_n}
\end{equation}
is independent of $\cU^{(p)}$ (for suitably chosen $c$).
\end{proposition}
\begin{proof}
If $\cU' \subseteq \cU \subset \GL_3(\A\subf^{(p)})$ are two levels, with $\cU \cap H$ fixing $\Phi^{(p)}$, then one may check that
\begin{equation}\label{eq:change of U}
	\big(\operatorname{pr}^{\cU}_{\cU'}\big)_\ast\left({}_c\xi_{n}^{[a,j]}\big(\cU, \Phi^{(p)}\big)\right) = [\cU \cap H : \cU' \cap H] \cdot  {}_c\xi_{n}^{[a,j]}\left(\cU', \Phi^{(p)}\right).
\end{equation}
Let $\cU_1^{(p)}, \cU^{(p)}_2$ be arbitrary. Then their intersection $\cU_3^{(p)}$ also satisfies the conditions, and we see the result by applying \eqref{eq:change of U} to $\cU_3^{(p)} \subset \cU_1^{(p)}, \cU_2^{(p)}$.
\end{proof}

For finite-slope eigenclasses, we can also vary $n$.

\begin{corollary}
Suppose  $U_{p,1}\phi = \alpha\phi$, with $0 \neq \alpha \in \cO_L$.
	\begin{itemize}
		\item[(i)]
The element
\[
{}_c\Xi^{[a,j]}(\phi, \Phi^{(p)}) \defeq \left(\alpha^{-n} \cdot {}_c\Xi_n^{[a,j]}(\phi, \Phi^{(p)})\right)_{n \ge 1} \in L[\![ \Zp^\times]\!]
\]
is a well-defined distribution on $\Zp^\times$ (a linear functional on locally-constant functions). 

\item[(ii)]If $v_p(\alpha)=0$, then ${}_c\Xi^{[a,j]}(\phi, \Phi^{(p)}) \in \cO_L[\![\Zp^\times]\!]$ is a measure.
\end{itemize}
\end{corollary}
\begin{proof}
By the above remarks, for (i) it suffices to show that
\begin{equation}\label{eq:measure relation}
	\mathrm{norm}^{\Delta_{n+1}}_{\Delta_n}\left( \alpha^{-(n+1)}\cdot 	{}_c\Xi_{n+1}^{[a,j]}\big(\phi, \Phi^{(p)}\big)\right) = \alpha^{-n} \cdot 	{}_c\Xi_n^{[a,j]}\left(\phi, \Phi^{(p)}\right).
\end{equation}
By definition, we have $\phi_{n+1} = (\mathrm{pr}_{\cV_n}^{\cV_{n+1}})^*(\phi_n).$ Since $c, [a,j], \Phi^{(p)}$ and $\cU^{(p)}$ are fixed, we drop them from notation for clarity, and see
\begin{align}
	\mathrm{norm}^{\Delta_{n+1}}_{\Delta_n}&\left(\alpha^{-(n+1)} \Big\langle \phi_n, \xi_{n+1} \Big\rangle_{\cV_{n+1}} \right) \notag\\
	& = \alpha^{-(n+1)}\left\langle (\mathrm{pr}_{\cV_n}^{\cV_{n+1}})^*\phi_n, \left[1 \otimes \mathrm{norm}^{\Delta_{n+1}}_{\Delta_n}\right]\xi_{n+1} \right\rangle_{\cV_{n+1}} \notag\\
	& = \alpha^{-(n+1)}\left\langle \phi_n, \left[ (\mathrm{pr}_{\cV_n}^{\cV_{n+1}})_* \otimes \mathrm{norm}^{\Delta_{n+1}}_{\Delta_n}\right]\xi_{n+1} \right\rangle_{\cV_n} \notag\\
	&= \alpha^{-(n+1)}\left\langle \phi_n,\  U_{p,1}'\cdot \xi_{n} \right\rangle_{\cV_n}\label{eq:norm relation apply}\\
	&=\alpha^{-(n+1)}\left\langle U_{p,1}\cdot\phi_n, \  \xi_{n} \right\rangle_{\cV_n} = \alpha^{-n}\big\langle \phi_n, \xi_{n} \big\rangle_{\cV_n}. \notag
\end{align}
In \eqref{eq:norm relation apply} we use \cref{thm:norm relation}, in the penultimate step use that $U_{p,1}$ and $U_{p,1}'$ are adjoint under $\langle-,-\rangle$, and the final equality uses $U_{p,1}\phi= \alpha\phi$. We conclude (i) after rescaling by $\operatorname{vol}(\cU^{(p)} \cap H)$.

Part (ii) is immediate as $\alpha^{-n}\cdot {}_c\Xi_n^{[a,j]}(\phi,\Phi^{(p)}) \in \cO_L[\Delta_n]$ for all $n$. 
\end{proof}

\subsection{Getting rid of \texorpdfstring{$c$}{c}} \label{sec:getting rid of c}

We fix $\Phi^{(p)}$ and $\cU^{(p)}$ and, for now, drop them from notation. We also fix $\phi$ as above, with $U_{p,1}\phi = \alpha\phi$, with $v_p(\alpha) = 0$.

Define `non-smoothed' analogues $\Xi_n^{[a,j]}$ and $\Xi^{[a,j]}$ of ${}_c\Xi_n^{[a,j]}$ and ${}_c\Xi^{[a,j]}$ by using $\xi_n^{[a,j]}$ instead of ${}_c \xi_n^{[a,j]}$. Since the non-smoothed Eisenstein classes are $\Qp$- rather than $\Zp$-valued, these are ostensibly only distributions in $L[\![\Zp^\times]\!]$. We now show that they actually lie in $\cO_L[\![\Zp^\times]\!]$.

\begin{proposition}\label{prop:get rid of c}
\begin{itemize}\setlength{\itemsep}{0pt}
	\item[(i)]	We have, for each $n$,
	\[
		{}_c \Xi_n^{[a,j]}(\phi) = \big(c^2 - c^{-j}\eta_2(c)^{-1}\omega_{\Pi}(c)^{-1} [c^{-1}]\big) \Xi^{[a,j]}_n(\phi).\]
	\item[(ii)] If $\eta_2$ is chosen so that $\eta_2 \omega_{\Pi}$ is not congruent mod $p$ to any character of $p$-power conductor, then $\Xi^{[a,j]}(\phi) \in \cO_L[\![\Zp^\times]\!]$, i.e.\ it is a measure.
\end{itemize}
\end{proposition}
\begin{proof}
(i) This follows from the formula \eqref{eq:c-dependence} above relating ${}_c \xi_n^{[a,j]}$ and the non-$c$-version, since the transpose of $\langle c \rangle$ with respect to Poincar\'e duality will be $\langle c^{-1} \rangle$, which acts on any vector in $\Pi$ as multiplication by $\omega_{\Pi}(c)^{-1}$.

(ii)  Choose $c$ with $c = 1 \bmod p^n$ and  $(\omega_{\Pi} \eta_2)(c) \ne 1 \bmod p$. It follows that the factor relating ${}_c \Xi_n^{[a,j]}$ and $\Xi^{[a,j]}_n$ is invertible in $\Zp[\Delta_n]$. Thus $\Xi^{[a,j]}_n$ is integral for all $n$ and $\Xi^{[a,j]} $ is a measure.
\end{proof}

The auxiliary character $\eta_2$ was introduced solely for \cref{prop:get rid of c}. We must now carry it through all notation, but its only contribution in the rest of the paper will be to the periods $\Omega_\Pi^\pm$.

\subsection{The Manin relations: compatibility in varying \texorpdfstring{$j$}{j}}\label{sec:vary j}
\begin{theorem}\label{theorem:compatibility}
For $0 \leq j \leq a$ and $f : \Zp^\times \to \overline{\Q}_p$, we have a compatibility
\[
\int_{\Zp^\times} f(x) \cdot \mathrm{d}\Xi^{[a,j]}(\phi)(x) = \int_{\Zp^\times} f(x)x^{j} \cdot \mathrm{d} \Xi^{[a,0]}(\phi)(x).
\]

\end{theorem}

\begin{proof}
We shall prove this by adapting the methods of \cite{spherical_II}. However, the result we seek is not quite a direct consequence of the main theorem of \emph{op.cit.}, since the weight $\lambda = (a, 0, -a)$ of our coefficient sheaf is not induced from a 1-dimensional character of the Levi $L_1$ of $P_1$, and our level group does not always contain $L_1(\Zp)$. So we shall briefly indicate the modifications needed to the theory of \emph{op.cit.} in order to prove the theorem at hand.

Since our classes are built up from the norm-compatible families ${}_c \xi^{[a, j]}_n$, it suffices to prove a compatibility for ${}_c \xi^{[a, j]}_n$ and ${}_c \xi^{[a, 0]}_n$ modulo $p^n$, regarded as elements of the module \[\h^3(Y^{\GL_3}(\cV_n), \sV_\lambda(\Z/p^n)) \otimes (\Z/p^n)[\Delta_n].\] More precisely, we shall prove the following: the map $\operatorname{mom}^j_{\Delta_n}: (\Z/p^n)[\Delta_n] \to (\Z/p^n)[\Delta_n]$, defined on group elements by sending $[a]$ to $a^j[a]$, maps ${}_c \xi^{[a, 0]}_n$ to ${}_c \xi^{[a, j]}_n$.

We have ${}_c \xi^{[a, j]}_n = [(\tau^n)_\star \otimes 1]\left({}_c z^{[a, j]}_n\right)$, where $\tau$ denotes $\sdthree{p}{1}{1}$. The action of $\tau_\star$ on our coefficients is given by a map of $\cU_n$-representations
\[ (\tau^n)_{\sharp}: V_\lambda(\Zp) \to \tau_n^\star\left( V_\lambda(\Zp)\right)\]
which is just $p^{an}$ times the action of $\tau^{-n}$ on $V_\lambda(\Qp)$; the factor $p^{an}$ is the image of $\tau^n$ under the the character $\lambda$ (compare Definition 2.5.1 of \cite{spherical_II}).

Since $V_\lambda(\Zp)$ is an admissible lattice, it is a direct sum of eigenspaces for the diagonal torus, corresponding to the weights of $V_\lambda$. The map $(\tau^n)_\sharp$ acts on each of these by a non-negative power of $p^n$; so on the mod $p^n$ coefficient sheaf $\sV_\lambda(\Z/p^n)$, it is zero on all eigenspaces except the highest relative weight space for the torus $Z(L_1)$. So it suffices to prove that the classes ${}_c z^{[a, j]}_n$ and $\mom^j_{\Delta_n}\left( {}_c z^{[a, 0]}_n\right)$ have the same image in the highest relative weight space.

Exactly as in \cite{spherical_II}, the homomorphism $\operatorname{br}^{[a, j]}$ is determined by the image of the highest-weight vector of $V_{(j, 0; -j)}^H$; let us call this vector $f^{a, j}$. If we model $V_\lambda$ as a space of functions on $\bar{N}_1 \backslash G$ taking values in the weight $\lambda$ representation of $L_1$, and satisfying $f(\bar{n} \ell g) = \ell \cdot f(g)$, then the value $f(u)$ determines $f^{a, j}$ uniquely on $\overline{P}_1 u^{-1} Q_H^0$, which is open in $G$; so $f^{a, j}$ is uniquely determined by $f^{a, j}(u^{-1})$, and this value lies in a one-dimensional subspace of the weight $\lambda$ $L_1$-representation which is independent of $j$ (in fact it is the \emph{lowest} weight space). So we may suppose that the $f^{a, j}$ are normalised so that $f^{a, j}(u^{-1})$ is independent of $j$.

Then, exactly as in \cite{spherical_II}, we obtain a compatibility between moment maps for $G$ and for $H$ under pushforward, where $H = \GL_2 \times \GL_1$ and $G = \GL_3 \times \GL_1$. The moment map for $G$ gives the twisting operator $\operatorname{mom}^j_{\Delta_n}$; and the twisting operator on $H$ maps the weight 0 Eisenstein class to its weight $j$ equivalent, by Kings' theory of $\Lambda$-adic Eisenstein classes (see \cite{KLZ17} for a summary in our present notations).
\end{proof}

\section{Values of the distribution as global Rankin--Selberg integrals}\label{sec:measure to integral}

Let $\Pi$ be a RACAR of $\GL_3(\A)$ of weight $\lambda = (a,0,-a)$, and let $\sigma_p$ be a finite-slope unramified $P_1$-refinement (as in \S\ref{sec:ordinarity}). Recall $\sigma_p$ is determined by $\alpha_p = \sigma_p(p) \neq 0$. We now apply the previous section to attach distributions to appropriate eigenforms in $\Pi_f$, and show, in \cref{prop:measure to integral}, that they compute global $\GL_3 \times \GL_2$ Rankin--Selberg integrals for $\Pi$ with Eisenstein series.

 Let $\varphi\subf = \otimes \varphi_\ell \in \Pi\subf$ such that:
\begin{itemize}\setlength{\itemsep}{0pt}
\item[(a)] $\varphi_p$ is fixed by $\cU_{1,p}^{(P_1)}(p^{r(\Pi,\alpha_p)})$ and $U_{p,1}\varphi_p = p^{a+1}\alpha_p \cdot \varphi_p$ (cf.\ \cref{prop:P_1 conductor}); and
\item[(b)] $W_{\varphi\subf}$ is algebraic, so that we get a class $\phi_{\varphi\subf} \defeq \phi_\Pi(W_{\varphi\subf})/\Theta_\Pi \in \hc{2}(Y^{\GL_3},\sV_\lambda^\vee(\cO_L))$, for $L/\Qp$ sufficiently large and $\phi_\Pi, \Theta_\Pi$ as in \S\ref{sec:auto cohom classes}, normalised as in \S\ref{sec:integral periods}.
\end{itemize}

We fix tame data $\Phi^{(p)} \in \cS(\A\subf^{(p),2},\Z)$ and $\cU^{(p)} \subset \GL_3(\A\subf^{(p)})$ such that $\varphi\subf$ (resp.\ $\Phi^{(p)}$) is fixed by $\cU^{(p)}$ (resp.\ $\cU^{(p)}\cap H$). These choices fix $\Phi_{\mathrm{f},n}$ and $\cV_n$, but we otherwise drop them from notation throughout this section.

By \cref{prop:get rid of c}, attached to all this data, and a suitable fixed auxiliary character $\eta_2$, is a distribution $\Xi^{[a,j]}(\phi_{\varphi\subf}) \in L[\![\Zp^\times]\!]$; and when $\sigma_p$ is $P_1$-ordinary, $\Xi^{[a,j]}(\phi_{\varphi\subf}) \in \cO_L[\![\Zp^\times]\!]$ is a measure.

\subsection{Values of the distribution}\label{sec:first values}
Let $\eta$ be a Dirichlet character of conductor $p^{n_1}$, and let $n = \mathrm{max}(n_1,1)$. Consider $\eta$ as a character on $\Delta_n = (\Z/p^n\Z)^\times$, and lift it it a character on $\Zp^\times$ under the natural surjection. We integrate this against our distribution. This factors through $\Delta_n$, and writing $x$ for the variable on $\Zp^\times$, we have
\begin{align}\label{eq:measure 1}
\int_{\Zp^\times} \eta(x) \  \cdot \ &\mathrm{d}\Xi^{[a,j]}\big(\phi_{\varphi\subf}\big)(x) = 	(p^{a+1}\alpha_p)^{-n} \int_{\Delta_n} \eta(x)  \cdot \mathrm{d}\Xi_n^{[a,j]}\Big(\phi_{\varphi\subf}\Big)(x)\notag\\
&= (p^{a+1}\alpha_p)^{-n}\big[1 \otimes \eta\big] \left( \operatorname{vol}(\cU^{(p)} \cap H) \cdot \left\langle \left(\phi_{\varphi\subf}\right)_n,  \xi_{n}^{[a, j]} \right\rangle_{\cV_n}\right).
\end{align}
We consider this value in a finite extension of $\Qp$, but by construction (via the embedding $\iota$) it actually lies in a number field. For the rest of \S\ref{sec:measure to integral}, we will forget this algebraicity and, abusing notation, consider it just as a complex number.

To study \eqref{eq:measure 1}, we use the description and notation of \cref{rem:connected components}. Set $t = n$, and for $x \in \Delta_n$, let $i_{x} : \widetilde{Y}^H_{x} \hookrightarrow \widetilde{Y}^H(\cU_{n}^H)$ be the natural inclusion. By the remark, the coefficient of $[x]$ in $ \zeta_n^{[a,j]}$ is
\[
\sum_{y \in (\Z/D)^\times} \eta_2(y) \cdot p^{an} \ \tau_\star^n \circ u_\star \circ \iota_{\star} \circ \mathrm{br}^{[a,j]}_\star \circ \mathrm{tw}_{j} \circ i_{x,y}^\star (\mathrm{Eis}_{\Phi\subf^{n}}^{j}).
\]
Substituting, we see that \eqref{eq:measure 1} is equal to
\begin{align}
&= \tfrac{ p^{an}\operatorname{vol}(\cU^{(p)} \cap H)}{(p^{a+1}\alpha_p)^n}	\sum_{x \in \Delta_n} \eta(x)\\ &\hspace{10pt}\times \sum_{y \in (\Z/D)^\times} \eta_2(y)\cdot \left\langle \left(\phi_{\varphi\subf}\right)_n, \	\tau_\star^n  u_\star \iota_{\star} \mathrm{br}^{[a,j]}_\star  \mathrm{tw}_{j} i_{x,y}^\star(\mathrm{Eis}_{\Phi\subf^{2n}}^{j}) \right\rangle_{\cV_n}\notag\\
&= \tfrac{  \operatorname{vol}(\cU^{(p)} \cap H)}{p^{n}\alpha_p^n}	\sum_{x \in \Delta_n} \eta(x)\label{eq:measure 2}\\
&\hspace{10pt}\times \sum_{y \in (\Z/D)^\times} \eta_2(y) \cdot \left\langle \mathrm{br}^{[a,j],\star}  \iota^{\star}  u^\star \tau^{n,\star}[\left(\phi_{\varphi\subf}\right)_n], \ 	  \mathrm{tw}_{j}  i_{x,y}^\star(\mathrm{Eis}_{\Phi\subf^{2n}}^{j}) \right\rangle_{\widetilde{Y}^H_{x}}, \notag
\end{align}
via adjointness of pushforward and pullback under the Poincar\'e duality pairing.

\subsection{Cup products as integrals}\label{sec:cup product integral}
We now express these cup products as integrals. This follows a standard, if technical, procedure (cf.\ \cite[Lem.\ 1.3]{Mah98}, \cite[Lem.\ 3.2]{Mah00}). This technicality will manifest itself exclusively in a local zeta  integral at infinity.

\subsubsection{The differential form for $\Pi$}\label{sec:diff form pi}
Recall $\zeta_\infty \in \h^2\big(\mathfrak{gl}_3,K_{3,\infty}^\circ; \Pi_\infty \otimes V_{\lambda}^{\vee}(\C)\big)$ from \eqref{eq:omega infinity}.
We have
\[
\h^i\big(\mathfrak{gl}_3,K_{3,\infty}^\circ; \Pi_\infty \otimes V_{\lambda}^{\vee}\big) \subset \bigwedge{}^i\left(\mathfrak{gl}_3/\mathfrak{K}_{3,\infty}^\circ\right)^\vee \otimes \Pi_\infty \otimes V_\lambda^\vee(\C),
\]
where $\mathfrak{K}_{3,\infty}^\circ = \mathrm{Lie}(K_{3,\infty}^\circ)$. Choose (arbitrary) bases $\{\delta_1',...\delta_5'\}$ of $\left(\mathfrak{gl}_3/\mathfrak{K}_{3,\infty}^\circ\right)^\vee$ and $\{v_\alpha\}$ of $V_\lambda^\vee(\C)$. Given these choices, there exist vectors $\varphi_{\infty,r,s,\alpha} \in \Pi_\infty$  such that
\[
\zeta_\infty = \sum_{r,s = 1,...,5} \sum_{\beta} \big[\delta_r' \wedge \delta_s'\big] \otimes \varphi_{\infty,r,s,\alpha} \otimes v_\alpha.
\]
Letting $\varphi_{r,s,\alpha} \defeq \varphi_{\infty,r,s,\alpha} \otimes \varphi\subf \in \Pi$, we see the differential form associated to $\phi_{\varphi\subf} = \phi_\Pi(\varphi\subf)/\Theta_\Pi$ is
\begin{equation}\label{eq:diff form pi}
\Theta_\Pi^{-1}\sum_{r,s = 1,...,5} \sum_{\beta} \big[\delta_r' \wedge \delta_s'\big] \otimes \varphi_{r,s,\alpha} \otimes v_\alpha \in \bigwedge{}^i\left(\mathfrak{gl}_3/\mathfrak{K}_{3,\infty}^\circ\right)^\vee \otimes \Pi \otimes V_\lambda^\vee(\C).
\end{equation}

\subsubsection{The Eisenstein differential}
By \cref{cor:betti-eisenstein}, the differential on $Y^{\GL_2}(\cU_{n}^{H}\cap\GL_2)$ associated to $\mathrm{Eis}_{\Phif}^{j}$ is $-\cE_{\Phif}^{j+2}(g) \cdot (\mathrm{d}z)^{\otimes j} \cdot \mathrm{d}\tau$, recalling $\cE_{\Phif}^{j+2}(g) \in I(\|\cdot\|^{-1/2}, \|\cdot\|^{j+1/2})$ \eqref{eq:classical E j chi}.

For comparison with the $\GL_3$ setting, and with \cite{Mah98,Mah00}, it is convenient to rephrase this in the (inexplicit) language of \cref{sec:diff form pi}. Combining the above discussion, we see we may choose bases $\{\delta_1,\delta_2\}$ of $(\mathfrak{gl}_2/\mathfrak{K}_{2,\infty}^\circ)^\vee$ (corresponding to $\mathrm{d}\tau$) and $\{w_\beta^{[j]}\}$ of $V_{(0,-j)}^{\GL_2}(\C)$ (corresponding to $(\mathrm{d}z)^{\otimes j}$) such that as a differential form on $Y^{\GL_2}$, we may describe $\mathrm{Eis}_{\Phif}^{j}$ as
\begin{equation}\label{eq:diff form Eis}
\Big[\delta_1 + \delta_2\Big] \otimes \cE^{j+2}_{\Phif} \otimes \Big[\sum_\beta w_\beta^{[j]}\Big] \in \left(\mathfrak{gl}_2/\mathfrak{K}_{2,\infty}^\circ\right)^\vee \otimes I\Big(\|\cdot\|^{-\tfrac{1}{2}}, \|\cdot\|^{j+\tfrac{1}{2}}\Big) \otimes V_{(0,-j)}^{\GL_2}(\C)
\end{equation}
(where $\cE_{\Phif}^{j+2}(g) \in I \big(\|\cdot\|^{-1/2}, \|\cdot\|^{j+1/2}\big) $ by \cref{cor:compatibility} and \eqref{eq:I_j(chi)}, with $\eta = 1$).

\subsubsection{The cup product on components}
We pass to $H$. We identify the basis elements $\{\delta_1,\delta_2\}$ with their image under the natural pullback $\left(\mathfrak{gl}_2/\mathfrak{K}_{2,\infty}^\circ\right)^\vee \to \left( \mathrm{Lie}(H(\R))/\mathrm{Lie}(K_{H,\infty}^\circ \iota^{-1}(Z_{G,\infty}^\circ))\right)^\vee$, and then extend to a basis $\{\delta_1,\delta_2,\delta_3\}$ of this latter space. By construction, the pullback of $\mathrm{Eis}_{\Phif}^{j}$ to $\widetilde{Y}^H$ corresponds to a differential in which only $\delta_1$ and $\delta_2$ appear.

Since the basis $\{\delta_i'\}$ of $\left(\mathfrak{gl}_3/\mathfrak{K}_{3,\infty}^\circ\right)^\vee$ was arbitrary, we may rescale so that under the map
\[
\iota^\ast : \left(\mathfrak{gl}_3/\mathfrak{K}_{3,\infty}^\circ\right)^\vee \longrightarrow \left( \mathrm{Lie}(H(\R))/\mathrm{Lie}(K_{H,\infty}^\circ \iota^{-1}(Z_{G,\infty}^\circ))\right)^\vee,
\]
we have $\iota^\ast(\delta_i') = \delta_i$ for $i = 1,2,3$ and $\iota^\ast(\delta_i') = 0$ for $i = 4,5$ (cf.\ \cite[p.277]{Mah00}).

Recall $\langle-,-\rangle_{a,j}$ from \eqref{eq:branching pairing}. By definition,
\begin{multline*}
\Big\langle \mathrm{br}^{[a,j],\star}  \iota^{\star}  u^\star  \tau^{n,\star}[\left(\phi_{\varphi\subf}\right)_n], \ 	  \mathrm{tw}_{j}  i_{x}^\star(\mathrm{Eis}_{\Phi\subf^{2n}}^{j}) \Big\rangle_{\widetilde{Y}^H_{x,y}}\\
= \int_{\widetilde{Y}^H_{x,y}} \bigg[ \Theta_\Pi^{-1}\sum_{r,s,t = 1}^3 \sum_{\alpha,\beta} \left\langle v_\alpha, w_\beta^{[j]}\right\rangle_{a,j} \cdot \varphi_{r,s,\alpha}\Big(\iota(h)u\tau^{n}\Big)  \\
\times \cE_{\Phif}^{j+2} \big(\mathrm{pr}_1(h)\big) \cdot \|\nu_1^{-j}(h)\|\cdot  \big[\delta_r\wedge \delta_s \wedge \delta_t\big]\bigg].
\end{multline*}
Since only $\delta_1$ and $\delta_2$ arise in the Eisenstein differential, the only non-zero terms in the sum over $r,s,t$ are $rst = 132,231$. As $\widetilde{\cH}_H$ is 3-dimensional, we identify $\delta_1\wedge \delta_2\wedge \delta_3$ with a fixed choice of Haar measure $\mathrm{d}h_\infty$ on $H(\R)$. The expression becomes
\begin{equation}\label{eq:cup product component}
= \Theta_\Pi^{-1}\sum_{r,s,t} \varepsilon_{rst} \sum_{\alpha,\beta} \left\langle v_\alpha, w_\beta^{[j]}\right\rangle_{a,j}\int_{\widetilde{Y}_{x,y}^H}  \varphi_{r,s,\alpha}\Big(\iota(h)u\tau^{n}\Big)   \cE_{\Phif}^{j+2} \big(\mathrm{pr}_1(h)\big) \cdot \|\nu_1^{-j}(h)\| \cdot \mathrm{d}h,
\end{equation}
where $\varepsilon_{231} = 1$, $\varepsilon_{132} = -1$, and $\varepsilon_{rst} = 0$ otherwise; and $\mathrm{d}h = \mathrm{d}h\subf\mathrm{d}h_\infty$, recalling $\mathrm{d}h\subf$ from \cref{sec:pairing with cusp form}.

\subsection{Passing to the Rankin--Selberg integral}
By definition, on $\widetilde{Y}_{x,y}^{H}$ we have $x = \nu_1(h), y =\nu_2(h)$. Moreover, the characters $\eta$ on $\Zp^\times$ and $\eta_2$ on $(\Z/D)^\times$ lift to the (finite order) Hecke characters $\etahat$  and $\widehat{\eta_2}$ (as in \cref{sect:dirichlet}). Combining \eqref{eq:measure 1}, \eqref{eq:measure 2} and \eqref{eq:cup product component} now gives
\begin{multline*}
\int_{\Zp^\times} \eta(x)  \cdot \mathrm{d}\Xi^{[a,j]}(\phi_{\varphi\subf})(x,y) =
\tfrac{\operatorname{vol}(\cU^{(p)} \cap H)}{p^{n}\alpha_{p}^n\Theta_\Pi} \sum_{r,s,t} \varepsilon_{rst} \sum_{\alpha,\beta} \left\langle v_\alpha, w_\beta^{[j]}\right\rangle_{a,j}\\
\int_{\widetilde{Y}^H(\cU_{n}^H)}  \varphi_{r,s,\alpha}\Big(\iota(h)u\tau^{n}\Big)   \cE_{\Phif}^{j+2} \big(\mathrm{pr}_1(h)\big) \cdot \etahat\|\cdot\|^{-j}(\nu_1(h))\cdot  \widehat{\eta_2}(\nu_2(h)) \mathrm{d}h.
\end{multline*}
Up to renormalising our choices at infinity (and hence $\mathrm{d}h_\infty$), we may take $\mathrm{vol}(K_{H,\infty}^\circ) = 1$. Also cancelling the $\mathrm{vol}(\cU^{(p)}\cap H)$ -- which, at $p$, introduces $\mathrm{vol}(\cU_{n,p}^H)$ -- we thus obtain
\begin{multline*}
= \tfrac{\operatorname{vol}(\cU_{n,p}^H)^{-1}}{{p^{n}\alpha_{p}^n} \Theta_\Pi} \sum_{r,s,t} \varepsilon_{rst} \sum_{\alpha,\beta} \langle v_\alpha, w_\beta^{[j]}\rangle_{a,j}
\int_{H(\Q)\backslash H(\A)/\R_{>0}}  \varphi_{r,s,\alpha}\Big(\iota(h)u\tau^{n}\Big)   \\ \times \cE_{\Phif}^{j+2} \big(\mathrm{pr}_1(h)\big) \cdot \etahat\|\cdot\|^{-j}(\nu_1(h)) \cdot \widehat{\eta_2}(\nu_2(h))   \mathrm{d}h,
\end{multline*}
where $z \in \R_{>0}$ embeds as $[\smallmatrd{z}{}{}{z}, z]$. Write $h = (\gamma, z) \in \GL_2(\A) \times \GL_1(\A)$, so $\nu_1(h) = \tfrac{\det(\gamma)}{z}$, and
\begin{align*}
&= \tfrac{\operatorname{vol}(\cU_{n,p}^H)^{-1}}{{p^{n}\alpha_{p}^n} \Theta_\Pi} \sum_{r,s,t} \varepsilon_{rst} \sum_{\alpha,\beta} \langle v_\alpha, w_\beta^{[j]}\rangle_{a,j}
\int_{H(\Q)\backslash H(\A)/\R_{>0}}  \varphi_{r,s,\alpha}\Big(\iota(\gamma,z)u\tau^{n}\Big)   \\
&\hspace{120pt}\times \cE_{\Phif}^{j+2} \big(\gamma \big) \cdot \etahat\|\cdot\|^{-j}\left(\tfrac{\det(\gamma)}{z}\right)\cdot \widehat{\eta_2}(z)  \mathrm{d}(\gamma,z)\\
&= \tfrac{\operatorname{vol}(\cU_{n,p}^H)^{-1}}{p^{n}\alpha_{p}^n \Theta_\Pi} \sum_{r,s,t} \varepsilon_{rst} \sum_{\alpha,\beta} \langle v_\alpha, w_\beta^{[j]}\rangle_{a,j}
\int_{H(\Q)\backslash H(\A)/\R_{>0}}  \varphi_{r,s,\alpha} \Big(\iota\left(\tfrac{\gamma}{z},1\right)u\tau^{n}\Big)\omegahat_\Pi(z)   \\
&\hspace{120pt}\times \cE_{\Phif}^{j+2} \big(\gamma \big) \cdot \etahat\|\cdot\|^{-j}\Big(\det\left(\tfrac{\gamma}{z}\right) \Big) \cdot \etahat\widehat{\eta_2}\|\cdot\|^{-j}(z)  \mathrm{d}(\gamma,z),
\end{align*}
noting $\tfrac{\det(\gamma)}{z} = z\cdot \det\big(\tfrac{\gamma}{z}\big)$. Now make the change of variables $g = \gamma/z$: this identifies $H(\A)$ with itself, but now $\R_{>0}$ embeds as $z \mapsto [\smallmatrd{1}{}{}{1}, z]$, so the domain becomes
\[
\GL_2(\Q)\backslash \GL_2(\A) \times \Q^\times\backslash\A^\times/\R_{>0} = \big[\GL_2(\Q)\backslash \GL_2(\A)\big] \times \widehat{\Z}^{\times}.
\]
By translation-invariance of Haar measures, we have $dg = d\gamma$. As $\|z\|^{-j} = 1$ for $z \in \widehat{\Z}^\times$, we get
\begin{align}
&= \tfrac{\operatorname{vol}(\cU_{n,p}^H)^{-1}}{p^{n}\alpha_{p}^n \Theta_\Pi} \sum_{r,s,t} \varepsilon_{rst} \sum_{\alpha,\beta} \langle v_\alpha, w_\beta^{[j]}\rangle_{a,j}
\int_{[\GL_2(\Q)\backslash \GL_2(\A)] \times \widehat{\Z}^{\times}}  \varphi_{r,s,\alpha} \Big(\iota\left(g,1\right)u\tau^{n}\Big)  \notag  \\
&\hspace{120pt}\times \cE_{\Phif}^{j+2} \big(g z\big)  \cdot \omegahat_\Pi\etahat\widehat{\eta_2}(z)  \cdot \etahat\|\cdot\|^{-j}\big(\det\left(g\right) \big) \mathrm{d}(g,z),\label{eq:measure 3}
\end{align}

\begin{lemma}\label{lem:twist Eisenstein}
Let $k \geq 0$ and $\chi$ be a Dirichlet character. Then we have
\[
\int_{\widehat{\Z}^\times} \cE_{\Phif}^{j+2}(gz)\chihat(z)\mathrm{d}z = E_{\Phi}^{j+2,\chi}(g) = \|\det g\|^{\tfrac{j}{2}} E_\Phi\big(g; \chihat, -\tfrac{j}{2}\big).
\]
\end{lemma}
\begin{proof}
By definition, for $z \in \widehat{\Z}^\times$ we have $\cE_{\Phif}^{j+2}(gz) = \cE_{z\cdot \Phif}^{j+2}(g)$, where $z$ acts as $\smallmatrd{z}{}{}{z}$. Substituting the definitions, we see that the left-hand side is $\cE_{R_\chi(\Phif)}^{j+2}(g) = \cE_{\Phif}^{j+2,\chi}(g)$, where the equality follows from \cref{def:classical eisenstein}(iii). We conclude by \eqref{eq:compatibility} (giving the first equality) and \eqref{eq:E j chi} (the second).
\end{proof}

Taking $\chi = \omega_\Pi\eta\eta_2$, and collapsing \eqref{eq:measure 3} with \cref{lem:twist Eisenstein}, we conclude
\begin{multline*}
\int_{\Zp^\times} \eta(x)  \cdot \mathrm{d}\Xi^{[a,j]}(\phi_{\varphi\subf})(x) = \tfrac{\operatorname{vol}(\cU_{n,p}^H)^{-1}}{p^{n}\alpha_{p}^n \Theta_\Pi} \sum_{r,s,t} \varepsilon_{rst} \sum_{\alpha,\beta} \langle v_\alpha, w_\beta^{[j]}\rangle_{a,j}\\
\int_{\GL_2(\Q)\backslash \GL_2(\A)}  \varphi_{r,s,\alpha} \Big(\iota(g,1)u\tau^{n}\Big)
\cdot E_{\Phi}\big(g; \omegahat_\Pi\etahat\widehat{\eta_2}, -\tfrac{j}{2}\big) \cdot  \etahat(\det g) \|\det g\|^{-\tfrac{j}{2}}  \mathrm{d}g.
\end{multline*}

\begin{definition}\label{def:global zeta integral}
Let $\Pi$ be a unitary automorphic representation of $\GL_3(\A)$, and $\chi_1$ and $\chi_2$ be Dirichlet characters. For $\varphi \in \Pi$ and $\Phi \in \cS(\A^2,\C)$, define
\begin{multline*}
	\cZ\big(\varphi, \Phi; \chi_1,\chi_2,s_1,s_2\big) \defeq \\
	\int_{\GL_2(\Q) \backslash \GL_2(\A)} \varphi\Big(\iota(g,1)
	\Big)\cdot E_\Phi(g; \chihat_2, s_2)\cdot \chihat_1(\det g)\|\det g\|^{s_1-\tfrac{1}{2}} \mathrm{d}g.
\end{multline*}
\end{definition}

Summarising all of the above, we have shown:

\begin{proposition}\label{prop:measure to integral}
Let $\eta$ be a Dirichlet character of conductor $p^n$. Then
\begin{multline*}
	\int_{\Zp^\times} \eta(x)  \cdot \mathrm{d}\Xi^{[a,j]}(\phi_{\varphi\subf})(x) =  \frac{\operatorname{vol}(\cU_{n,p}^H)^{-1}}{p^{n}\alpha_{p}^n \Theta_\Pi} \sum_{r,s,t} \varepsilon_{rst}\\
	\sum_{\alpha,\beta} \langle v_\alpha, w_\beta^{[j]}\rangle_{a,j} \cZ\Big(u\tau^n\cdot \varphi_{r,s,\alpha}, \Phi; \eta, \omega_\Pi\eta\eta_2, \tfrac{1-j}{2},-\tfrac{j}{2}\Big).
\end{multline*}
\end{proposition}

Exactly as in \cite[Prop.\ 3.1]{Mah98} (following the proof of \cite[Prop.\ 3.3]{JS_EulerProducts_II}), for $\Re(s) \gg 0$ the integral $\cZ(-)$ has an Eulerian factorisation in terms of local zeta integrals, namely
\begin{equation}\label{eq:global to local}
\cZ\big(\varphi, \Phi; \chi_1,\chi_2,s_1,s_2\big) = \prod_v Z_v\big(W_{\varphi,v}, \Phi_v; \chihat_{1,v},\chihat_{2,v},s_1,s_2\big).
\end{equation}
We will define and study the local integrals $Z_v(-)$ in the next section (\cref{sec:local zeta integrals}).

We emphasise that in \cref{prop:measure to integral}, $\varphi\subf$ is the finite part of each the $\varphi_{r,s,\alpha}$ (i.e.\ the $\varphi_{r,s,\alpha}$ differ \emph{only} at infinity). This will allow us to move both sums into the local zeta integral at infinity.

%%%%%%%%%%%%%%%%%%%%%%%%%%%%%%%%%%%%%%%%%%%%%%%%%%%%%%%%%%%%%

\section{Local zeta integrals}\label{sec:local zeta integrals}

Throughout this section we work locally at a place $v$ of $\Q$, and largely drop $v$ from notation.

\subsection{\texorpdfstring{$\GL_2$}{GL(2)} principal series}
Given $\Phi \in \cS(\Q_v^2,\C)$, continuous $\chi : \Q_v^\times \to \C^\times$, and $g \in \GL_2(\Q_v)$, we define a local Godement--Siegel section $f_\Phi(g;\chi,s)$, exactly as in \cref{sec:Eis-GL2} (but with the integral over $\Q_v^\times$ rather than $\A^\times$). We let $W_\Phi(g;\chi,s)$ be its Whittaker transform as in \cite[\S 8.1]{LPSZ19}, defined by analytic continuation of an integral convergent for $\Re(s) \gg 0$; the function $W_{\Phi}$ is entire, while $f_{\Phi}$ has the same poles as $\Phi(0, 0) \cdot L(\chi, 2s)$.
Note the maps $\Phi \to f_{\Phi}, \Phi \to W_\Phi$ have the equivariance property
\begin{align*}
f_{h\Phi}(g; \chi, s) &= |\det h|^{-s} f_{\Phi}(gh; \chi, s),\\
W_{h\Phi}(g; \chi, s) &= |\det h|^{-s} W_{\Phi}(gh; \chi, s).
\end{align*}

For a fixed $s$ where the (local) $L$-factor $L(\chi, 2s) \ne \infty$, the space of functions $f_{\Phi}(-; \chi, s)$ for varying $\Phi$ is exactly the induced representation $I\big(|\cdot|^{s-\tfrac{1}{2}}, |\cdot|^{\tfrac{1}{2}-s}\chi^{-1}\big)$, while the space of $W_{\Phi}(-; \chi, s)$ is the Whittaker model of this representation (with respect to $\psi^{-1}$, not $\psi$).

If $L(\chi, 2s) = \infty$, the $f_{\Phi}$ may have poles, and $I\big(|\cdot|^{s-\tfrac{1}{2}}, |\cdot|^{\tfrac{1}{2}-s}\chi^{-1}\big)$ is non-generic: it has a 1-dimensional subrepresentation spanned by the residues of the $f_{\Phi}$, and the $W_{\Phi}$ lie in the Whittaker model of $I\big( |\cdot|^{\tfrac{1}{2}-s}\chi^{-1}, |\cdot|^{s-\tfrac{1}{2}}\big)$ instead.

\subsection{A two-parameter zeta integral}\label{sec:two-parameter}

For \S\S\ref{sec:two-parameter} and \ref{sec:second zeta integral} only, let $\pi$ be any generic representation of $\GL_3(\Q_v)$, and let $\chi_1, \chi_2$ be two smooth characters of $\Q_v^\times$.

\begin{definition}\label{def:local zeta Z}
For $s_1, s_2$ complex numbers, $W \in \cW_{\psi}(\pi)$, and $\Phi \in \cS(\Q_v^2,\C)$, we set
\[Z(W, \Phi; \chi_1, \chi_2, s_1, s_2) = 
	\int_{(N_2 \backslash \GL_2)(\Q_v) } W(\iota(g, 1)) W_{\Phi}(g; \chi_2, s_2) \chi_1(\det g)|\det g|^{s_1 - \tfrac{1}{2}}\ \mathrm{d}g,
\]
which is convergent for $\Re(s_1) \gg 0$ (for fixed $s_2$) and has meromorphic continuation to all $s_1$ and $s_2$, as a rational function in $\ell^{(\pm s_1 \pm s_2)}$ if $v = \ell$ is a finite place.
\end{definition}

\begin{theorem}[Jacquet, Piatetski-Shapiro, Shalika] \label{thm:jpss}
	\begin{itemize}
\item[(i)] The function
\[
\widetilde{Z}(W, \Phi; \chi_1, \chi_2, s_1, s_2) \defeq \frac{Z(W, \Phi; \chi_1, \chi_2, s_1, s_2)}{L(\pi \times \chi_1, s_1 + s_2 - \tfrac{1}{2}) L(\pi \times \chi_1 \chi_2^{-1}, s_1 - s_2 + \tfrac{1}{2})}
\]
is entire as a function of the $s_i$, and is a polynomial in $\ell^{(\pm s_1 \pm s_2)}$ if $v = \ell$ is a finite place. 

\item[(ii)] The ideal generated by these functions for varying $(W, \Phi)$ is the unit ideal. In particular, if $v$ is a finite place, there exist finite collections $\{W_i\}_{i\in I}$ and $\{\Phi_i\}_{i\in I}$ defined over $\overline{\Q}$ such that $\sum_i\widetilde{Z}(W_i,\Phi_i; \chi_1,\chi_2,s_1,s_2) = 1$ for all $s_1,s_2$. 

\item[(iii)]If $\pi$ and the $\chi_i$ are unramified, and $W$ and $\Phi$ are the normalised spherical data, then $\widetilde{Z}(W,\Phi; \chi_1,\chi_2,s_1,s_2) = 1$ for all $s_1,s_2$.
\end{itemize}
\end{theorem}

Note that for $W_i$ and $\Phi_i$ as in the theorem, we have \[
\sum_i \widetilde{Z}(W_i,\Phi_i;\chi_1\theta_1,\chi_2\theta_2,s_1,s_2) = 1
\]
for all unramified characters $\theta_1,\theta_2$ (since we can move such $\theta_j$ into the $s_j$).

\begin{proof}
We know that, for each $s_2$, the functions $W_{\Phi}(-; \chi_2, s_2)$ for varying $\Phi$ form the Whittaker model of the representation $\Pi' = I(|\cdot|^{s_2 - 1/2}, |\cdot|^{1/2 - s_2} \chi_2^{-1})$; and evidently the $W(g)\chi_1(\det g)$ span the Whittaker model of $\pi \times \chi_1$. So, by the results of \cite{JPSS-RankinSelberg}, %Theorem 2.7.
 the greatest common divisor of the $Z(-;\chi_1, \chi_2, s_1, s_2)$, as functions of $s_1$, is the $L$-factor
\[
L(\pi \times \Pi' \times \chi_1, s_1) = L(\pi \times \chi_1, s_1 + s_2 - \tfrac{1}{2}) L(\pi \times \chi_1\chi_2^{-1}, s_1 - s_2 + \tfrac{1}{2}),
\]
where the latter equality follows from the compatibility of Rankin--Selberg $L$-factors with parabolic induction (also proved in \emph{op.cit.}).

So, for each fixed $s_2$, the normalised integrals $\widetilde{Z}(-; \chi_1, \chi_2, s_1, s_2)$ generate the unit ideal. As these functions are meromorphic in $s_1 \pm s_2$, they in fact generate the unit ideal in $\C[\ell^{\pm s_1 \pm s_2}]$.
\end{proof}

Note that
\[ Z(W, \Phi; \chi_1, \chi_2, s_1, 1-s_2) = Z(W, \hat\Phi; \chi_1 \chi_2^{-1}, \chi_2^{-1}, s_1, s_2). \]
This is immediate from the functional equation of $W_{\Phi}$, cf.~\cite[Eq. (8.1)]{LPSZ19}. The denominator in the definition of $\widetilde{Z}$ is the same for both sides (the factors get swapped) so this relation also holds for $\widetilde{Z}$ in place of $Z$. There is also a functional equation in $s_1$ (for fixed $s_2$), but this is more complicated to state, and we shall not use it here.

\subsection{A second zeta-integral}\label{sec:second zeta integral}

We consider a second zeta-integral, studied in \cite{LSZ-unitary} (following a series of earlier works). In this section, we shall assume $v$ is a finite place.

\begin{definition}\label{def:local zeta Y}
For $(\pi, \chi_1, \chi_2, W, \Phi)$ as before, we define
\begin{multline*}
	Y(W, \Phi; \chi_1, \chi_2, s_1, s_2) = \\
	\int_{(N_2 \backslash \GL_2)(\Q_v)}
	W\left[ \smallthreemat{1}{}{}{}{}{1}{}{-1}{} \iota(g, 1) \right] f_{\Phi}(g; \chi_2, s_2) |\det g|^{s_1 - \tfrac{1}{2}}\chi_1(\det g)\, \mathrm{d}g.
\end{multline*}
\end{definition}

This integral converges for $\Re(s_2) \gg 0$ for any fixed $s_1$ [sic!] and extends to a meromorphic function of $s_1$ and $s_2$, which is a a rational function in $\ell^{\pm s_1 \pm s_2}$ if $v = \ell$ is finite.

\subsubsection{Relation to the $Z$ integral}  We will prove the following theorem in an appendix below.

\begin{theorem}
\label{thm:lfactor}
We have the identity
\[ Y(W, \Phi; \chi_1, \chi_2, s_1, s_2) = \gamma(\pi \times \chi_1/\chi_2, s_1 - s_2 + \tfrac{1}{2})\cdot  Z(W, \Phi; \chi_1, \chi_2, s_1, s_2), \]
where
\[ \gamma(\pi, s) = \frac{\varepsilon(\pi, s) L(\pi^\vee, 1-s)}{L(\pi, s)}\]
is the local $\gamma$-factor. In particular, the greatest common divisor of the $Y(-)$ as $(W, \Phi)$ vary is the principal ideal generated by the product of $L$-factors
\[
L(\pi \times \chi_1, s_1 + s_2 - \tfrac{1}{2}) \cdot L(\pi^\vee \times \chi_1^{-1} \chi_2, s_2 - s_1 + \tfrac{1}{2}).
\]
\end{theorem}

\subsubsection{Torus integrals} We now give an alternative, more ``computable'' formula for $Y(-)$.

\begin{definition}\label{def:y}
For $W \in \cW(\pi)$, define a function on $\GL_2(\Q_v)$ by
\begin{multline*}
	y(W; \chi_1, \chi_2, s_1, s_2)(g) = \chi_1(\det g)|\det g|^{s_1 - \tfrac{1}{2}} \times \\
	\int_{(\Q_v^\times)^2} W\left[\sdthree{x}{1}{y^{-1}}\smallthreemat{1}{}{} {}{}{1} {}{-1}{}\iota(g, 1)\right] \cdot |x|^{s_1 + s_2 - \tfrac{3}{2}} \chi_1(x)\cdot  |y|^{s_2 - s_1-\tfrac{1}{2}} \tfrac{\chi_2}{\chi_1}(y) \ \mathrm{d}^\times x\,  \mathrm{d}^\times y.\end{multline*}
\end{definition}

The integral converges for $\Re(s_2) \gg 0$ (for fixed $s_1$), as before; more precisely, it converges in some quadrant $\Re(s_2 - s_1) > C, \Re(s_2 + s_1) > C$. A computation shows that it transforms as an element of $I(|\cdot|^{1/2-s_2}, |\cdot|^{s_2 - 1/2} \chi_2)$, which is the dual space of the representation in which $f_{\Phi}(-;\chi_2, s_2)$ lies, and the duality pairing recovers $Y$: that is, we have
\begin{equation}\label{eq:Y and y}
Y(W, \Phi; \chi_1, \chi_2, s_1, s_2) = \left\langle y(W; \chi_1, \chi_2, s_1, s_2), f_{\Phi}(-;\chi_2, s_2)\right\rangle.
\end{equation}

\subsection{Parahoric level test data}\label{sec:test data at p}

We now take $\pi = \Pi_p$, and evaluate (in \cref{thm:local zeta p}) the local zeta-integral directly for certain specific test data; the integral we consider here is a local factor of the global integral in \cref{prop:measure to integral}). We suppose that $v$ is a finite place, and denote it by $p$ (for compatibility with our applications below).

Suppose $\Pi$ has an unramified $P_1$-refinement $\sigma_p$, so there is an irreducible $\GL_2$-representation $\sigma_p'$ such that $\sigma_p \times \sigma_p' \hookrightarrow J_{P_1}(\Pi_p)$. Recall $\alpha_p = \sigma_p(p)$, and from \cref{prop:P_1 conductor} that $r(\Pi,\alpha_p)$ is the conductor of $\sigma_p'$. 

Henceforth, we fix the following characters and test data.  

\begin{notation}\label{not:test data p} \
\begin{itemize}\setlength{\itemsep}{0pt}
	\item Let  $\chi_1 = \etahat_{1, p}$, where $\eta_1$ is a Dirichlet character of conductor $p^{n_1}$ for $n_1 \in \Z_{\geq 0}$ (so that $\chi_1(p) = 1$ and $\chi_1|_{\Zp^\times} = \eta_1^{-1}$). Let $n = \max(1,n_1) \ge 1$.
	\item Let $\chi_2 = \omegahat_{\Pi, p} \etahat_{1, p}\etahat_{2, p}$, where $\eta_2$ is the auxiliary Dirichlet character from \cref{not:eta_2}.
	\item  Let $W^{\alpha} \in \Pi_p$ be the $P_1$-stabilised newvector of $\Poref$ as defined in \cref{sect:refnewforms}, so that $W^{\alpha}$ is stabilized by $\cU_{1,p}^{(P_1)}(p^{r(\Pi,\alpha_p)})$ and lies in the $U_{p,1} = p^{a+1}\alpha_p$ eigenspace.
	\item Let $u = \smallthreemat{1}{}{1}{}{1}{}{}{}{1}\smallthreemat{1}{}{}{}{}{-1}{}{1}{} \in \GL_3(\Zp)$ as before, and take for $W$ the element $u\tau_{ 1}^n  \cdot W^{\alpha}$

	\item Let $\Phi$ be the characteristic function $\Phi_R = \operatorname{ch}( (0, 1) + p^R \Zp^2)$, for $R = \mathrm{max}(n,r(\Pi,\alpha_p))$.
\end{itemize}
To ease notation, we will drop the $\widehat{\cdot}$ from local characters, writing for example $\eta_{i,p}$ for $\widehat{\eta}_{i,p}$.
\end{notation}

Recall the Coates--Perrin-Riou factor  $e_p(\Pi_p \times \eta_{1,p}, -j)$  from \cref{def:e_p}, and its explicit description in our present setting from \cref{ex:unramified P1 refinement}.
\begin{theorem}\label{thm:local zeta p}
We have
\begin{multline*}
	Y\Big(W, \Phi_R; \chi_1, \chi_2, \tfrac{1-j}{2}, \tfrac{-j}{2}\Big) = \\
	\frac{\alpha_p^n}{p^{2R + n}(1 - p^{-1})(1 - p^{-2})} \cdot e_p(\Pi_p \times \eta_{1,p}, -j) \cdot \cE_0 \cdot L(\Pi_p^\vee \times \omega_{\Pi,p}\eta_{2, p}, 0),
\end{multline*}
where $\cE_0$ is an explicit $p$-adic unit. It follows that
\begin{multline*}
	Z\Big(W, \Phi_R; \chi_1, \chi_2, \tfrac{1-j}{2}, \tfrac{-j}{2}\Big) = \\
	\frac{\alpha_p^n }{p^{2R + n}(1 - p^{-1})(1 - p^{-2})} \cdot e_p(\Pi_p \times \eta_{1,p}, -j) \cdot \frac{\cE_0 \cdot L(\Pi_p\times \omega_{\Pi,p}^{-1}\eta_{2, p}^{-1}, 1)}{\varepsilon(\Pi_p\times \omega_{\Pi,p}^{-1}\eta_{2, p}^{-1},1)}.
\end{multline*}
\end{theorem}

The factor $\alpha_p^n$, and the first denominator term (which has an interpretation as the index of a subgroup of $H(\Zp)$), will naturally cancel out in our $p$-adic interpolation computations. The local $L$-factor will later contribute to the period $\Omega_\Pi^-$ (see \eqref{eq:Omega}).

\begin{proof}
The statement for $Z$ follows immediately from the statement for $Y$ and \cref{thm:lfactor}.

We focus, then, on $Y$. To save notation we give the proof supposing $\eta_{2, p}$ is trivial; we indicate the necessary modifications for general $\eta_{2,p}$ in \cref{rem:non-trivial eta}.
Since $R \ge 1$, one computes that $f_{\Phi_R}(-, \chi_2, s_2)$ has support in $B_2(\Qp) K_0(p^R)$, whose measure in the quotient $B_2(\Qp) \backslash \GL_2(\Qp)$ is $\tfrac{1}{p^R(1+1/p)}$; and its value at the identity is $\tfrac{1}{p^R(1-1/p)}$. Hence for all sufficiently large $R$ (depending on $W$) we have
\[ Y(W, \Phi_R; \chi_1, \chi_2, s_1, s_2) = \tfrac{1}{p^{2R}(1-p^{-2})} y(W; \chi_1, \chi_2, s_1, s_2)(1).\]
Thus the Whittaker function in the integral $y\left(W; \chi_1, \chi_2, s_1, s_2\right)(1)$ is given by
\[ W^{\alpha}\left(\sdthree{x}{1}{y^{-1}}\smallthreemat{1}{}{}{}{}{1}{}{-1}{}u \tau_{1}^n\right) =
\psi(x) \omega_\Pi(y)^{-1} W^{\alpha}
\left(\smallthreemat{p^n xy}{}{}{}{y}{}{}{}{1}\right).
\]
Hence the integral is given by
\[
\sum_{(a, b) \in \Z^2} W^{\alpha}\left(\sdthree{p^{n+a+b}}{p^b}{1}\right) p^{-a(s_1 + s_2 - \tfrac{3}{2})}  p^{-b(s_2 - s_1-\tfrac{1}{2})} \int_{\Zp^\times}\psi(p^a x) \chi_1(x) \ \mathrm{d}^\times x.
\]
The values of the Whittaker function $W^{\alpha}$ along the torus are given by \cref{prop:torusWh}; they vanish unless $a \geq -n$ and $b \geq 0$. The integral over $\Zp^\times$ is standard: if $\eta_{1,p}$ is ramified, then it is $G(\eta_{1,p}^{-1}) / p^n(1-p^{-1})$ when $a = -n$, and vanishes otherwise; if $\eta_{1,p}$ is unramified (so $n = 1$), then it vanishes for $a < -1$, takes the value $\tfrac{-1}{p-1}$ for $a = -1$, and is 1 for $a \ge 0$.

For non-trivial $\eta_{1,p}$, we thus compute that $ Y(W, \Phi_R, \chi_1, \chi_2, s_1, s_2)$ is equal to
\begin{align*}
	&= \frac{G(\eta_{1,p}^{-1})}{p^{2R+n}(1-p^{-1})(1-p^{-2})} p^{n(s_1 + s_2 - \tfrac{3}{2})} \sum_{b \ge 0} W^{\alpha}\left(\sdthree{p^b}{p^b}{1}\right) p^{-b(s_2 - s_1-\tfrac{1}{2})}.
	\end{align*}
	Using \cref{prop:torusWh}, the sum is
	\begin{align*}
		\sum_{b \ge 0} W^{\alpha}\left(\sdthree{p^b}{p^b}{1}\right) p^{-b(s_2 - s_1-\tfrac{1}{2})} &=  \sum_{b \ge 0} W^{\mathrm{new}}_{\sigma_p'}\left(\matrd{p^b}{}{}{1}\right) p^{-b(s_2 - s_1)}\alpha_p^b\\
		&= \int_{\Qp^\times}W_{\sigma_p'}^{\mathrm{new}}\matrd{x}{}{}{1}\sigma_p(x)|x|^{s_2-s_1}d^\times x\\
		& = L(\sigma_p'\times\sigma_p, s_2-s_1+\tfrac{1}{2}),
		\end{align*}
	recalling $\sigma_p \times \sigma_p'$ is our $P_1$-refinement; the last equality is \cite[Prop.\ 6.17]{Gel75}. Thus
\[
Y(W, \Phi_R; \chi_1, \chi_2, s_1, s_2) = \frac{G(\eta_{1,p}^{-1}) p^{n(s_1 + s_2 -\tfrac{3}{2})}}{p^{2R+n}(1-p^{-1})(1-p^{-2})} \cdot L(\sigma_p' \times \sigma_p,s_2 - s_1 + \tfrac{1}{2}).
\]
If $\eta_{1,p}$ is trivial, we obtain instead
\[
\frac{\alpha_p}{p^{2R+1}(1 - p^{-1})(1 - p^{-2})} \cdot \frac{(1 - \alpha_p^{-1}p^{-(\tfrac{3}{2}-s_1-s_2)})}{(1 - \alpha_p p^{-(s_1 + s_2 - \tfrac{1}{2})})} \cdot L(\sigma_p' \times \sigma_p, s_2 - s_1 + \tfrac{1}{2}).
\]
We see that when $s_1 = \tfrac{1-j}{2}$ and $s_2 = -\tfrac{j}{2}$,
then in either case we obtain
\[
Y\big(W, \Phi_R; \chi_1, \chi_2, \tfrac{1-j}{2}, -\tfrac{j}{2}\big) = \frac{\alpha_p^n}{p^{2R+n}(1-p^{-1})(1-p^{-2})} \cdot e_p(\Pi_p \times \eta_{1,p}, -j) \cdot L(\sigma_p' \times \sigma_p,0).
\]
To remove the $L$-factor $L(\sigma_p'\times\sigma_p,0)$, let
\begin{equation}\label{eq:E_s}
	\cE_s = \frac{L(\sigma_p' \times \sigma_p,s)}{L(\Pi_p^\vee \times \omega_{\Pi,p}, s)}.
\end{equation}
The statement for $Y$ in the theorem follows by replacing $L(\sigma_p'\times\sigma_p,0)$ with $\cE_0\cdot L(\Pi_p^\vee \times \omega_{\Pi,p}, 0)$. It only remains to show $\cE_0$ is well-defined and a $p$-adic unit, which we do in \cref{lem:E_0 unit} below.
\end{proof}

\begin{lemma}\label{lem:E_0 unit}
If $\Pi_p$ is irreducibly induced from a representation $\alpha \times \sigma$ of $\GL_1 \times \GL_2$, then we have
\[
\frac{L(\sigma_p' \times \sigma_p, s)}{L(\Pi_p^\vee \times \omega_{\Pi,p}, s)} = \frac{1}{L(\omegahat_{\Pi, p} \sigma_p^{-1} , s)} = 1-\omega_{\Pi}(p)\alpha_p^{-1}p^{-s}.
\]
If the induction of $\sigma_p \times \sigma_p'$ is reducible, then this ratio is identically 1.

In particular, if $\Poref$ is $P_1$-ordinary, then $\cE_0$ is a $p$-adic unit.
\end{lemma}

%(We have taken $\eta_{2,p} = 1$ again, but clearly the analogous statement with $\eta_{2,p} \neq 1$ holds).

\begin{proof}
	For $\GL_3$, we have $L(\Pi_p^\vee\times \omega_{\Pi,p},s) = L(\wedge^2 \Pi_p , s)$ (see e.g.\ \cite[\S1]{Kim03}). If $\Pi_p$ is irreducibly induced, then we have $L(\wedge^2 \Pi_p, s) = L(\sigma_p' \times \sigma_p, s) L(\wedge^2 \sigma_p', s)$, and we have $\wedge^2 \sigma_p' = \omega_{\sigma_p'} = \widehat{\omega}_{\Pi,p} \sigma_p^{-1}$; and $\widehat{\omega}_{\Pi,p}(p)\sigma_p^{-1}(p) = \omega_{\Pi}(p)\alpha_p^{-1}$. In the remaining cases, one can write down the Langlands parameter as in \cref{ex:langlands} and argue similarly to find that the ratio is 1.

 For the final claim, we note that $\cE_0$ is either 1 or $1 -  \omega_{\Pi}(p)\alpha_p^{-1}$. As $\alpha_p^{-1}$ has valuation $1 + a > 0$ with respect to our choice of embedding into $\overline{\Q}_p$, whilst $\omega_{\Pi}(p)$ is a root of unity, we see $\cE_0$ must be a $p$-adic unit.
\end{proof}

\begin{remark}\label{rem:non-trivial eta}
	We indicate the changes required if $\eta_{2,p}$ is not trivial. In the sum over $(a,b)\in\Z^2$, we must add $\eta_{2,p}(p^b)$; and in the expression for $Y$, the $L$-value that appears is $L(\sigma_p'\times\sigma_p\eta_{2,p},s_2-s_1+1/2)$. Then $\cE_s = L(\sigma_p'\times\sigma_p\eta_{2,p},s)/L(\Pi_p^\vee\times\omega_{\Pi,p}\eta_{2,p},s)$. In \cref{lem:E_0 unit} we find the ratio is either 1 or $1-\omega_\pi(p)\eta_2(p)\alpha_p^{-1}p^{-s}$, and conclude as before, since $\eta_2(p)$ is a root of unity.
\end{remark}

\subsection{Local zeta integrals at infinity}
Now consider $v=\infty$. The following is directly analogous to the discussion after \cite[Lem.\ 1.1]{Mah00}.  Recall $\Phi_\infty^{j+2}$ from \eqref{eq:schwartz infinity}, and let $W_\infty \in \cW_\psi(\pi_\infty)$.

\begin{lemma}\emph{\cite{JPSS-L-factors}.} \label{lem:local zeta infty}
There exists a polynomial $P_j(W_\infty; T) \in \C[T]$ such that
\begin{align*}
	Z(W_\infty,\Phi_\infty^{j+2}; &\chi_1,\chi_2,s,-\tfrac{j}{2})  \\
	&= P_j(W_\infty; s +\tfrac{j}{2}) \cdot  L(\pi \times \chi_1, s - \tfrac{j + 1}{2})\cdot  L(\pi \times \chi_1\chi_2^{-1}, s+\tfrac{j + 1}{2}).
\end{align*}
\end{lemma}

\begin{remark}\label{rem:P}
Our $P_j(W_\infty;T)$ is not the direct analogue of $P_{w_\infty,\psi_\infty}(T)$ from \cite{Mah00}. Instead, it is scaled by a non-zero rational multiple of a power of $\pi$, which depends on $j$ but not $W_\infty$. In particular, we incorporate (the analogue of) Mahnkopf's $A\cdot \pi^{\flat-l}$ into $P$.
\end{remark}

%%%%%%%%%%%%%%%%%%%%%%%%%%%%%%%%%%%%%%%%%%%%%%%%%%%%%%%%%%%%%5

\section{The \texorpdfstring{$p$-adic $L$-function}{p-adic L-function}}\label{sec:p-adic L-function}

We collect our constructions and prove the interpolation formula, completing the proof of \cref{thm:intro} from the introduction for twists in the ``left half'' of the critical strip. Fix $\Poref = (\Pi,\alpha_p)$ a $P_1$-ordinary $P_1$-refined RACAR of $\GL_3(\A)$ of weight $\lambda = (a,0,-a)$. Recall the auxiliary character $\eta_{2}$ from \cref{not:eta_2}.

\subsection{Main result}
 We fix test data at the finite primes using the following recipe.

\begin{notation}\label{not:test data} \
\begin{itemize}\setlength{\itemsep}{0pt}
	
 \item[--] At $v \nmid p\infty$, choose a finite index set $I_v$, and a collection of Whittaker functions $W_{v, i} \in \cW(\Pi_v, E)$ and Schwartz functions $\Phi_{v, i}$ indexed by $i \in I_v$, such that 
 \[ \sum_{i \in I_v} \widetilde{Z}(W_v,\Phi_v; 1,\omega_{\Pi,v}\eta_{2,v},s_1,s_2) = 1.\]
 This is possible by \cref{thm:jpss}. We let $\varphi_{v, i} \in \Pi_v$ be such that $\cW_\psi(\varphi_{v, i}) = W_{v, i}$. 
 
 We can (and do) suppose that the $\Phi_{v, i}$ take values in $\Z$, and that for primes such that $\Pi_v$ is unramified, we have $|I_v| = 1$ and the corresponding $W_{v, i}$ and $\Phi_{v, i}$ are the normalized spherical data.
 
 \item[--] We define $I = \prod_{v \nmid p\infty} I_v$ (which is finite since almost all the $I_v$ are singletons), so that tensoring together the $W_{v, i}$ and $\Phi_{v, i}$ gives us a finite collection of vectors $W_i^{(p\infty)} \in \cW(\Pi^{(p)}\subf, E)$, $\Phi_i^{(p\infty)} \in \cS((\A\subf^{p})^2, \Z)$ with
 \[ \sum_{i \in I} \widetilde{Z}\left(W_i^{(p\infty)}, \Phi_i^{(p\infty)}; 1,\omega_{\Pi}^{(p\infty)} \eta_{2}^{(p\infty)} ,s_1,s_2\right) = 1.\]

	\item[--] At $v = p$, we take $\varphi_p$ such that $\cW_\psi(\varphi_p) =  W^\alpha$ and $\Phi_p = \Phi_R$ (as in \cref{not:test data p}).
\end{itemize}
\end{notation}

This gives us a finite collection of vectors $\varphi_{\mathrm{f}, i} \in \Pif$ and $\Phi_{\mathrm{f}, i} \in \cS(\A\subf^2, \Z)$. We fix a choice of level group $\cU^{(p)}$ away from $p$ fixing all of the $\varphi_{\mathrm{f}, i}$, and work at level $\cU= \cU^{(p)}\cU_p$ where $\cU_p$ is the subgroup $\cU_p^{(P_1)}(p^{r(\Pi_p, \alpha_p)})$ as above. We normalize our periods, as in \S\ref{sec:integral periods} and \cref{sec:measure to integral}, so that we have $\phi_{\varphi_{\mathrm{f}, i}} \defeq \phi_\Pi(W_{\varphi_{\mathrm{f}, i}})/\Theta_\Pi \in \hc{2}(Y^{\GL_3}(\cU),\sV_\lambda^\vee(\cO_L)) / \{\mathrm{torsion}\}$ for all $i$, for $L/\Qp$ finite (containing $E$).

Recall $\Xi^{[a,j]}(\phi_{\varphi\subf},\Phi^{(p)}) \in \cO_L[\![\Zp^\times]\!]$ from \cref {sec:pairing with cusp form,sec:getting rid of c}.

\begin{definition}
 We write $\iota$ for the involution of $\cO_L[\![\Zp^\times]\!]$ corresponding to $x \mapsto x^{-1}$ on $\Zp^\times$.
\end{definition}

\begin{definition}\label{def:p-adic L-function}
The \emph{left-half $p$-adic $L$-function} of $\Pi$ is defined by
\[ L_p^-(\Pi) \defeq \iota\left( \sum_{i \in I} \Xi^{[a,0]}(\phi_{\varphi_{\mathrm{f}, i}},\Phi^{(p)}_{\mathrm{f}, i})\right) \in \cO_L[\![\Zp^\times]\!].\]
\end{definition}

\begin{remark}
		We include the involution $\iota$ in order to obtain a measure with an interpolating property at integers in $[-a, 0]$ rather than $[0, a]$, to better match the critical range for the complex $L$-function.
		\end{remark}

Recall $\cE_0 \neq 0$ from \eqref{eq:E_s} and \cref{lem:E_0 unit}, and define a period
\begin{equation}\label{eq:Omega}
\Omega_\Pi^- \defeq \frac{\Theta_\Pi \cdot \varepsilon(\Pi_p \times (\omega_{\Pi,p}\eta_{2,p})^{-1},1)}{ \cE_0 \cdot L(\Pi\times(\omega_{\Pi}\eta_{2})^{-1},1)} \in \C^\times.
\end{equation}

The following interpolation, which proves \cref{conj:intro}(i) of the introduction for $n=3$, is our main theorem. The proof will occupy the rest of \cref{sec:p-adic L-function}.

\begin{theorem}\label{thm:main thm}
	Let $\Pi$ be a RACAR of $\GL_3(\A)$ of weight $(a,0,-a)$, admitting an ordinary $P_1$-refinement at $p$. The $p$-adic $L$-function $L_p^-(\Pi) \in \cO_L[\![\Zp^\times]\!]$ from \cref{def:p-adic L-function} satisfies the following interpolation property: for all $(-j,\eta) \in \Crit^-_p(\Pi)$, we have
\begin{equation}\label{eq:interpolation final}
	\int_{\Zp^\times} \eta^{-1}(x)x^{-j} \cdot \mathrm{d}L_p^-(\Pi)(x) = e_\infty(\Pi_\infty \times \eta_\infty, -j) e_p(\Pi_p \times \eta_p, -j)\cdot \frac{L^{(p)}(\Pi\times\eta,-j)}{\Omega_\Pi^-}.
\end{equation}
Here $\Crit^-_p(\Pi) = \{(-j,\eta) : 0 \leq j \leq a, \ \mathrm{cond}(\eta) \mid p^\infty\}$ was defined in \S\ref{sec:rationality}, $e_\infty$ in \cref{def:modified infinity}, and $e_p$ in \cref{def:e_p}.
\end{theorem}

\subsection{Interpolation of \texorpdfstring{$L$}{L}-values, part I}\label{sec:interpolation 1}
Recall from  \cref{sec:cup product integral}: our choices of $\zeta_\infty$ (in \eqref{eq:omega infinity}) and bases $v_\alpha,w_\beta^{[j]}$ determined forms $\varphi_{\infty,r,s,\alpha} \in \Pi_\infty$, and we set $\varphi_{r,s,\alpha} = \varphi_{\infty,r,s,\alpha} \cdot \varphi\subf$. By \cref{theorem:compatibility} and \cref{prop:measure to integral}, the left-hand side of \eqref{eq:interpolation final} is then
\begin{align}
&\stackrel{5.4}{=} \int_{\Zp^\times} \eta(x)  \cdot \mathrm{d}\Xi^{[a,j]}(\phi_{\varphi\subf})(x) \notag \\
&\stackrel{7.2}{=}    \frac{1 }{p^{n}\operatorname{vol}(\cU_{n,p}^H)\alpha_{p}^n  \Theta_\Pi} \sum_{r,s,t} \varepsilon_{rst} \sum_{\alpha,\beta} \langle v_\alpha, w_\beta^{[j]}\rangle_{a,j} \cZ\Big(u\tau_1^n\cdot \varphi_{r,s,\alpha}, \Phi; \eta, \omega_\Pi\eta\eta_2, \tfrac{1-j}{2}, -\tfrac{j}{2}\Big). \label{eq:interpolation intermediate}
\end{align}

Each $\cZ(-)$ is a product of local integrals $Z_v(-)$ by \eqref{eq:global to local}. As the \emph{finite} parts of the $\varphi_{r,s,\alpha}$ are the same, for each $v\nmid\infty$ the integral $Z_v(-)$ is the same in each summand. Such $Z_v(-)$ is computed in \cref{thm:jpss} ($v\nmid p\infty$) and \cref{thm:local zeta p} ($v = p$, noting $W = u\tau_1^n \cdot W^\alpha$). We conclude
\begin{multline*}
\cZ\big(u\tau_1^n \cdot \varphi_{r,s,\alpha}, \Phi; \eta,\omega_\Pi\eta\eta_2,\tfrac{1-j}{2},-\tfrac{j}{2}\big)  =  \\  Z_\infty\big(W_{\infty, r,s,\alpha}, \Phi_\infty^{j+2}; \etahat_{\infty}, \omegahat_{\Pi,\infty}\etahat_{\infty}, \tfrac{1-j}{2}, -\tfrac{j}{2}\big)
\times \frac{\alpha_p^n \cdot \cE_0 }{p^{2R + n}(1 - p^{-1})(1 - p^{-2})}\\
\times e_p(\Pi_p \times \eta_{p}, -j) \cdot L^{(p)}(\Pi \times \eta, -j) \cdot \frac{L(\Pi \times (\omega_{\Pi}\eta_2)^{-1}, 1)}{\varepsilon(\Pi_p \times (\omega_{\Pi,p}\eta_{2,p})^{-1},1)},
\end{multline*}
where  $W_{\infty,r,s,\alpha} = \cW_\psi(\varphi_{\infty,r,s,\alpha}) \in \cW_\psi(\Pi_\infty)$ (recalling that $\eta_2$ is even, so $\etahat_{2,\infty} = 1$).

We now sweep the sums into the local zeta integral at $\infty$. Let
\[
\widetilde{e}_\infty(\zeta_\infty \times \eta_\infty, -j)\defeq \sum_{r,s,t}\varepsilon_{rst}\sum_{\alpha,\beta}\left\langle v_\alpha, w_\beta^{[j]}\right\rangle_{a,j} Z_\infty(W_{\infty, r,s,\alpha}, \Phi_\infty^{j+2}; \etahat_\infty,\omegahat_{\Pi,\infty}\etahat_\infty,\tfrac{1-j}{2},-\tfrac{j}{2}).
\]
 Then
\begin{align}
\eqref{eq:interpolation intermediate} &=  \frac{1 }{p^{n}\operatorname{vol}(\cU_{n,p}^H)\alpha_{p}^n  \Theta_\Pi}  \cdot 	\frac{\alpha_p^n \cdot \cE_0 }{p^{2R + n}(1 - p^{-1})(1 - p^{-2})}  \cdot \widetilde{e}_\infty(\zeta_\infty \times \eta_\infty,-j) \\
&\hspace{50pt} \times e_p(\Pi_p \times \eta_{p}, -j) \cdot L^{(p)}(\Pi \times \eta, -j) \cdot \frac{L(\Pi \times (\omega_{\Pi}\eta_2)^{-1}, 1)}{\varepsilon(\Pi_p \times (\omega_{\Pi,p}\eta_{2,p})^{-1},1)}.\notag
\end{align}
We compute $\mathrm{vol}(\cU_{n,p}^H)^{-1} = p^{2R+2n}(1-p^{-1})(1-p^{-2})$ via \cref{prop:intersection level}, whence by \eqref{eq:Omega} we find
\begin{equation} \label{eq:interpolation intermediate 2}
\int_{\Zp^\times} \eta(x)x^{j} \cdot \mathrm{d}L_p(\Poref)(x)  =
\widetilde{e}_\infty(\zeta_\infty \times \eta_\infty,-j) \cdot e_p(\Pi_p \times \eta_{p}, -j)\cdot \frac{L^{(p)}(\Pi \times \eta, -j)}{\Omega_\Pi^-}.
\end{equation}

\subsection{Non-vanishing at infinity}
It remains to evaluate $\widetilde{e}_\infty(\zeta_\infty \times \eta_\infty,-j)$. We first show it is non-zero. For each $Z_\infty(W_{\infty,r,s,\alpha},-)$ in its definition, by \cref{lem:local zeta infty} we get an associated polynomial $P_j(W_{\infty,r,s,\alpha};T) \in \C[T]$. Define
\[
P_j(\zeta_\infty;s+\tfrac{j}{2}) \defeq \sum_{r,s,t}\varepsilon_{rst}\sum_{\alpha,\beta}\left\langle v_\alpha, w_\beta^{[j]}\right\rangle_{a,j} \cdot P_j(W_{\infty,r,s,\alpha};s + \tfrac{j}{2}).
\]
This is (an explicit non-zero multiple of) the analogue of $P_l(s)$ in \cite[(3.1)]{Mah00} (cf.\ \cref{rem:P}). Combining with \cref{lem:local zeta infty}, we see
\begin{align*}
\widetilde{e}_\infty(\zeta_\infty \times \eta_\infty,-j) =  P_j(\zeta_\infty; \tfrac{1}{2}) \cdot  L(\Pi_\infty \times \eta_\infty, -j)\cdot  L(\Pi_\infty \times \omega_{\Pi,\infty}^{-1}, 1).
\end{align*}
By \cite{KS13}, we know $P_j(\zeta_\infty; \tfrac{1}{2}) \neq 0$ (cf.\ \cite[Thm.\ 2.1]{Ger15}), and hence $\widetilde{e}_\infty(\zeta_\infty \times \eta_\infty,-j) \neq 0$.

\subsection{Symmetric square \texorpdfstring{$p$-adic $L$-functions}{p-adic L-functions}}
\label{sec:symmetric square}
To pin the term at infinity down more precisely, we exploit the fact that \cref{thm:main thm} is known in full when $\Pi$ is essentially self-dual, i.e.\ $\Pi$ is a (twist of a) symmetric square lift. We recall this result.

Let $f$ be a classical cuspidal $p$-ordinary newform of weight $a+2$ and level $N$, and $\theta$ a finite order Hecke character over $\Q$ of prime-to-$p$ conductor. Let $\pi' = \mathrm{Sym}^2(f) \times \theta$ be the symmetric square lift to $\GL_3$, which has weight $(2a,a,0)$, and let $\pi \defeq \pi' \times \|\cdot\|^{-a}$ (which has weight $(a,0,-a)$).

The form $f$ has a unique ordinary $p$-refinement; write $p^{1/2}A_p$ for its $U_p$-eigenvalue, which is a $p$-adic unit (so $v_p(A_p) = -1/2$). Then $\pi$ has an ordinary $P_1$-refinement defined by an unramified character $\sigma_p$, where $\sigma_p(p) = \alpha_p \defeq p^{-a}A_p^2\theta(p)$; as $\theta(p)$ is a root of unity, we have $v_p(\alpha_p) = -a-1$.

If $\omega_\pi(-1) = -1$, let $b = 0$; otherwise let $b=1$.
\begin{theorem}[Schmidt, Hida, Dabrowski--Delbourgo, Rosso] \label{thm:sym square}
There exists $\cL_p^-(\pi) \in \Cp\otimes_{\Zp}\Zp[\![\Zp^\times]\!]$ such that for finite-order characters $\eta$ of $\Zp^\times$, and $0 \leq j \leq a$ with $(-1)^j = \omega_\pi\eta(-1)$, we have
\begin{equation}\label{eq:sym square}
	\int_{\Zp^\times} \eta^{-1}(x)x^{-j}\cdot \mathrm{d}\cL_p(\tilde\pi)(x) = \frac{\omega_\pi(-1)\cdot \Gamma(-j+a+1)}{2^{2a+4}\cdot i^b \cdot (2\pi i)^{-j}} \cdot e_p(\pi_p \times \eta_p,-j)\cdot \frac{L^{(p)}(\pi\times\eta,-j)}{\pi^{a+1}\langle f,f\rangle}.
\end{equation}
\end{theorem}

\begin{proof}
This is summarised in \cite[Thm.\ 2.3.2(i)]{LZ-symmetric}, at least when $p\nmid N$; the case where $p||N$ is described in \cite{Ros16}. For the convenience of the reader we indicate how the statement in \cite{LZ-symmetric} translates to that above; the comparison with \cite{Ros16} follows as both are compatible with Conjecture \ref{conj:intro}(i), which has a uniform statement whether $p$ divides $N$ or not. 

Since $L(\pi,s) = L(\mathrm{Sym}^2(f)\times \eta,s+a+1)$ we renormalise, first defining $\cL_p^-(\tilde\pi)$ so that
\[
\int_{\Zp^\times} f(x) \cdot \mathrm{d}\cL_p^-(\tilde\pi) = \int_{\Zp^\times}x^{a+1}f(x)\cdot \mathrm{d}L_p(\mathrm{Sym}^2f \otimes \theta)
\]
(in the notation of \cite{LZ16}), and making the substitution $-j = s-a-1 = s-k+1$ (with $k = a+2$). Translating between the data attached to the refinements of $f$ and $\pi$ then equates $G(\eta) \cdot \cE_p(s,\eta)$ \emph{op.\ cit}.\ with $e_p(\pi_p \times \eta_p,-j)$ here.   The parity condition \emph{op.\ cit}.\ is $(-1)^{a+1-j}\eta(-1) = - \theta(-1)$; since $\omega_\pi(-1) = \omega_{\mathrm{Sym}^2(f)}(-1)\theta(-1) = (-1)^a\theta(-1)$, this is equivalent to $(-1)^j = \omega_\pi\eta(-1)$. Combining all of this shows that $\int_{\Zp^\times} \eta^{-1}(x)x^{-j} \cdot \cL_p^-(\tilde\pi)(x)$ is equal to the right-hand side of \eqref{eq:sym square}.
\end{proof}

\begin{remark}
The proof of \cref{thm:sym square} is rather circuitous, owing to complications with local Euler factors at the bad primes.
Taking $\theta = 1$ for simplicity, the method initially interpolates values of the ``imprimitive'' symmetric square $L$-function
\[
L^{\mathrm{imp}}(\operatorname{Sym}^2(f), s) = L^{(N_f)}(2s - 2a - 2, \omega_f^2) \cdot \sum_{n \ge 1} a_{n^2}(f) n^{-s},
\]
which differs from the ``true'' symmetric square $L$-function by a product of local error terms at the primes dividing $N_f$. Having constructed an imprimitive $p$-adic $L$-function, one can attempt to define a primitive $p$-adic $L$-function by dividing out by the error term; however, it remains to be shown that the resulting function does not have poles at the zeroes of the error term, and that it has the expected interpolation property at all $(j, \eta)$ in the interpolation range (even if the error term vanishes there, which can occur). This requires rather lengthy case-by-case analysis according to the local factors of $f$ (see \cite[\S 6]{Hid90} and \cite[\S 3.1]{DD97}).

However, for our present purpose of identifying the factors $\widetilde{e}_\infty(\zeta_\infty \times \eta_\infty,-j)$, it suffices to know that there is a (possibly meromorphic)  $p$-adic $L$-function which satisfies the conclusion of Theorem 10.3 for \emph{almost all} $(j, \eta)$ in the appropriate range. This is much more straightforward to prove using the methods of \cite{Sch88}. Combining this partial result towards \cref{thm:sym square} with the output of our present construction, we obtain the full strength of \cref{thm:sym square} as a consequence, yielding an alternative proof not requiring the intricate local computations involved in the previous approach.
\end{remark}

\subsection{Interpolation of \texorpdfstring{$L$}{L}-values, part II}\label{sec:interpolation 2}
We are evaluating $\widetilde{e}_\infty(\zeta_\infty \times \eta_\infty,-j)$. This term depends only on data attached to $\Pi_\infty$ and $j$, as (to be critical) $\eta_\infty = \omega_{\Pi,\infty} \cdot \mathrm{sgn}^j$.

Note in particular $\widetilde{e}_\infty(\zeta_\infty,-j)$ does not depend on any data at $p$. We now exploit this, and existence of symmetric square $3$-adic $L$-functions, to show:

\begin{proposition}\label{prop:ratio}
	For any $0 \leq j \leq a$, we have
	\[
		\widetilde{e}_\infty(\zeta_\infty \times \omega_{\Pi,\infty}\mathrm{sgn}^j,-j) = (2\pi i)^{j} \cdot \tfrac{\Gamma(-j+a+1)}{\Gamma(a+1)} \cdot \widetilde{e}_\infty(\zeta_\infty\times \omega_{\Pi,\infty},0).
	\]
\end{proposition}

\begin{proof}
	 To ease notation, we write $\widetilde{e}_\infty(\zeta_\infty,-j) = \widetilde{e}_\infty(\zeta_\infty \times \omega_{\Pi,\infty}\mathrm{sgn}^j,-j)$. 
	
From \eqref{eq:Pi_infty}, for fixed $a$ there are only two possibilities for $\Pi_\infty$; denote them by
\[
\Pi_\infty^{a,+} \defeq \mathrm{Ind}_{P_2(\R)}^{\GL_3(\R)} (D_{2a+3}, \mathrm{id}) \qquad \Pi_\infty^{a,-} \defeq \mathrm{Ind}_{P_2(\R)}^{\GL_3(\R)} (D_{2a+3}, \mathrm{sgn}).
\]

\begin{lemma}
	Let $a \in \Z_{\geq 0}$. For $\epsilon \in \{\pm\}$, there exists a RACAR $\pi^{a,\epsilon}$ such that:
	\begin{itemize}
		\item[(i)] $\pi^{a,\epsilon}_\infty = \Pi_{\infty}^{a,\epsilon}$;
		\item[(ii)] $\pi^{a,\epsilon}_3$ is $B$-ordinary;
		\item[(iii)] There exists an elliptic modular newform $f_a$ of weight $a+2$ and a character $\theta$ such that $\pi^{a,\epsilon} \cong \mathrm{Sym}^2(f_a)\otimes \theta||\cdot||^{-a}$.
	\end{itemize}
\end{lemma}
\begin{proof}
	First suppose $a = 0$ and $\epsilon = +$. Take $f_0$ to be the unique newform of level 15, trivial character and weight 2. This newform is 3-ordinary and not of CM-type, so $\pi^{0,+} \defeq \mathrm{Sym}^2(f)$ is the required RACAR.

	By Hida theory, for any $a+2 \ge 2$ there exists a 3-ordinary newform $f_a$ of weight $a+2$ congruent to $f_0$ mod 3 (with trivial character if $k$ is even, and quadratic character mod 3 if $k$ is odd). Since the mod 3 Galois representation associated to $f_0$ is surjective, $f_a$ cannot be of CM-type, so lifts to a RACAR of $\GL_3$. We may thus take $\pi^{a,+}\defeq \mathrm{Sym}^2(f_a)\otimes||\cdot||^{-a}$.

	Finally we take $\pi^{a,-} \defeq \pi^{a,+} \otimes \theta$, for $\theta$ an odd finite order Hecke character unramified at $3$.
\end{proof}

Now we return to the proof of \cref{prop:ratio}. The lemma implies that for our given $\Pi$, there exists a RACAR $\pi \cong \mathrm{Sym}^2(f) \otimes \theta||\cdot||^{-a}$ such that $\Pi_\infty \cong \pi_\infty$ and $\pi_3$ is $B$-ordinary (hence $P_1$-ordinary). As $\widetilde{e}_\infty(\zeta_\infty,-j)$ only depends on the factor at infinity, it suffices to work with $\pi$ rather than $\Pi$. Let $L_3^-(\pi)$ and $\cL_3^-(\pi)$ be the $3$-adic $L$-functions of \cref{def:p-adic L-function} and \cref{thm:sym square} respectively.

Note that $L_3^-(\pi)$ and $\cL_3^-(\pi)$ interpolate the same $L$-values $L(\pi\times\eta,-j)$, so when considered as (bounded) rigid analytic functions on weight space $\sW$, they are supported on the same half of $\sW$. In particular, we can make sense of $L_3^-(\pi)/\cL_3^-(\pi) \in \mathrm{Frac}(\C_3\otimes_{\Z_3}\Z_3[\![\Z_3^\times]\!])$ as a well-defined meromorphic function on $\sW$. This quotient is uniquely determined by its integral against finite-order characters $\eta$ of $\Z_3^\times$ such that $\eta(-1) = \omega_\pi(-1)$ (and vanishes when $\eta(-1) = - \omega_\pi(-1)$). By considering $j = 0$ in \eqref{eq:interpolation intermediate 2} and \cref{thm:sym square}, we deduce that
\[
\int_{\Z_3^\times} \eta(x) \cdot \mathrm{d}\frac{L_3^-(\pi)}{\cL_3^-(\pi)} =  \widetilde{e}_\infty(\zeta_\infty,0) \cdot \left[\frac{\omega_\pi(-1)\cdot \Gamma(a+1)}{2^{2a+4}\cdot i^b \cdot \pi^{a+1}}\right]^{-1} \cdot \frac{\langle f,f\rangle}{\Omega_\Pi^-}.
\]
As this is independent of $\eta$, the quotient $L_3^-(\pi)/\cL_3^-(\pi)$ is \emph{constant}, say equal to $C \in \C_3^\times$.

Now let $0 \leq j \leq a$, and $\eta$ such that $(-1)^j = \omega_\pi\eta(-1)$. By constancy, this is
\begin{align*}
C &= \int_{\Z_3^\times} \eta(x)x^{j}(x) \cdot \mathrm{d}\frac{L_3^-(\pi)}{\cL_3^-(\pi)}\\
&= \widetilde{e}_\infty(\zeta_\infty,-j)\cdot\left[ \frac{\omega_\pi(-1)\cdot \Gamma(-j+a+1)}{2^{2a+4}\cdot i^b \cdot (2\pi i)^{-j}\cdot \pi^{a+1}}\right]^{-1}  \cdot \frac{\langle f,f\rangle}{\Omega_\Pi^-},
\end{align*}
where the second equality is the interpolation formula; and hence we have
\begin{equation}\label{eq:ratio}
\widetilde{e}_\infty(\zeta_\infty,-j) = (2\pi i)^{j} \cdot \tfrac{\Gamma(-j+a+1)}{\Gamma(a+1)} \cdot \widetilde{e}_\infty(\zeta_\infty,0),
\end{equation}
proving \cref{prop:ratio}.
\end{proof}

Now return to our original representation $\Pi$.  We are free to renormalise $\zeta_\infty$ by any element of $\C^\times$; this then rescales $\Theta_\Pi$, $\Omega_\Pi^-$, $P_j(\zeta_\infty;s)$, and hence also $\widetilde{e}_\infty(\zeta_\infty \times \omega_{\Pi,\infty}\mathrm{sgn}^j,-j)$. We renormalise it so that $\widetilde{e}_\infty(\zeta_\infty \times \omega_{\Pi,\infty},0) = e_\infty(\Pi_\infty \times \omega_{\Pi,\infty},0).$ Since by definition 
\[
e_\infty(\Pi_\infty \times \omega_{\Pi,\infty}\mathrm{sgn}^j,-j) = (2\pi i)^{j} \cdot \tfrac{\Gamma(-j+a+1)}{\Gamma(a+1)} \cdot e_\infty(\Pi_\infty \times \omega_{\Pi,\infty},0),
\]
from \cref{prop:ratio} we deduce
\begin{equation}\label{eq:infty compatibility}
\widetilde{e}_\infty(\zeta_\infty \times \omega_{\Pi,\infty}\mathrm{sgn}^j,-j) = e_\infty(\Pi_\infty \times \omega_{\Pi,\infty}\mathrm{sgn}^j,-j).
\end{equation}
\cref{thm:main thm} follows by combining this with \eqref{eq:interpolation intermediate 2}, recalling that if $(-j,\eta)$ is critical, then $\eta_\infty =  \omega_{\Pi,\infty}\mathrm{sgn}^j$.

\begin{remark}
Combining the equality \eqref{eq:infty compatibility} with \cref{thm:mahnkopf} also completes the proof of the Algebraicity Conjecture (\cref{conj:intro-alg}) for $n=3$ for $(-j,\eta) \in \Crit^-(\Pi)$; the analogous result for $\Crit^+(\Pi)$ will be proved in the next section.
\end{remark}

\section{Duality and functional equations}
 \label{sec:functional equation}

 Our results so far have focused entirely on $P_1$-nearly-ordinary representations, and on interpolation of $L$-values $L(\Pi\times \eta,-j)$ for $(-j,\eta) \in \Crit^-(\Pi)$, the left half of the critical range. In this section we complete the proof of \cref{conj:intro} by studying the mirror-image picture for $P_2$-nearly-ordinary refinements, for which we obtain interpolation over $\Crit^+_p(\Pi)$, that is, in the right-half of the critical range.

 Suppose $\Pi$ is a RACAR of $\GL_3$ admitting a (necessarily unique) $P_2$-nearly-ordinary refinement $\sigma_p \times \sigma_p'$. Using the invariance of \cref{conj:intro} under twisting, as in \cref{prop:nearly-ordinary vs ordinary}, we may assume without loss of generality that $\sigma_p'$ is an unramified character\footnote{We might call such refinements ``co-ordinary''.}. Hence the dual $\sigma_p^{\prime\vee}$ defines a $P_1$-ordinary refinement of $\Pi^\vee$; and \cref{thm:main thm} thus yields a $p$-adic $L$-function $L_p^-\big(\Pi^\vee\big) \in \cO_L[\![\Zp^\times]\!]$.

 \begin{definition} \
 	\begin{itemize}
 		\item[(i)] If $C \in \Zp^\times$, let $[C] \in \cO[\![\Zp^\times]\!]^\times$ denote the Dirac measure defined by
 		\[
  	\int_{\Zp^\times} f(x) \cdot \mathrm{d}[C] = f(C).
 		\]
   \item[(ii)] Let $\operatorname{tw}_1$ be the unique automorphism of $\cO_L[\![\Zp^\times]\!]$ sending $[C]$ to $C \cdot [C]$, and (as above) $\iota$ the involution sending $[C]$ to $[C^{-1}]$, so 
   \[  \int x^s\, \mathrm{d}\iota(\operatorname{tw}_1 \mu) = \int x^{1-s}\, \mathrm{d}\mu
   \]
   for all $s \in \Z$ and $\mu \in \cO_L[\![\Zp^\times]\!]$.

 		\item[(iii)] Define the measure $L_p^+(\Pi)$ to be
 		\[
 		L_p^+\big(\Pi\big) \defeq \frac{1}{\varepsilon^{(p)}(\Pi^{\vee,(p)}, 0)} \cdot (\iota \circ \operatorname{tw}_1) \left(\left[N_{\Pi}^{(p)}\right] L_p^-\big( \Pi^\vee \big)\right) \in \cO_L[\![\Zp^\times]\!].
 		\]
 	\end{itemize}
 \end{definition}

 Here $N_\Pi^{(p)}$ is the prime-to-$p$ part of the conductor of $\Pi$ (the integer such that the functional equation of the complex $L$-function involves a factor of $N_{\Pi}^{-s}$); the Dirac measure $[1/N_\Pi^{(p)}]$ is a $p$-adic avatar of this factor.
 Let also $\Omega_\Pi^+ \defeq \Omega_{\Pi^\vee}^-.$

 \begin{proposition}
 The $p$-adic $L$-function $L_p^+(\Pi)$ satisfies the following interpolation property: for all $(j+1,\eta) \in \Crit^+(\Pi)$, we have
 \begin{equation}\label{eq:interpolation final +}
 	\int_{\Zp^\times} \eta(x)x^{j+1} \cdot \mathrm{d}L_p^+(\Pi)(x) = e_\infty(\Pi_\infty \times \eta_\infty^{-1},j+1)\cdot e_p\big(\Pi_p \times \eta_p^{-1}, j+1\big)\cdot \frac{L^{(p)}(\Pi\times\eta^{-1},j+1)}{\Omega_\Pi^+}.
 \end{equation}
 \end{proposition}

 \begin{proof}
 	By \cref{thm:main thm}, we have
 	\begin{align*}
 		\int_{\Zp^\times} \eta(x)x^{j+1} \cdot \mathrm{d}L_p^+(\Pi)(x) = (N_{\Pi}^{(p)})^{-j}&\cdot \eta(N_{\Pi}^{(p)})^{-1} \cdot \varepsilon^{(p)}(\Pi^{\vee,(p)},0)^{-1}\\
 		&\times e_\infty(\Pi_\infty^\vee \times \eta_\infty, -j)\cdot e_p(\Pi^\vee_p \times \eta_p,-j)\cdot \frac{L^{(p)}(\Pi^\vee\times\eta,-j)}{\Omega_{\Pi^\vee}^-}.
 	\end{align*}
 	In \cite[(20)]{coates89}, Coates proves the equality
 	\begin{align*}
 	 e_\infty(&\Pi_\infty^\vee \times \eta_\infty,-j) \cdot e_p(\Pi^\vee_p \times \eta_p, -j)\cdot \frac{L^{(p)}(\Pi^\vee\times\eta,-j)}{\Omega_{\Pi^\vee}^-}\\
 	 &=\Big(\prod_{\ell \neq p} \varepsilon_\ell(\Pi_\ell^\vee\times \eta_\ell,-j)\Big) \cdot 	e_\infty(\Pi_\infty \times \eta_\infty^{-1}, j+1)\cdot e_p(\Pi_p \times \eta_p^{-1},j + 1)\cdot \frac{L^{(p)}(\Pi\times\eta^{-1},j+1)}{\Omega_{\Pi}^+}.
 	\end{align*}
 	Here we recall that the modified Coates--Perrin-Riou factors were set up to use the additive character $\psi$ for $\Crit^-_p(\Pi^\vee)$ and $\psi^{-1}$ for $\Crit^+_p(\Pi)$. As $\eta$ is unramified at $\ell \neq p$, as in \cite{coates89} we can exploit \cite[(3.4.6)]{Tat79} to see
 	\[
 	 \varepsilon_\ell(\Pi_\ell^\vee\times \widehat{\eta}_\ell,-j) = \varepsilon_\ell(\Pi_\ell^\vee,0) \cdot \ell^{jc_\ell}\cdot \widehat{\eta}_\ell^{-1}(\ell^{c_\ell}),
 	\]
 	where $c_\ell$ is the exponent of $\ell$ in $N_{\Pi^\vee} = N_{\Pi}$. In particular, we have
 	\[
 	 \prod_{\ell \neq p} \varepsilon_\ell(\Pi_\ell^\vee\times \widehat{\eta}_\ell,-j) = \varepsilon^{(p)}(\Pi^{\vee,(p)},0) \cdot  (N_\Pi^{(p)})^j \cdot \widehat{\eta}^{(p)}(N_\Pi^{(p)})^{-1}.
 	\]
 	We conclude by combining all terms, noting $\hat\eta^{(p)}(N_\Pi^{(p)})^{-1} = \widehat{\eta}_p(N_\Pi^{(p)})^{-1} = \eta(N_\Pi^{(p)})$ via the conventions in \S\ref{sect:dirichlet}.
 \end{proof}
 
 With this, we have completed the proof of \cref{thm:intro}.

\appendix
\normalsize
\section{Proof of \texorpdfstring{\cref{thm:lfactor}}{\cref*{thm:lfactor}}}

 We now give the proof of the ``partial functional equation'' relating the $Y$ and $Z$ integrals. Compare Proposition A.3 of \cite{LZ-distinguished}, which is a closely related computation in the case of $\GL_2 \times \GL_2$, rather than $\GL_3 \times \GL_2$. Replacing $(\pi, \chi_1, \chi_2)$ with $(\pi \times \chi_1, \mathrm{id}, \chi_2)$, we may suppose without loss of generality that $\chi_1 = \mathrm{id}$.

\begin{remark}
The overall shape of the computation is as follows: the integral $Z(W, \Phi; \chi_1, \chi_2, s_1, s_2)$ involves the function $W^{\Phi}(-; \chi_2, s_2)$, which lies in the Whittaker model of a $\GL_2$ principal-series representation (depending on the parameters $\chi_2$ and $s_2$). Meanwhile, the $Y$ integral involves $f^{\Phi}(-; \chi_2, s_2)$, which lies in the natural model of the $\GL_2$ principal series (as a space of functions on $\GL_2$ transforming by a character under left-translation by the Borel). There is an explicit intertwining operator from the natural induced model of the principal-series representation to its Whittaker model, sending $f^{\Phi}(-; \chi_2, s_2)$ to $W^{\Phi}(-; \chi_2, s_2)$; and we need to show that the composite of the $Z$ integral and the intertwining operator agrees with the $Y$ integral up to the constant factor $\gamma(\pi \times \chi_1/\chi_2, s_1 - s_2 + \tfrac{1}{2})$. 

We note that both integrals define elements of the space $\Hom_H(\pi \times \sigma, \C)$, where $\sigma$ is a principal-series $H$-representation (depending on the parameters $s_i, \chi_i$). For sufficiently general values of $s_i$, the representation $\sigma$ is irreducible and generic, so this Hom-space has dimension 1, by a well-known multiplicity-one result (see the introduction of \cite{prasad93}, where the result is attributed to J.~Bernstein). Hence the two integrals must be proportional, and it remains only to identify the constant as a $\gamma$-factor. 
\end{remark}

We begin by recalling some $\GL_3 \times \GL_1$ zeta-integrals introduced by Jacquet et al. We let $F = \Q_p^\times$ (in fact any nonarchimedean local field would work here).

\begin{definition}
For $W \in \cW(\Pi)$ we consider the integrals
\[ \Psi_0(W, \chi, s) = \int_{F^\times} W\left(\sdthree{a}{1}{1}\right) \chi(a) |a|^{s-1}\ \mathrm{d}^\times\!a\]
and
\[ \Psi_1(W, \chi, s) = \int_{F \times F^\times} W\left(\smallthreemat{a}{}{}{x}{1}{}{}{}{1}\right) \chi(a) |a|^{s-1} \ \mathrm{d}x\, \mathrm{d}^\times\!a.\]
\end{definition}

By \cite[Thm.\ 2.7(i),(ii)]{JPSS-RankinSelberg}, both integrals converge for $\Re(s) \gg 0$ and have meromorphic continuation to all $s$, and the greatest common divisor of the values of either integral is $L(\pi \times \chi, s)$. By part (iii) of the cited theorem, we also have a functional equation for $\Psi_j$; let us write
\[ w_r = \smallthreemat{}{}{1}{}{\iddots}{}{1}{}{} \in \GL_r, \qquad w_{r, t} = \stbt{\mathrm{id}_t}{}{}{w_{r-t}}\ (0 \le t \le r).\]
For $W \in \cW(\pi)$, we write $\tilde{W}(g) = W(w_r {}^t\! g^{-1})$; the functions $\{ \tilde{W} : W \in \cW(\pi, \psi)\}$ form the Whittaker model $\cW(\pi^\vee, \psi^{-1})$. Then for $j = 0, 1$ we have the functional equation
\begin{equation}\label{eq:JPSS functional equation}
\Psi_{(1-j)}(w_{3, 1} \cdot \tilde{W}, \chi^{-1}, 1-s) =
\gamma(\pi \times \chi, \psi, s) \cdot \Psi_j(W, \chi, s).
\end{equation}

We now write the torus integral $y(\cdots)$ (from \cref{def:y}) in terms of $\Psi_0$:
\begin{align*}
y(W,& \mathrm{id}, \chi_2, s_1, s_2)(g)\defeq
|\det g|^{s_1-\tfrac{1}{2}} \int_{a,b} W\left(\sdthree{a}{1}{-b^{-1}} w_{3, 1} g\right)\\
&\hspace{40pt}|a|^{s_1 + s_2 - 3/2} |b|^{s_2 - s_1 - 1/2} \chi_2(b)\  \mathrm{d}^\times\! a\, \mathrm{d}^\times\! b \\
&= |\det g|^{s_1-\tfrac{1}{2}} \int_a |a|^{s_1 + s_2 - 3/2} \\
&\hspace{40pt}\left[ \int_b W\left(w_3 \sdthree{-b^{-1}}{1}{1} w_3 \sdthree{a}{1}{1} w_{3, 1} g\right) |b|^{s_2 - s_1 - 1/2} \chi_2(b)\mathrm{d}^\times b \right] \mathrm{d}^\times\! a\\
&= |\det g|^{s_1-\tfrac{1}{2}} \int_a |a|^{s_1 + s_2 - 3/2} \\
&\hspace{10pt}\left[ \int_b w_{3, 1} \cdot \left[ w_{3, 1} w_3 \sdthree{a}{1}{1} w_{3, 1} g \cdot W\right]^{\sim}\left( \sdthree{-b}{1}{1} \right) |b|^{s_2 - s_1 - 1/2} \chi_2(b)\mathrm{d}^\times b \right] \mathrm{d}^\times\! a\\
&= \chi_2(-1)|\det g|^{s_1-\tfrac{1}{2}} \int_a |a|^{s_1 + s_2 - 3/2}\\
&\hspace{40pt}\Psi_0\left( w_{3, 1} \cdot \left[ w_{3, 1} w_3 \sdthree{a}{1}{1} w_{3, 1} g \cdot W\right]^{\sim}, \chi_2, s_2 - s_1 + \tfrac{1}{2}\right)\ \mathrm{d}^\times\! a,
\end{align*}
with the $\chi_2(-1)$ arising from the change of variables $b \leftrightarrow -b$. Applying the functional equation \eqref{eq:JPSS functional equation} to the inner term, we can write this as
\begin{multline*}
y(W, \mathrm{id}, \chi_2, s_1, s_2)(g) =  \\
\gamma(\pi \times \chi_2^{-1},\psi, s_1 - s_2 + \tfrac{1}{2})\chi_2(-1)|\det g|^{s_1-\tfrac{1}{2}} \int_a |a|^{s_1 + s_2 - 3/2}\\
\Psi_1\left(  w_{3, 1} w_3 \sdthree{a}{1}{1} w_{3, 1} g \cdot W, \chi_2^{-1}, s_1 - s_2 - \tfrac{1}{2}\right)\ \mathrm{d}^\times a.
\end{multline*}
The integral expands to
\begin{align*}
\int_a &|a|^{s_1 + s_2 - 3/2} \int_{b,x} |b|^{s_1 - s_2 - \tfrac{1}{2}} \chi_2(b)^{-1} W\left(
\smallthreemat{b}{}{} {x}{1}{} {}{}{1}
w_{3, 1} w_3 \sdthree{a}{1}{1} w_{3, 1} g\right) \ \mathrm{d}x\, \mathrm{d}^\times\! a\, \mathrm{d}^\times\! b\\
&= \int_{a,b,x} |a|^{s_1 + s_2 - 3/2} |b|^{s_1 - s_2 - \tfrac{1}{2}} \chi_2(b)^{-1} W\left(\iota(w_2 \stbt{1}{x/b}{}{1}\stbt{a}{}{}{b} g)\right) \ \mathrm{d}x\, \mathrm{d}^\times a\, \mathrm{d}^\times b\\
&= \int_{a,b,x} |a|^{s_1 + s_2 - 3/2} |b|^{s_1 - s_2 + \tfrac{1}{2}} \chi_2(b)^{-1} W\left(\iota(w_2 \stbt{1}{x}{}{1}\stbt{a}{}{}{b} g)\right) \ \mathrm{d}x\, \mathrm{d}^\times a\, \mathrm{d}^\times b.
\end{align*}
By \eqref{eq:Y and y}, integrating this against $f_{\Phi}$ gives $Y(\cdots)$. In this pairing, we obtain an inner integral over $B_2$ and an outer integral over $B_2 \backslash \GL_2$; so this rearranges into
\begin{multline*}
Y(W, \Phi; \mathrm{id}, \chi_2, s_1, s_2) = \\
\chi_2(-1)\gamma(\pi \times \chi_2^{-1},\psi, s_1 - s_2 + \tfrac{1}{2}) \int_{\GL_2(F)} W(\iota(g))f_{\Phi}(w_2 g) |\det g|^{s_1-\tfrac{1}{2}}\ \mathrm{d}g.
\end{multline*}
Splitting into an integral over $N_2$ and an integral over $N_2 \backslash \GL_2$, the $N_2$ factor acts on $W$ via $\psi$, so the integral over $\GL_2$ equals
\[
\int_{N_2 \backslash \GL_2(F)} W(\iota(g)) |\det g|^{s_1-\tfrac{1}{2}} \left(\int_F f_{\Phi}(w_2 \stbt 1 x 0 1 g; \chi_2, s_2) \psi(x)\ \mathrm{d}x\right)\mathrm{d}g.
\]
Since $w_2 = \smallmatrd{1}{}{}{-1}\smallmatrd{}{1}{-1}{}$, and $f_{\Phi}$ transforms as $\chi_2(-1)$ under $\smallmatrd{1}{}{}{-1}$, by \cite[\S8.1]{LPSZ19} the second integral is precisely $\chi_2(-1)W_\Phi(g;\chi_2,s_2)$. The $\chi_2(-1)$'s cancel; we conclude since
\begin{align*}
Y(W, \Phi; \mathrm{id}, \chi_2, s_1, s_2) &= \gamma(\pi \times \chi_2^{-1},\psi, s_1 - s_2 + \tfrac{1}{2})\\
&\hspace{20pt}  \times \int_{N_2 \backslash \GL_2(F)} W(\iota(g))  W_\Phi(g; \chi_2,s_2) |\det g|^{s_1-\tfrac{1}{2}}\mathrm{d}g\\
&=  \gamma(\pi \times \chi_2^{-1},\psi, s_1 - s_2 + \tfrac{1}{2})  Z(W, \Phi; \mathrm{id}, \chi_2, s_1, s_2).
\qed\end{align*}

\subsection*{Glossary of notation/terminology}

\footnotesize \hspace{1pt}
\addcontentsline{toc}{section}{Glossary of notation}
\begin{multicols}{2}
	\noindent $a$\dotfill integer $\geq 0$; weight of $\Pi$ (\cref{sec:cohomological})\\
	$\alpha_p = \sigma_p(p)$ \dotfill (for $\sigma_p = P_1$-refinement)\\
	$B \subset \GL_3$\dotfill upper-triangular Borel\\
	$\beta_p$  \dotfill Satake parameter (\cref{sec:L-function})\\
	$\mathrm{br}^{[a,j]}$ \dotfill branching law (\cref{sec:branching laws})\\
	$\Crit_p^\pm(\Pi)$ \dotfill critical values (\cref{sec:rationality})\\
	$c \in \Z$ \dotfill aux.\ integer prime to $6p$\\
	$\chi$ \dotfill Dirichlet character (\cref{sect:dirichlet})\\
	$\chihat$ \dotfill Hecke character (\cref{sect:dirichlet})\\
	$\Delta_n$ \dotfill $(\Z/p^n\Z)^\times$\\
	$E$\dotfill suff.\ large number field (\cref{sec:periods})\\
	$E_\Phi$ \dotfill adelic Eisenstein series (\cref{sec:Eis-GL2})\\
	$E_\Phi^{j+2,\chi}$\dotfill \eqref{eq:E j chi}\\
	$\cE_{\Phif}^{j+2}$ \dotfill classical Eisenstein series (\cref{sec:classical eisenstein})\\
	$\cE_{\Phif}^{j+2,\chi}$ \dotfill \eqref{eq:classical E j chi}\\
	$\mathrm{Eis}_{\Phif}^{j+2}$ \dotfill Betti--Eisenstein class (Cor.\ \ref{cor:betti-eisenstein})\\
	${}_c\mathrm{Eis}_{\Phif}^{j+2}$ \dotfill integral $\mathrm{Eis}_{\Phif}^{j+2}$ (Thm.\ \ref{thm:integral betti eisenstein})\\
	${}_c \cE\cI_{\Phi^{(p)}}$\dotfill Eisenstein--Iwasawa class (\cref{sec:Eisenstein interpolation})\\
	$e_\infty(\Pi_\infty \times \eta_\infty,t)$\dotfill mod.\ Euler factor at $\infty$ (\cref{sec:rationality})\\
	$\widetilde{e}_\infty(\zeta_\infty \times \eta_\infty,t)$ \dotfill Mahnkopf factor at $\infty$ \eqref{eq:interpolation intermediate 2}\\
	$e_p(\Pi_p \times \eta_{p}, t)$ \dotfill mod.\ Euler factor at $p$ (\cref{sec:coates--perrin-riou})\\
	$\eta$ \dotfill Dirichlet char.\ of cond.\ $p^n$\\
	$\eta_2$ \dotfill  Dir.\ char., cond. $D$ prime to $p$ (Not.\ \ref{not:eta_2})\\
	$G(\eta)$ \dotfill Gauss sum (\cref{sect:dirichlet})\\
	$\gamma_p$  \dotfill Satake parameter (\cref{sec:L-function})\\
	$H$ \dotfill $\GL_2 \times \GL_1$\\
	$I(-,-)$ \dotfill principal series for $\GL_2$ (\cref{sec:Eis-GL2})\\
	$I_j(\chihat)$ \dotfill $I(\|\cdot\|^{-1/2}, \chihat^{-1}\|\cdot\|^{k+1/2})$ \eqref{eq:I_j(chi)}\\
	$\iota : H \hookrightarrow \GL_3$ \dotfill \eqref{eq:iota}\\
	$\cH_J$ \dotfill symmetric space for $J$ (\S\ref{sec:symmetric spaces})\\
	$J$ \dotfill general reductive group\\
%	$\cJ_{p,1} \subset \GL_3(\Zp)$ \dotfill parahoric for $P_1$ \eqref{eq:parahoric}\\
	$j$ \dotfill integer $0 \leq j \leq a$ \\
	$K_{\GL_n,\infty} \subset \GL_n(\R)$ \dotfill max.\ cpct subgroup\\
	$K_{n,\infty}^\circ \subset \GL_n(\R)$\dotfill $= K_{\GL_n,\infty}^\circ Z_{\GL_n,\infty}^\circ$\\
	$L$ \dotfill suff.\ large $p$-adic field (Rem.\ \ref{rem:L})\\
	$\lambda = (a,0,-a)$\dotfill weight of $\Pi$\\
	$N \subset \GL_3$ \dotfill upper-triangular unipotent\\
	$N_{P_i}$ \dotfill unipotent in $P_i$\\
	%$N_\Pi$ \dotfill conductor of $\Pi$ (\cref{sec:conductors})\\
	$\nu_1 : H \to \GL_1$ \dotfill $(\gamma,z) \mapsto \det(\gamma)/z$\\
	$\nu_2 : H \to \GL_1$ \dotfill $(\gamma,z) \mapsto z$\\
	$\nu_{1,(n)} : \widetilde{Y}^H \to \Delta_n$ \dotfill \eqref{eq:nu_n}\\
	$\Omega_\Pi^-$ \dotfill \eqref{eq:Omega}\\
	$\Omega_\Pi^+$ \dotfill $= \Omega_{\Pi^\vee}^-$ (\cref{sec:functional equation}) \\
	$\omegahat_\Pi$ \dotfill central char. of $\Pi$\\
	$P_1,P_2 \subset \GL_3$ \dotfill max.\ standard parabolics (\cref{sec:ordinarity})\\
	$\Pi$ \dotfill RACAR of $\GL_3$\\
	$\Phi, \Phif$ \dotfill Schwartz functions (\cref{sec:schwartz})\\
	$\Phi\subf^t$ \dotfill specific Schwartz function (\cref{sec:Eisenstein interpolation})\\
	$\phi_\Pi$\dotfill map $\cW_\psi(\Pif) \to \hc{2}(-)$ (\cref{sec:auto cohom classes}) \\
	$\mathrm{pr}_{\GL_2}$ \dotfill projection $H \to \GL_2$\\
	$\varphi \in \Pi$ \dotfill cusp form (in $L^2_0(\GL_3(\Q)\backslash\GL_3(\A))$)\\
	$\varphi_{\infty,r,s,\alpha}$ \dotfill vectors at $\infty$ (\cref{sec:diff form pi})\\
	$\psi$ \dotfill additive character of $\Q\backslash\A$ (\cref{sect:dirichlet})\\
	RACAR \dotfill regular alg.\ cusp.\ auto.\ rep.\\
	$R_\chi$ \dotfill $\chi^{-1}$-projector \eqref{eq:R_chi}\\
	$\cS,\cS_0$ \dotfill spaces of Schwartz functions (\cref{sec:schwartz})\\
	$\mathrm{tw}_j$ \dotfill twist by $\|\nu_1\|^{-j}$ (\cref{sec:construction branching})\\
	$\tau = \tau_1 \in \GL_3(\Qp)$ \dotfill $\mathrm{diag}(p,1,1)$ \eqref{eq:tau}\\
		$\tau_2 \in \GL_3(\Qp)$ \dotfill $\mathrm{diag}(p,p,1)$\\
	$\Theta_\Pi$ \dotfill cohomological period (\cref{sec:periods})\\
	$U_{p,1},U_p$ \dotfill (normalised) Hecke ops.\ at $p$ (\eqref{eq:tau}, \S\ref{sec:local systems})\\
	$U_{p,1}',U_p'$ \dotfill (normalised) adjoint Hecke ops.\  (\cref{sec:local systems})\\
	$\cU^J \subset J(\A\subf)$ \dotfill open compact level\\
	$\cU^{(p)}$ \dotfill prime-to-$p$ level\\
	$\cU_n \subset \GL_3(\Zp)$ \dotfill Def.\ \ref{def:U_n}\\
	%$\cU_1^{\GL_3}(N)$ \dotfill Whittaker new lvl.\ for $\Pi$ (\cref{sec:conductors})\\
	$\cU_1^H(p^t) \subset H(\A\subf)$ \dotfill (\cref{sec:level tower})\\
	$u$ \dotfill $\smallthreemat{1}{0}{1}{}{1}{0}{}{}{1}\smallthreemat{1}{}{}{}{}{-1}{}{1}{}$\\
	$V_\lambda$ \dotfill highest weight rep. of $\GL_3$ (\cref{sec:alg rep})\\
	$V_\mu^J$ \dotfill highest weight rep. of $J$\\
	$\cV_n \subset \GL_3(\Zp)$ \dotfill Def.\ \ref{def:V_n}\\
	$\sV_*^\ast$ \dotfill local system attached to $V_*^\ast$ (\cref{sec:local systems}) \\
	$\cW_\psi$ \dotfill Whittaker transform (\cref{sec:whittaker})\\
	$W$ \dotfill element in $\cW_\psi(\Pi)$\\
	${}_c\Xi^{[a,j]}_n, {}_c\Xi^{[a,j]}$ \dotfill $\h^3(Y^{\GL_3})$-valued measures (\S\ref{sec:pairing with cusp form})\\
	$\Xi^{[a,j]}_n, \Xi^{[a,j]}$ \dotfill measures (\cref{sec:getting rid of c})\\
	${}_c \xi^{[a,j]}_n, \xi^{[a,j]}_n$ \dotfill Def.\ \ref{def:xi aj}\\
	$Y^J(\cU)$ \dotfill loc.\ sym.\ space for $J$ (\S \ref{sec:symmetric spaces})\\
	$Y(-)$ \dotfill local zeta integral (Def.\ \ref{def:local zeta Y})\\
	$\widetilde{Y}^H$ \dotfill modified space for $H$ \eqref{eq:tilde Y}\\
	$Z_{\GL_n,\infty}$ \dotfill centre of $\GL_n(\R)$\\
	$Z(-)$ \dotfill local zeta integral (Def.\ \ref{def:local zeta Z})\\
	$\widetilde{Z}(-)$ \dotfill modified $Z^{k+2,\chi}$ (Thm.\ \ref{thm:jpss})\\
	$\cZ(-)$ \dotfill global zeta integral (Def.\ \ref{def:global zeta integral})\\
	${}_c z^{[a,j]}_n, z^{[a,j]}_n$ \dotfill Def.\ \ref{def:zaj}\\
	$\zeta_\infty$ \dotfill  $\in \h^2(\mathfrak{gl}_3,K_{3,\infty}^\circ; \Pi_\infty \otimes V_{\lambda}^{\vee}(\C) )$ \eqref{eq:omega infinity}\\
	$\langle-,-\rangle_{\cU}$ \dotfill Poincar\'e duality pairing (\cref{sec:cup product pairing})\\
	$\langle-,-\rangle_{a,j}$ \dotfill branching law pairing \eqref{eq:branching pairing}\\
	$(-)^\circ$ \dotfill identity component
\end{multicols}

%%%%%%%%%%%%%%%%%%%%%%%%%%%%%%%%%%%%%%%%%%%%%%%%%%%%%%%%%%%%%%%%%%%%
%
%                                REFERENCES
%
%%%%%%%%%%%%%%%%%%%%%%%%%%%%%%%%%%%%%%%%%%%%%%%%%%%%%%%%%%%%%%%%%%%%
\footnotesize
\setlength{\parskip}{0ex}
\renewcommand{\baselinestretch}{1}
\renewcommand{\refname}{\normalsize References}

\renewcommand{\MRhref}[1]{\relax}
\providecommand{\bysame}{\leavevmode\hbox to3em{\hrulefill}\thinspace}
\providecommand{\MR}[1]{}
\providecommand{\href}[2]{#2}
\newcommand{\articlehref}[2]{\href{#1}{#2}}

\Addresses


\begin{thebibliography}{GKTV97}

\bibitem[AG94]{AG94}
\textsc{A.~Ash and D.~Ginzburg},
  \articlehref{https://0-doi-org.pugwash.lib.warwick.ac.uk/10.1007/BF01231556}{\emph{{$p$}-adic
  {$L$}-functions for {${\rm GL}(2n)$}}}, Invent. Math. \textbf{116} (1994),
  no.~1-3, 27--73. \MR{1253188}

\bibitem[AP08]{AP08}
\textsc{A.~Ash and D.~Pollack},
  \articlehref{https://doi.org/10.1142/S1793042108001602}{\emph{Everywhere
  unramified automorphic cohomology for {${\rm SL}_3(\mathbb{Z})$}}}, Int. J.
  Number Theory \textbf{4} (2008), no.~4, 663--675. \MR{2441799}

\bibitem[BDW]{BDW20}
\textsc{D.~{Barrera Salazar}, M.~Dimitrov, and C.~Williams}, \emph{On $p$-adic
  {$L$}-functions for $\mathrm{{GL}}(2n)$ in finite slope {S}halika families},
  Preprint: \url{https://arxiv.org/abs/2103.10907}.

\bibitem[Bei86]{Bei86}
\textsc{A.~A. Beilinson},
  \articlehref{https://0-doi-org.pugwash.lib.warwick.ac.uk/10.1090/conm/055.1/862627}{\emph{Higher
  regulators of modular curves}}, Applications of algebraic {$K$}-theory to
  algebraic geometry and number theory, {P}art {I}, {II} ({B}oulder, {C}olo.,
  1983), Contemp. Math., vol.~55, Amer. Math. Soc., Providence, RI, 1986,
  pp.~1--34. \MR{862627}

\bibitem[BKL18]{BKL14}
\textsc{A.~Beilinson, G.~Kings, and A.~Levin},
  \articlehref{https://0-doi-org.pugwash.lib.warwick.ac.uk/10.1007/s00208-018-1645-4}{\emph{Topological
  polylogarithms and {$p$}-adic interpolation of {$L$}-values of totally real
  fields}}, Math. Ann. \textbf{371} (2018), no.~3-4, 1449--1495. \MR{3831278}

\bibitem[Bum97]{Bump}
\textsc{D.~Bump},
  \articlehref{https://0-doi-org.pugwash.lib.warwick.ac.uk/10.1017/CBO9780511609572}{\emph{Automorphic
  forms and representations}}, Cambridge Studies in Advanced Mathematics,
  vol.~55, Cambridge University Press, Cambridge, 1997. \MR{1431508}

\bibitem[BG14]{BG14}
\textsc{K.~Buzzard and T.~Gee},
  \articlehref{https://0-doi-org.pugwash.lib.warwick.ac.uk/10.1017/CBO9781107446335.006}{\emph{The
  conjectural connections between automorphic representations and {G}alois
  representations}}, Automorphic forms and {G}alois representations. {V}ol. 1,
  London Math. Soc. Lecture Note Ser., vol. 414, Cambridge Univ. Press,
  Cambridge, 2014, pp.~135--187. \MR{3444225}

\bibitem[Clo90]{Clo90}
\textsc{L.~Clozel}, \emph{Motifs et formes automorphes: applications du
  principe de fonctorialit\'{e}}, Automorphic forms, {S}himura varieties, and
  {$L$}-functions, {V}ol. {I} ({A}nn {A}rbor, {MI}, 1988), Perspect. Math.,
  vol.~10, Academic Press, Boston, MA, 1990, pp.~77--159. \MR{1044819}

\bibitem[Coa89]{coates89}
\textsc{J.~Coates},
  \articlehref{https://0-doi-org.pugwash.lib.warwick.ac.uk/10.1007/BF02585471}{\emph{On
  {$p$}-adic {$L$}-functions attached to motives over {${\bf Q}$}. {II}}}, Bol.
  Soc. Brasil. Mat. (N.S.) \textbf{20} (1989), no.~1, 101--112. \MR{1129081}

\bibitem[CPR89]{coatesperrinriou89}
\textsc{J.~Coates and B.~Perrin-Riou}, \emph{On {$p$}-adic {$L$}-functions
  attached to motives over {${\bf Q}$}}, Algebraic number theory, Adv. Stud.
  Pure Math., vol.~17, Academic Press, Boston, MA, 1989, pp.~23--54.
  \MR{1097608}

\bibitem[DD97]{DD97}
\textsc{A.~Dabrowski and D.~Delbourgo},
  \articlehref{https://0-doi-org.pugwash.lib.warwick.ac.uk/10.1112/S0024611597000191}{\emph{{$S$}-adic
  {$L$}-functions attached to the symmetric square of a newform}}, Proc. London
  Math. Soc. (3) \textbf{74} (1997), no.~3, 559--611. \MR{1434442}

\bibitem[Del79]{Del79}
\textsc{P.~Deligne}, \emph{Valeurs de fonctions {$L$} et p\'{e}riodes
  d’int\'{e}grales}, Proc. Symp. Pure Math. \textbf{33} (1979), 313--346.

\bibitem[DJR20]{DJR18}
\textsc{M.~Dimitrov, F.~Januszewski, and A.~Raghuram}, \emph{${L}$-functions of
  {$\mathrm{GL}(2n)$}: {$p$}-adic properties and nonvanishing of twists},
  Compositio Math. \textbf{156} (2020), no.~12, 2437--2468.

\bibitem[Eme06]{EmertonJacquetI}
\textsc{M.~Emerton},
  \articlehref{https://0-doi-org.pugwash.lib.warwick.ac.uk/10.1016/j.ansens.2006.08.001}{\emph{Jacquet
  modules of locally analytic representations of {$p$}-adic reductive groups.
  {I}. {C}onstruction and first properties}}, Ann. Sci. \'{E}cole Norm. Sup.
  (4) \textbf{39} (2006), no.~5, 775--839. \MR{2292633}

\bibitem[GKTV97]{GKTV97}
\textsc{B.~v. Geemen, W.~v.~d. Kallen, J.~Top, and A.~Verberkmoes},
  \articlehref{http://0-projecteuclid.org.pugwash.lib.warwick.ac.uk/euclid.em/1047650002}{\emph{Hecke
  eigenforms in the cohomology of congruence subgroups of {${\rm
  SL}(3,\mathbb{Z})$}}}, Experiment. Math. \textbf{6} (1997), no.~2, 163--174.
  \MR{1474576}

\bibitem[Geh18]{Geh18}
\textsc{L.~Gehrmann},
  \articlehref{https://0-doi-org.pugwash.lib.warwick.ac.uk/10.1007/s11856-018-1694-0}{\emph{On
  {S}halika models and {$p$}-adic {$L$}-functions}}, Israel J. Math.
  \textbf{226} (2018), no.~1, 237--294. \MR{3819693}

\bibitem[Gel75]{Gel75}
\textsc{S.~S. Gelbart}, \emph{Automorphic forms on ad\`ele groups}, Annals of
  Mathematics Studies, vol. No. 83, Princeton University Press, Princeton, NJ;
  University of Tokyo Press, Tokyo, 1975. \MR{379375}

\bibitem[Ger15]{Ger15}
\textsc{A.~Geroldinger},
  \articlehref{https://0-doi-org.pugwash.lib.warwick.ac.uk/10.1007/s11139-015-9673-0}{\emph{{$p$}-adic
  automorphic {$L$}-functions on {$\mathrm{GL}(3)$}}}, Ramanujan J. \textbf{38}
  (2015), no.~3, 641--682. \MR{3423019}

\bibitem[GW09]{goodmanwallach}
\textsc{R.~Goodman and N.~R. Wallach},
  \articlehref{https://0-doi-org.pugwash.lib.warwick.ac.uk/10.1007/978-0-387-79852-3}{\emph{Symmetry,
  representations, and invariants}}, Graduate Texts in Mathematics, vol. 255,
  Springer, Dordrecht, 2009. \MR{2522486}

\bibitem[HT17]{HT17}
\textsc{D.~Hansen and J.~A. Thorne},
  \articlehref{https://0-doi-org.pugwash.lib.warwick.ac.uk/10.1007/s00029-017-0303-0}{\emph{On
  the {$\mathrm{GL}_n$}-eigenvariety and a conjecture of {V}enkatesh}}, Selecta
  Math. (N.S.) \textbf{23} (2017), no.~2, 1205--1234. \MR{3624909}

\bibitem[HLTT16]{HLTT}
\textsc{M.~Harris, K.-W. Lan, R.~Taylor, and J.~Thorne},
  \articlehref{https://0-doi-org.pugwash.lib.warwick.ac.uk/10.1186/s40687-016-0078-5}{\emph{On
  the rigid cohomology of certain {S}himura varieties}}, Res. Math. Sci.
  \textbf{3} (2016), Paper No. 37, 308. \MR{3565594}

\bibitem[HM15]{helmmoss15}
\textsc{D.~Helm and G.~Moss},
  \articlehref{http://arxiv.org/abs/1510.08743}{\emph{Deligne--{L}anglands
  gamma factors in families}}, \path{arXiv:1510.08743}.

\bibitem[Hid90]{Hid90}
\textsc{H.~Hida}, \emph{{$p$}-adic {$L$}-functions for base change lifts of
  {${\rm GL}_2$} to {${\rm GL}_3$}}, Automorphic forms, {S}himura varieties,
  and {$L$}-functions, {V}ol. {II} ({A}nn {A}rbor, {MI}, 1988), Perspect.
  Math., vol.~11, Academic Press, Boston, MA, 1990, pp.~93--142. \MR{1044829}

\bibitem[JS81]{JS_EulerProducts_II}
\textsc{H.~Jacquet and J.~A. Shalika},
  \articlehref{https://0-doi-org.pugwash.lib.warwick.ac.uk/10.2307/2374050}{\emph{On
  {E}uler products and the classification of automorphic forms. {II}}}, Amer.
  J. Math. \textbf{103} (1981), no.~4, 777--815. \MR{623137}

\bibitem[Jac72]{jacquet72}
\textsc{H.~Jacquet}, \emph{Automorphic forms on {${\rm GL}(2)$}. {P}art {II}},
  Lecture Notes in Mathematics, Vol. 278, Springer-Verlag, Berlin-New York,
  1972. \MR{0562503}

\bibitem[JPSS79]{JPSS-L-factors}
\textsc{H.~Jacquet, I.~I. Piatetski-Shapiro, and J.~Shalika}, \emph{Facteurs
  {$L$} et {$\varepsilon $} du groupe lin\'{e}aire}, C. R. Acad. Sci. Paris
  S\'{e}r. A-B \textbf{289} (1979), no.~2, A59--A61. \MR{549066}

\bibitem[JPSS83]{JPSS-RankinSelberg}
\bysame,
  \articlehref{http://dx.doi.org/10.2307/2374264}{\emph{Rankin--{S}elberg
  convolutions}}, Amer. J. Math. \textbf{105} (1983), no.~2, 367--464.
  \MR{701565}

\bibitem[Jan]{Jan-Non-Abelian}
\textsc{F.~Januszewski}, \emph{Non-abelian $p$-adic {R}ankin--{S}elberg
  ${L}$-functions and non-vanishing of central ${L}$-values}, Amer. J. Math.,
  To appear: \url{https://arxiv.org/abs/1708.02616}.

\bibitem[KS13]{KS13}
\textsc{H.~Kasten and C.-G. Schmidt},
  \articlehref{https://0-doi-org.pugwash.lib.warwick.ac.uk/10.1142/S1793042112501345}{\emph{The
  critical values of {R}ankin-{S}elberg convolutions}}, Int. J. Number Theory
  \textbf{9} (2013), no.~1, 205--256. \MR{2997500}

\bibitem[KMS00]{KMS00}
\textsc{D.~Kazhdan, B.~Mazur, and C.-G. Schmidt},
  \articlehref{https://0-doi-org.pugwash.lib.warwick.ac.uk/10.1515/crll.2000.019}{\emph{Relative
  modular symbols and {R}ankin-{S}elberg convolutions}}, J. Reine Angew. Math.
  \textbf{519} (2000), 97--141. \MR{1739728}

\bibitem[Kim03]{Kim03}
\textsc{H.~H. Kim},
  \articlehref{https://0-doi-org.pugwash.lib.warwick.ac.uk/10.1090/S0894-0347-02-00410-1}{\emph{Functoriality
  for the exterior square of {${\rm GL}_4$} and the symmetric fourth of {${\rm
  GL}_2$}}}, J. Amer. Math. Soc. \textbf{16} (2003), no.~1, 139--183, With
  appendix 1 by Dinakar Ramakrishnan and appendix 2 by Kim and Peter Sarnak.
  \MR{1937203}

\bibitem[KLZ17]{KLZ17}
\textsc{G.~Kings, D.~Loeffler, and S.~L. Zerbes},
  \articlehref{https://0-doi-org.pugwash.lib.warwick.ac.uk/10.4310/CJM.2017.v5.n1.a1}{\emph{Rankin-{E}isenstein
  classes and explicit reciprocity laws}}, Camb. J. Math. \textbf{5} (2017),
  no.~1, 1--122. \MR{3637653}

\bibitem[KLZ20]{KLZ20}
\bysame,
  \articlehref{https://0-doi-org.pugwash.lib.warwick.ac.uk/10.1353/ajm.2020.0002}{\emph{Rankin-{E}isenstein
  classes for modular forms}}, Amer. J. Math. \textbf{142} (2020), no.~1,
  79--138. \MR{4060872}

\bibitem[Kna94]{Kna94}
\textsc{A.~W. Knapp},
  \articlehref{https://0-doi-org.pugwash.lib.warwick.ac.uk/10.1016/0022-4049(94)00057-p}{\emph{Local
  {L}anglands correspondence: the {A}rchimedean case}}, Motives ({S}eattle,
  {WA}, 1991), Proc. Sympos. Pure Math., vol.~55, Amer. Math. Soc., Providence,
  RI, 1994, pp.~393--410. \MR{1265560}

\bibitem[KL64]{KL64}
\textsc{T.~Kubota and H.-W. Leopoldt}, \emph{Eine {$p$}-adische {T}heorie der
  {Z}etawerte. {I}. {E}inf\"{u}hrung der {$p$}-adischen {D}irichletschen
  {$L$}-{F}unktionen}, J. Reine Angew. Math. \textbf{214/215} (1964), 328--339.
  \MR{163900}

\bibitem[Loe21]{loeffler-parabolics}
\textsc{D.~Loeffler},
  \articlehref{https://doi.org/10.1007/s10884-020-09844-5}{\emph{Spherical
  varieties and norm relations in {I}wasawa theory}}, J. Th\'{e}or. Nombres
  Bordeaux \textbf{33} (2021), no.~3, part 2, 1021--1043, Iwasawa 2019 special
  issue. \MR{4402388}

\bibitem[LPSZ21]{LPSZ19}
\textsc{D.~Loeffler, V.~Pilloni, C.~Skinner, and S.~L. Zerbes},
  \articlehref{https://doi.org/10.1215/00127094-2021-0049}{\emph{Higher {H}ida
  theory and {$p$}-adic {$L$}-functions for {$\mathrm{GSp}_4$}}}, Duke Math. J.
  \textbf{170} (2021), no.~18, 4033--4121. \MR{4348233}

\bibitem[LRZ]{spherical_II}
\textsc{D.~Loeffler, R.~Rockwood, and S.~L. Zerbes}, \emph{Spherical varieties
  and $p$-adic families of cohomology classes}, Preprint:
  \url{https://arxiv.org/abs/2106.16082}.

\bibitem[LSZ22a]{LSZ-unitary}
\textsc{D.~Loeffler, C.~Skinner, and S.~L. Zerbes},
  \articlehref{https://doi.org/10.1007/s00208-021-02224-4}{\emph{An {E}uler
  system for {$\rm GU(2,1)$}}}, Math. Ann. \textbf{382} (2022), no.~3-4,
  1091--1141. \MR{4403219}

\bibitem[LSZ22b]{LSZ17}
\bysame, \articlehref{https://doi.org/10.4171/jems/1124}{\emph{Euler systems
  for {${\mathrm{GSp}}(4)$}}}, J. Eur. Math. Soc. (JEMS) \textbf{24} (2022),
  no.~2, 669--733. \MR{4382481}

\bibitem[LW20]{LW18}
\textsc{D.~Loeffler and C.~Williams},
  \articlehref{https://0-doi-org.pugwash.lib.warwick.ac.uk/10.2140/ant.2020.14.1669}{\emph{$p$-adic
  {A}sai {$L$}-functions of {B}ianchi modular forms}}, Algebra Number Theory
  \textbf{14} (2020), no.~7, 1669--1710. \MR{4150247}

\bibitem[LZ16]{LZ16}
\textsc{D.~Loeffler and S.~L. Zerbes},
  \articlehref{http://dx.doi.org/10.1186/s40687-016-0077-6}{\emph{Rankin--{E}isenstein
  classes in {C}oleman families}}, Res. Math. Sci. \textbf{3} (2016), no.~29,
  special collection in honour of Robert F.\ Coleman.

\bibitem[LZ19]{LZ-symmetric}
\bysame,
  \articlehref{https://0-doi-org.pugwash.lib.warwick.ac.uk/10.1515/crelle-2016-0064}{\emph{Iwasawa
  theory for the symmetric square of a modular form}}, J. Reine Angew. Math.
  \textbf{752} (2019), 179--210. \MR{3975641}

\bibitem[LZ25]{LZ-distinguished}
\textsc{D.~Loeffler and S.~L. Zerbes}, \emph{Poles of {$p$}-adic {A}sai
  l-functions and distinguished representations}, preprint, 2025.

\bibitem[Mah98]{Mah98}
\textsc{J.~Mahnkopf},
  \articlehref{https://0-doi-org.pugwash.lib.warwick.ac.uk/10.1515/crll.1998.045}{\emph{Modular
  symbols and values of {$L$}-functions on {${\rm GL}_3$}}}, J. Reine Angew.
  Math. \textbf{497} (1998), 91--112. \MR{1617428}

\bibitem[Mah00]{Mah00}
\bysame,
  \articlehref{https://0-doi-org.pugwash.lib.warwick.ac.uk/10.1023/A:1026569231434}{\emph{Eisenstein
  cohomology and the construction of {$p$}-adic analytic {$L$}-functions}},
  Compositio Math. \textbf{124} (2000), no.~3, 253--304. \MR{1809337}

\bibitem[Mah05]{Mah05}
\bysame,
  \articlehref{https://0-doi-org.pugwash.lib.warwick.ac.uk/10.1017/S1474748005000186}{\emph{Cohomology
  of arithmetic groups, parabolic subgroups and the special values of
  {$L$}-functions on {${\rm GL}_n$}}}, J. Inst. Math. Jussieu \textbf{4}
  (2005), no.~4, 553--637. \MR{2171731}

\bibitem[MSD74]{MSD74}
\textsc{B.~Mazur and P.~Swinnerton-Dyer}, \emph{Arithmetic of {W}eil curves},
  Invent. Math. \textbf{25} (1974), 1 -- 61.

\bibitem[Nam22]{NamAsai}
\textsc{K.~Namikawa},
  \articlehref{https://doi.org/10.1007/s00229-021-01341-3}{\emph{A construction
  of {$p$}-adic {A}sai {$L$}-functions for {$\rm GL_2$} over {CM} fields}},
  Manuscripta Math. \textbf{169} (2022), no.~3-4, 461--517. \MR{4493647}

\bibitem[Pan94]{Pan94}
\textsc{A.~A. Panchishkin},
  \articlehref{http://www.numdam.org/item?id=AIF_1994__44_4_989_0}{\emph{Motives
  over totally real fields and {$p$}-adic {$L$}-functions}}, Ann. Inst. Fourier
  (Grenoble) \textbf{44} (1994), no.~4, 989--1023. \MR{1306547}

\bibitem[Pin90]{pink90}
\textsc{R.~Pink}, \emph{Arithmetical compactification of mixed {S}himura
  varieties}, Bonner Mathematische Schriften [Bonn Mathematical Publications],
  vol. 209, Universit\"{a}t Bonn, Mathematisches Institut, Bonn, 1990,
  Dissertation, Rheinische Friedrich-Wilhelms-Universit\"{a}t Bonn, Bonn, 1989.
  \MR{1128753}

\bibitem[Pra93]{prasad93}
\textsc{D.~Prasad},
  \articlehref{http://dx.doi.org/10.1215/S0012-7094-93-06908-6}{\emph{On the
  decomposition of a representation of {$\operatorname{GL}(3)$} restricted to
  {$\operatorname{GL}(2)$} over a {$p$}-adic field}}, Duke Math. J. \textbf{69}
  (1993), no.~1, 167--177. \MR{1201696}

\bibitem[RS17]{RagSac17}
\textsc{A.~Raghuram and G.~Sachdeva}, \emph{Critical values of {$L$}-functions
  for {$\rm GL_3\times GL_1$} over a totally real field}, L-functions and
  automorphic forms, Contrib. Math. Comput. Sci., vol.~10, Springer, Cham,
  2017, pp.~195--230. \MR{3931456}

\bibitem[Ram14]{Ram14}
\textsc{D.~Ramakrishnan},
  \articlehref{https://doi.org/10.1007/s13226-014-0088-1}{\emph{An exercise
  concerning the selfdual cusp forms on {${\rm GL}(3)$}}}, Indian J. Pure Appl.
  Math. \textbf{45} (2014), no.~5, 777--785. \MR{3286086}

\bibitem[Ros16]{Ros16}
\textsc{G.~Rosso}, \articlehref{https://doi.org/10.1353/ajm.2016.0023}{\emph{A
  formula for the derivative of the {$p$}-adic {$L$}-function of the symmetric
  square of a finite slope modular form}}, Amer. J. Math. \textbf{138} (2016),
  no.~3, 821--878. \MR{3506387}

\bibitem[Sch88]{Sch88}
\textsc{C.-G. Schmidt},
  \articlehref{https://0-doi-org.pugwash.lib.warwick.ac.uk/10.1007/BF01393749}{\emph{{$p$}-adic
  measures attached to automorphic representations of {${\rm GL}(3)$}}},
  Invent. Math. \textbf{92} (1988), no.~3, 597--631. \MR{939477}

\bibitem[Sch93]{Sch93}
\textsc{C.-G. Schmidt},
  \articlehref{https://0-doi-org.pugwash.lib.warwick.ac.uk/10.1007/BF01232425}{\emph{Relative
  modular symbols and {$p$}-adic {R}ankin-{S}elberg convolutions}}, Invent.
  Math. \textbf{112} (1993), no.~1, 31--76. \MR{1207477}

\bibitem[Sun17]{Sun17}
\textsc{B.~Sun},
  \articlehref{https://0-doi-org.pugwash.lib.warwick.ac.uk/10.1090/jams/855}{\emph{The
  nonvanishing hypothesis at infinity for {R}ankin-{S}elberg convolutions}}, J.
  Amer. Math. Soc. \textbf{30} (2017), no.~1, 1--25. \MR{3556287}

\bibitem[Tat79]{Tat79}
\textsc{J.~Tate}, \emph{Number theoretic background}, Automorphic forms,
  representations and {$L$}-functions Part 2, Proc. Sympos. Pure Math.,
  vol.~33, 1979.

\bibitem[Urb11]{Urb11}
\textsc{E.~Urban},
  \articlehref{https://0-doi-org.pugwash.lib.warwick.ac.uk/10.4007/annals.2011.174.3.7}{\emph{Eigenvarieties
  for reductive groups}}, Ann. of Math. (2) \textbf{174} (2011), no.~3,
  1685--1784. \MR{2846490}

\bibitem[Wei71]{Wei71}
\textsc{A.~Weil}, \emph{Dirichlet series and automorphic forms}, Lecture Notes
  in Math., vol. 189, Springer, 1971.

\end{thebibliography}
  \end{document}